\documentclass[letterpaper,reqno,11pt]{amsart}
\usepackage{amsmath,amsthm,amsfonts,amssymb,euscript,mathrsfs,graphics,color,latexsym,marginnote}
\usepackage{stmaryrd}
\usepackage[dvips]{graphicx}
\usepackage[left=0.75 in, right=0.75 in,top=0.75 in, bottom=0.75 in]{geometry}
\usepackage{hyperref}
\usepackage[toc,page]{appendix}
\usepackage{relsize}
\usepackage[shortlabels]{enumitem}

\usepackage[dvipsnames]{xcolor}

\usepackage{hyperref}
\hypersetup{colorlinks=true, pdfstartview=FitV,linkcolor=blue!70!black,citecolor=red!70!black, urlcolor=green!60!black}
\definecolor{labelkey}{rgb}{0.6,0,0} 

\makeatletter
\renewcommand \theequation {%
\ifnum \c@section>\z@ \@arabic\c@section.%
\fi\@arabic\c@equation} \@addtoreset{equation}{section}
\makeatother

\allowdisplaybreaks[4]
\newtheorem{theorem}{Theorem}[section]

\newtheorem{lemma}[theorem]{Lemma}
\newtheorem{proposition}[theorem]{Proposition}

\theoremstyle{definition}

\theoremstyle{remark}
\newtheorem{remark}{Remark}[section]

\def\charf {\mbox{{\text 1}\kern-.30em {\text l}}}

\def\XXint#1#2#3{{\setbox0=\hbox{$#1{#2#3}{\int}$ }
\vcenter{\hbox{$#2#3$ }}\kern-.6\wd0}}

\providecommand{\abs}[1]{\left\vert#1\right\vert}
\providecommand{\babs}[1]{\big\vert#1\big\vert}

\providecommand{\nm}[1]{\left\Vert#1\right\Vert}

\providecommand{\bnm}[1]{\big\Vert#1\big\Vert}

\providecommand{\br}[1]{\left\langle #1 \right\rangle}

\providecommand{\brv}[1]{\left( #1 \right)}
\providecommand{\Bbr}[1]{\Big\langle #1 \Big\rangle}

\providecommand{\abss}[1]{\left\vert#1\right\vert_{\sigma}}
\providecommand{\nms}[1]{\left\Vert#1\right\Vert_{\sigma}}
\providecommand{\bnms}[1]{\big\Vert#1\big\Vert_{\sigma}}
\providecommand{\tbs}[1]{\left\vert#1\right\vert_{L^2}}
\providecommand{\tnm}[1]{\left\Vert#1\right\Vert}
\providecommand{\btnm}[1]{\big\Vert#1\big\Vert}

\providecommand{\nmsw}[1]{\left\Vert#1\right\Vert_{w^0,\sigma}}
\providecommand{\tnmw}[1]{\left\Vert#1\right\Vert_{w^0}}
\providecommand{\nmsww}[1]{\left\Vert#1\right\Vert_{w^1,\sigma}}
\providecommand{\tnmww}[1]{\left\Vert#1\right\Vert_{w^1}}
\providecommand{\nmswww}[1]{\left\Vert#1\right\Vert_{w^2,\sigma}}
\providecommand{\tnmwww}[1]{\left\Vert#1\right\Vert_{w^2}}
\providecommand{\jump}[1]{\left\llbracket #1 \right\rrbracket }

\providecommand{\nmw}[1]{\left\Vert#1\right\Vert_{w^0}}
\providecommand{\nmww}[1]{\left\Vert#1\right\Vert_{w^1}}
\providecommand{\nmwww}[1]{\left\Vert#1\right\Vert_{w^2}}

\def\ud{\mathrm{d}}
\def\dt{\partial_t}
\def\p{\partial}
\def\ls{\lesssim}
\def\gs{\gtrsim}

\def\rt{\rightarrow}
\def\r{\mathbb{R}}
\def\no{\nonumber}
\def\ue{\mathrm{e}}

\def\ds{\displaystyle}

\def\R{\mathbb{R}}

\def\e{\varepsilon}
\def\nx{\nabla_x}

\def\li{\mathcal{L}}

\def\c{\mathcal{C}}
\def\m{\mathbf{M}}
\def\mh{\m^{\frac{1}{2}}}
\def\mhh{\m^{-\frac{1}{2}}}
\def\pk{\mathbf{P}}
\def\ik{\mathbf{I}}

\def\fr{F_R^{\e}}
\def\f{F^{\e}}
\def\fe{f_R^{\e}}
\def\hp{\hat{p}}
\def\ss{S}
\def\sb{\overline S}
\def\ee{\mathcal{E}}
\def\dd{\mathcal{D}}
\def\yy{Y}
\def\zz{Z}
\def\zzz{\mathcal{Z}}
\def\ww{W}

\def\aa{\mathcal{A}}
\def\kk{\mathcal{K}}

\def\en{\mathfrak{e}}
\def\tc{T_c}
\def\nc{N_c}

\begin{document}

\title{Hilbert Expansion for Coulomb Collisional Kinetic Models}

\author[Z. Ouyang]{Zhimeng Ouyang}
\address[Z. Ouyang]{
   \newline\indent Department of Mathematics, University of Chicago}
\email{zhimeng\_ouyang@alumni.brown.edu}
\thanks{Z. Ouyang was supported by an NSF Grant DMS-2202824. }

\author[L. Wu]{Lei Wu}
\address[L. Wu]{
   \newline\indent Department of Mathematics, Lehigh University}
\email{lew218@lehigh.edu}
\thanks{L. Wu was supported by an NSF Grant DMS-2104775.}

\author[Q. Xiao]{Qinghua Xiao}
\address[Q. Xiao]{
   \newline\indent Innovation Academy for Precision Measurement Science and Technology, Chinese Academy of Sciences}
\email{xiaoqh@apm.ac.cn}
\thanks{Q. Xiao was supported by NSFC Grants 11871469 and 12271506.}

\subjclass[2010]{Primary: 82C40; Secondary: 76P05, 35Q20}


\begin{abstract}
The relativistic Vlasov-Maxwell-Landau (r-VML) system and the relativistic Landau equation (r-LAN) are fundamental models that describe the dynamics of an electron gas. In this paper, we introduce a novel weighted energy method and establish the validity of the Hilbert expansion for the r-VML system and r-LAN equation. As the Knudsen number shrinks to zero, we rigorously demonstrate the relativistic Euler-Maxwell limit and relativistic Euler limit, respectively. This successfully resolves the long-standing open problem regarding the hydrodynamic limits of Landau-type equations. 
\end{abstract}

\keywords{Hilbert expansion; relativistic Landau equation; relativistic Vlasov-Maxwell-Landau system; local Maxwellian}

\maketitle

\pagestyle{myheadings} \thispagestyle{plain} \markboth{Z. OUYANG, L. WU, Q. XIAO}{HILBERT EXPANSION FOR COULOMB COLLISIONAL KINETIC MODELS}



\section{Introduction}

\subsection{Relativistic Vlasov-Maxwell-Landau System}


The relativistic Vlasov-Maxwell-Landau (r-VML) system is a fundamental and complete model describing the dynamics of a dilute collisional ionized plasma appearing in  nuclear fusion and the interior of stars, etc. Correspondingly, the relativistic Euler-Maxwell system, the foundation of the two-fluid theory in plasma physics, describes the dynamics of two compressible ion and electron fluids interacting with their own self-consistent electromagnetic field. It is also the origin of many celebrated dispersive PDE such as NLS, KP, KdV, Zaharov, etc, as various scaling limits and approximations.

Since the ion mass is
far larger than the electron mass in a plasma, the dynamics of ions is negligible for simplification
sometimes. In this special case, the plasma can be approximately described by the one-species r-VML system in the mesoscopic
level and treated as a single fluid in the macroscopic level. It has been a long-term open question if the general relativistic Euler-Maxwell system can be derived rigorously from its kinetic counter-part, the r-VML system, as the Knudsen number approaches zero. 

In this paper, we are able to answer this question in the affirmative. Consider the  r-VML  system for $\f(t,x,p)\in\r$:
\begin{align}\label{main1}
\dt  \f + c\hat{p}\cdot \nabla_x \f - e\big(E^{\e}+\hat{p}\times B^{\e} \big)\cdot\nabla_p \f =\dfrac{1}{\e}\c\left[\f,\f\right],
\end{align}
coupled with the Maxwell system for $\big(E^{\e}(t,x), B^{\e}(t,x)\big)\in\r^3\times\r^3$:
\begin{align}\label{main2}
\left\{
\begin{array}{l}
\ds\dt E^{\e}-  c\nabla_x \times B^{\e} =4\pi e\int_{\mathbb R^3}\hat{p} \f\ud p, \\\rule{0ex}{1.0em}
\dt B^{\e}+ c\nabla_x \times E^{\e}=0,\\\rule{0ex}{1.5em}
\ds\nabla_x\cdot E^{\e}=-4\pi e\Big(\overline{n}- \int_{\mathbb R^3}  \f\ud p\Big), \\\rule{0ex}{1.0em}
\nabla_x\cdot B^{\e}=0.
\end{array}
\right.
\end{align} 
Here $\e$ is the  Knudsen
number, $\f(t,x,p)$ is the number density function for electrons at time $t\geq0$, position $x=(x_1,x_2,x_3) \in \mathbb R^3$ and momentum $p=(p_1,p_2,p_3)\in \mathbb R^3$. $p^0=\sqrt{m^2c^2 + |p|^2}$ is the energy of an electron and $\hat{p}=\frac{p}{p^0} $. The constants $-e $ and $m$ are the electrons' charge and rest mass, respectively. $c$ is the speed of light, $\overline{n}$ is the uniform number density of ions, and $\big(E^{\e}(t,x), B^{\e}(t,x)\big)$ are the electromagnetic fields.

Denote the four-momentums $p^{\mu} = \left(p^0, p\right)$ and $q^{\mu} = \left(q^0, q\right)$.  We use the Einstein convention that repeated up-down indices be
summed and we raise and lower indices using the Minkowski metric $g_{\mu\nu}
:=
\text{diag}(-1, 1, 1, 1)$. The Lorentz inner product is then given by
\begin{align}
p^{\mu}q_{\mu}:=-p^0q^0 +
\sum_{i=1}^3 p_iq_i .
\end{align}
The collision operator $\c$ on the R.H.S. of \eqref{main1}, which registers binary collisions between particles, takes the following form:
\begin{align} \label{coll}
\c[g ,h] := \nabla_{p}\cdot\left\{\int_{\R^3} \Phi(p,q) \Big[\nabla_pg(p) h(q) -g(p) \nabla_q h(q)\Big]\ud q \right\},
\end{align}
where the collision kernel $\Phi(p,q)$ is a $3\times3$ non-negative matrix 
\begin{align}\label{cker}
    \Phi(p,q):=\frac{\Lambda(p,q)}{p^0q^0}\mathcal{S}(p,q)
\end{align}
with
\begin{align*}
\Lambda(p,q):=&\;\frac{1}{m^4c^4}\left(p^{\mu}q_{\mu}\right)^2\left(\frac{1}{m^4c^4}\left(p^{\mu}q_{\mu}\right)^2-1\right)^{-\frac{3}{2}},\\
\mathcal{S}(p,q):=&\;\left(\frac{1}{m^4c^4}\left(p^{\mu}q_{\mu}\right)^2-1\right)I_3-\frac{1}{m^2c^2}\big(p-q\big)\otimes\big(p-q\big)\\
&\;-\frac{1}{m^2c^2}\left(\frac{1}{m^2c^2}\left(p^{\mu}q_{\mu}\right)-1\right)\big(p\otimes q+q\otimes p\big).\no
\end{align*}

It is well-known that $\Phi(p,q)$ satisfies
\begin{align}\label{Phi}
    \sum_{i=1}^3\Phi^{ij}(p,q)\left(\frac{q_i}{q^0}-\frac{p_i}{p^0}\right)=\sum_{j=1}^3\Phi^{ij}(p,q)\left(\frac{q_j}{q^0}-\frac{p_j}{p^0}\right)=0.
\end{align}

The collision operator $\c$ satisfies the orthogonality property:
\begin{align}
    \int_{\r^3}\left\{\begin{pmatrix}1\\p\\p^0\end{pmatrix}\c[g,h](p)\right\}\ud p=\mathbf{0},
\end{align}
which, combined with \eqref{main1} and \eqref{main2}, yields the conservation laws
\begin{align*}
    &\frac{\ud}{\ud t}\iint_{\r^3\times\r^3}\f(t,x,p)\,\ud p\ud x=0,\\ 
    &\frac{\ud}{\ud t}\bigg\{\iint_{\r^3\times\r^3}p\f(t,x,p)\,\ud p\ud x +\frac{1}{4\pi}\int_{\r^3}\Big(E^{\e}(t,x)\times B^{\e}(t,x)\Big)\,\ud x\bigg\}=\mathbf{0},\\
    &\frac{\ud}{\ud t}\bigg\{\iint_{\r^3\times\r^3}p^0\f(t,x,p)\,\ud p\ud x+\frac{1}{4\pi} \int_{\r^3} \Big(|E^{\e}(t,x)|^2+|B^{\e}(t,x)|^2\Big)\,\ud x\bigg\}=0.\no
\end{align*}


Corresponding to \eqref{main1}--\eqref{main2}, at the
hydrodynamic level, the electron gas obeys the relativistic Euler-Maxwell system for 
$\big(n(t,x),u(t,x), T(t,x)\big)\in \r\times\r^3\times\r$: 
\begin{align}\label{rem}
\left\{
\begin{array}{l}
\dfrac{1}{c}\dt \big(n u^0\big) + \nabla_x\cdot\big(n u \big) =0,\\\rule{0ex}{2.0em}
\dfrac{1}{c}\dt \Big((\en+P)u^0u \Big) +\nabla_x\cdot\Big((\en+P)u \otimes u \Big)+c^2\nabla_xP+ c e n\big(u^0E+u\times B)=\mathbf{0},\\\rule{0ex}{2.0em}
\displaystyle \dfrac{1}{c}\dt \Big(\en(u^0)^2+P|u |^2\Big) +\nabla_x\cdot\Big((\en+P)u^0u \Big)+c e n(u \cdot E)=0, 
\end{array}
\right.
\end{align}
coupled with the Maxwell system for $\big(E(t,x), B(t,x)\big)\in\r^3\times\r^3$:
\begin{align}\label{rem'}
\left\{
\begin{array}{l}
\dt E-c\nabla_x\times B =4\pi e\dfrac{nu}{c} ,\\\rule{0ex}{1.5em}
\dt B+c\nabla_x\times E =0,\\\rule{0ex}{2em}
\nabla_x\cdot E=-4\pi e\Big(\overline{n}-\dfrac{nu^0}{c}\Big),\\\rule{0ex}{1.5em}
\nabla_x\cdot B=0,
\end{array}
\right.
\end{align}
where $n$ is the electrons' number density, $u =(u_{1}, u_{2}, u_{3})$, $u^0=\sqrt{|u |^2+c^2} $, and $T$ is the temperature. In particular, $\en(t,x)$ is the total energy (including the rest energy and internal energy) and $P(t,x)$ is the pressure given by
\begin{align}\label{aa 01}
    P &:=  \frac{nm  c^2}{\gamma } = \frac{k_B}{m }\rho T  \, , \\
    \en&:= \frac{nm  c^2}{  K_2(\gamma )}  \left\{ K_3 \left(\gamma 
    \right) - \frac{1}{\gamma }  K_2
\left(\gamma    \right) \right\},\label{aa 02}
\end{align}
where $\rho:=nm $ is the mass density, $\gamma :=m c^2(k_BT)^{-1}$ is a dimensionless variable,
$k_B$ is Boltzmann's constant,
$K_j(\gamma)$ for $j=0, 1, 2, \ldots$ are the modified second-order Bessel functions:
\begin{equation}\label{defini}
	 K_j(\gamma):=\frac{(2^j)j!}{(2j)!}\frac{1}{\gamma^j}\int_{\gamma}^{\infty}e^{-\lambda}(\lambda^2-\gamma^2)^{j-1/2}d\lambda,
    \quad(j\geq0).
\end{equation}

The system \eqref{rem'} has been well-studied in the irrotational context. Denote Faraday's tensor
\begin{align}
    \mathscr{F}^{ij}:=\begin{pmatrix}
    0&-c^{-1}E_1&-c^{-1}E_2&-c^{-1}E_3\\
    c^{-1}E_1&0&-B_3&B_2\\
    c^{-1}E_2&B_3&0&-B_1\\
    c^{-1}E_3&-B_2&B_1&0
    \end{pmatrix}.
\end{align}
Let $h$ be the specific enthalpy defined by $h'(x)=\frac{P(x)}{x}$ with $h>0$. Then we say the solution to \eqref{rem} and \eqref{rem'} is irrotational if 
\begin{align}\label{irrotational}
    e\mathscr{F}_{jk}=&-\p_j\left(hnu_{k}\right)+\p_k\left(hnu_{j}\right).
\end{align}

\begin{theorem}[Theorem 2.2 of \cite{Guo.Ionescu.Pausader2014}]\label{Euler-Maxwell}
Assume that the initial datum $\Big(n(0,x), u (0,x), T(0,x), E(0,x), B(0,x)\Big)$ satisfies \eqref{irrotational} and is sufficiently close to the equilibrium $\big(\overline n, \mathbf{0}, \overline T, \mathbf{0}, \mathbf{0}\big)$ for some constants $\overline n>0$ and $\overline T$. Then there exists a unique global solution $\Big(n(t,x), u (t,x), T(t,x), E(t,x), B(t,x)\Big)$ to the one-fluid relativistic Euler-Maxwell system \eqref{rem} and \eqref{rem'} that satisfies \eqref{irrotational} for any $t>0$ and
\begin{align}\label{decay} &\sup_{t\in[0,\infty)}\nm{\left(n(t)-\overline{n}, u (t), T(t)-\overline T, E(t), B(t)\right)}_{H^{\nc}}\\
&+\sup_{t\in[0,\infty)}\sup_{|\rho|\leq \overline N}\left((1+t)^{\beta_0}\nm{\nabla_x^{\rho}\left(n(t)-\overline{n},u (t),T(t)-\overline T,  E(t),B(t)\right)}_{L^{\infty}}\right)\lesssim \overline{\varepsilon}_0,  \no
\end{align}
where $\nc\in\mathbb{N}$ is a sufficiently large constant, $\overline{N}\geq3$ is a constant,  $\beta_0=\frac{101}{100}$ and $\overline{\varepsilon}_0$ is a sufficiently small positive constant.
\end{theorem}


In this article, we rigorously prove that solutions of the r-VML system \eqref{main1}–\eqref{main2} converge to solutions of the relativistic
Euler–Maxwell system \eqref{rem}-\eqref{rem'} globally in time, as the Knudsen number $\e$ tends to zero.

\begin{theorem}\label{main} 
Assume that $\Big(n (t,x),u (t,x), T (t,x), E (t,x)$, $B (t,x)\Big)$ is the global solution constructed in Theorem \ref{Euler-Maxwell} and $\m (t,x,p)= \frac{n }{4 \pi  m^2
c k_BT K_2(\gamma )} \exp\left\{\frac{{u }^{\mu}p_{\mu} }{k_BT }\right\}$ is the corresponding local Maxwellian. 
Then there exists an $\e_0 > 0$ such
that for any $0 \leq \e\leq\e_0$, $k\geq3$,  and $0<t\leq\overline{t}$ with $\overline{t}=\e^{-1/3}$, the Hilbert expansion \eqref{expan2} holds. Moreover, if $\f(0,x,p)\geq0$, and
\begin{align}\label{mainlim0}
    \nm{\m^{-\frac12}\left(\f-\m\right)(0)}_{H^2_xL^2_{v}}+\nm{\big(E^\e-E \big)(0)}_{H^2}+\nm{\big(B^{\e}-B \big)(0)}_{H^2}=O(\e),
\end{align}
then $\f(t,x,p)\geq0$ and \begin{align}\label{mainlim}
    \lim_{\e\rt0}\sup_{0\leq t\leq \overline{t}}\left\{\nm{\m^{-\frac12}\left(\f-\m\right)(t)}_{H^2_xL^2_{v}}+\nm{\big(E^{\e}-E \big)(t)}_{H^2}+\nm{\big(B^{\e}-B \big)(t)}_{H^2}\right\}=0.
\end{align}
\end{theorem}

\subsection{Relativistic Landau Equation}


When the effects of electromagnetic fields are negligible, the relativistic Landau (r-LAN) equation provides a much easier yet still accurate description of the dynamics of a fast moving dilute plasma when the grazing collisions between particles are predominant in the collisions. 

Let $F^{\e}= F^{\e}(t,x,p)$  be the number density function for particles at the phase-space position $(x,p)=(x_1,x_2,x_3,p_1,p_2,p_3) \in \mathbb R^3 \times \mathbb R^3$, at time $t \in \mathbb R_+$. Then $F^{\e}$ satisfies the r-LAN equation
\begin{align} \label{main1=}
& \dt F^{\e} + c\hp \cdot \nx F^{\e} = \frac{1}{\e}\,\c\left[F^{\e},F^{\e}\right],
\end{align}
where $\hp=\frac{p}{p^0}$, $p^0=\sqrt{m^2c^2+|p|^2}$ is the energy of the particle, constants  $c, m$ are the speed of light and the rest mass of a particle, respectively. $0<\e\ll1$ is the Knudsen number.

Similar to \eqref{coll}, 
the collision operator $\c$
yields the conservation laws
\begin{align*}
    &\frac{\ud}{\ud t}\iint_{\r^3\times\r^3}\f(t,x,p)\,\ud p\ud x=
    \frac{\ud}{\ud t}\iint_{\r^3\times\r^3}p^0\f(t,x,p)\,\ud p\ud x=0,\\
    &\frac{\ud}{\ud t}\iint_{\r^3\times\r^3}p\f(t,x,p)\,\ud p\ud x=\mathbf{0}.
\end{align*}


Corresponding to \eqref{main1=}, at the hydrodynamic level, the plasma obeys the relativistic Euler equations for $(n(t,x),u(t,x),T(t,x))\in\r\times\r^3\times\r$:
\begin{align}\label{re=}
\left\{
\begin{array}{l}
\dfrac{1}{c}\dt\big(n u^0\big) + \nabla_x\cdot\big(n u\big) =0,\\\rule{0ex}{2.0em}
 \dfrac{1}{c}\dt\Big\{\big(\mathfrak{e}+P\big)u^0u\Big\} +\nx\cdot\Big\{\big(\mathfrak{e}+P\big)\big(u\otimes u\big)\Big\}+c^2\nx P=\mathbf{0},\\\rule{0ex}{2.0em}
\dfrac{1}{c}\dt\Big\{\big(\mathfrak{e}+P\big)\left(u^0\right)^2-c^2\abs{u}^2P\Big\} +\nx\cdot\Big\{\big(\mathfrak{e}+P\big)u^0u\Big\}=0,
\end{array}
\right.
\end{align}
where $n$ is the particle number density, $u=(u_1, u_2, u_3)$, $u^0=\sqrt{|u|^2+c^2} $. Here, $P$ and
$\mathfrak{e}$ are defined as in \eqref{aa 01} and \eqref{aa 02} (see \cite{Speck.Strain2011}).


\begin{theorem}[Theorem 1 of \cite{Speck.Strain2011}] \label{Eul-exi}
Under proper regularity conditions, if the initial data $\Big(n(0),u(0),T(0)\Big)$ is sufficiently close to an equilibrium state $(\overline n,0,\overline T)$: for $N\geq3$ 
\begin{align}
    \nm{\big(n-\overline n,u,T-\overline T\big)(0)}_{H^N}\leq\delta\ll1,
\end{align}
then there exists a unique solution $(n,u,T)$ to \eqref{re=} for $t\in[0,\tilde{t}]$  with $\tilde{t}\geq \overline{c}\delta^{-1}$ for some constant $\overline{c}>0$ satisfying
\begin{align}
    \sup_{0\leq t\leq \tilde{t},x\in\mathbb R^3,0\leq \ell\leq N-2}\left\{\babs{\nabla_{t,x}^{\ell}(n,u,T)}\right\}\ll1.
\end{align}
\end{theorem}


In this article, we rigorously prove that solutions of the relativistic Landau equation
\eqref{main1=} converge to solutions of the relativistic Euler equations \eqref{re=} locally in time, as the Knudsen
number $\e$ tends to zero.

\begin{theorem}\label{main=} 
Assume that $\big(n (t,x), u (t,x), T (t,x)\big)$ is the  solution constructed in Theorem \ref{Eul-exi} and $\m(t,x,p)= \frac{n }{4 \pi m^2
c k_BT K_2(\gamma)} \exp\left\{\frac{u^{\mu}p_{\mu} }{k_BT }\right\}$ is the corresponding local Maxwellian. 
Then there exists a $\e_0 > 0$ such
that for any $0 \leq \e\leq\e_0$, $k\geq3$,  and $0<t\leq t_0$ with some $t_0$ depending on $\big(n (t,x),  u  (t,x), T (t,x)\big)$ but independent of $\e$, the Hilbert expansion \eqref{expan2=} holds. Moreover, if $\f(0,x,p)\geq0$, and
\begin{align}\label{mainlim0=}
    \nm{\m^{-\frac12}\left(\f-\m\right)(0)}_{H^2_xL^2_{v}}=O(\e),
\end{align}
then $\f(t,x,p)\geq0$ and 
\begin{align}\label{mainlim=}
    \lim_{\e\rt0} \sup_{0\leq t\leq t_0}\left\{\nm{\m^{-\frac12}\left(\f-\m\right)(t)}_{H^2_xL^2_{v}}\right\}=0.
\end{align}
\end{theorem}

\subsection{Background and Literature} 

As a key ingredient to attack the well-known Hilbert's sixth problem, the rigorous derivation of fluid equations (Euler equations or Navier-Stokes equations, etc.) from the kinetic equations (Boltzmann equation, Landau equation, etc.) has attracted a lot of attentions since the early twentieth century. The fundamental problem is to justify the asymptotic limits of kinetic solutions as the Knudsen number (which measures the relative mean free path) or the Strouhal number (which measures the relative time-varying speed) shrinks to zero. 

There are mainly two genres to study hydrodynamic limits: kinetic-based approach or fluid-based approach. We refer to \cite{Jiang.Luo2022, Jiang.Xu.Zhao2018} for more details.

Kinetic-based approach purely relies on the solution theory (well-posedness, regularity, etc.) of the kinetic equations and does NOT assume any a priori properties of the fluid limits. 
On one hand, in the context of the renormalized solution and entropy method, 
there are successful applications of this approach (usually referred as BGL Project) to the incompressible Euler/Navier-Stokes limit. We refer to Bardos-Golse \cite{Bardos.Golse1984},  Golse-Saint-Raymond \cite{Golse.Saint-Raymond2001, Golse.Saint-Raymond2004}, Saint-Raymond \cite{Saint-Raymond2003}, Masmoudi-Saint-Raymond \cite{Masmoudi.Saint-Raymond2003}, Arsenio-Saint-Raymond \cite{Arsenio.Saint-Raymond2019}, Bardos-Golse-Levermore \cite{Bardos.Golse.Levermore1991, Bardos.Golse.Levermore1993, Bardos.Golse.Levermore1998},  Lions-Masmoudi \cite{Lions.Masmoudi2001} and Masmoudi \cite{Masmoudi2002}. Interested readers may refer to the books by Saint-Raymond \cite{Saint-Raymond2009} and by Golse \cite{Golse2014}, and the references therein provide a nice summary of the progress. On the other hand, in the context of classical solutions, we refer to Nishida \cite{Nishida1978}, Bardos-Ukai \cite{Bardos.Ukai1991}, Briant \cite{Briant2015(=)}, Briant-Merino-Aceituno-Mouhot \cite{Briant.Merino-Aceituno.Mouhot2019}.

The fluid-based approach DOES assume a priori that we have a well-prepared fluid system, which has a unique smooth solution, and then justify the kinetic solution converging to this fluid solution. In some sense, this is essentially ``fluid-to-kinetic'' limit and we avoid the complications of possible fluid ill-posedness, like blow up or shock wave. This approach typically provides hydrodynamic limits in the stronger sense and utilizes the so-called Hilbert expansion techniques.
In this paper, we will focus on the fluid-based approach and discuss the progress in detail.

The Hilbert expansion dates back to 1912 by Hilbert \cite{Hilbert1916}, who proposed an asymptotic expansion of the distribution function solving the Boltzmann equation with respect to the Knudsen number and formally derived the limiting compressible Euler equations. The similar formal expansion can be naturally extended to treat the Landau equation, and Vlasov systems.

The first rigorous justification of the compressible Euler limit of the Boltzmann equation was due to Caflisch \cite{Caflisch1980}. Later, with the $L^2-L^{\infty}$ framework introduced in Guo \cite{Guo2010}, Guo-Jang-Jiang \cite{Guo.Jang.Jiang2009} improved Caflisch's result and removed the assumption on the initial data $F^{\e}_R(0,x,v)=0$. This framework was extended to treat the Vlasov-Poisson-Boltzmann (VPB) system in Guo-Jang \cite{Guo.Jang2010} and the relativistic Boltzmann (r-BOL) equation in Speck-Strain \cite{Speck.Strain2011}. Recently, this framework was further developed to the investigation of the relativistic Vlasov-Maxwellian-Boltzmann (r-VMB) system in Guo-Xiao \cite{Guo.Xiao2021} and the Boltzmann equation with boundary conditions in half-space in Guo-Huang-Wang \cite{Guo.Huang.Wang2021}, Jiang-Luo-Tang \cite{Jiang.Luo.Tang2021,Jiang.Luo.Tang2021(=)}.
We also refer to Grad \cite{Grad1965}, Ukai-Asano \cite{Ukai.Asano1983} De Masi-Esposito-Lebowitz \cite{Masi.Esposito.Lebowitz1989}, and the recent work Jang-Kim \cite{Jang.Kim2021} and Kim-La \cite{Kim.La2022} for the incompressible Euler limit. 

For the convergence of the Boltzmann equation to the basic waves of the Euler equations: the shock waves, rarefaction waves and contact discontinuity, the interested readers may refer to Huang-Wang-Yang \cite{Huang.Wang.Yang2010, Huang.Wang.Yang2010(=), Huang.Wang.Yang2013}, Xin-Zeng \cite{Xin.Zeng2010} and Yu \cite{Yu2005}.

As for the incompressible Navier-Stokes limit of the Boltzmann equation, there are too many references and we only list some closely related works. The early development tracks back to De Masi-Esposito-Lebowitz \cite{Masi.Esposito.Lebowitz1989} in 2D. Then Guo \cite{Guo2006} justified the diffusive limit in the periodic domain via the nonlinear energy method. This result was extended to the whole space in Liu-Zhao \cite{Liu.Zhao2011}, to more general initial data with initial layer in Jiang-Xiong \cite{Jiang.Xiong2015}, and to the Vlasov-Maxwell-Boltzmann (VMB) system in Jang \cite{Jang2009}. See also the recent work Gallagher-Tristani \cite{Gallagher.Tristani2020}. We also mention the very recent work  Duan-Yang-Yu \cite{Duan.Yang.Yu2022==} for the compressible Euler-Maxwell limit of the one-species VMB system.  

For stationary Boltzmann equation and other settings, we refer to Di-Meo-Esposito \cite{Di.Esposito1996}, Arkeryd-Esposito-Marra-Nouri \cite{Arkeryd.Esposito.Marra.Nouri2010}, Esposito-Lebowitz-Marra \cite{Esposito.Lebowitz.Marra1994}, Esposito-Guo-Marra \cite{Esposito.Guo.Marra2018}, Esposito-Guo-Kim-Marra \cite{Esposito.Guo.Kim.Marra2015}, Wu \cite{AA013}, Wu-Ouyang \cite{BB002, AA017, AA018}.

Despite the fruitful progress in the hydrodynamic limits of the Boltzmann-type equations, there are very limited works in this direction for Landau-type equations. For Landau equation, we refer to Guo \cite{Guo2006} for the incompressible Navier-Stokes limit, Duan-Yang-Yu \cite{Duan.Yang.Yu2021, Duan.Yang.Yu2022=} for the rarefaction wave limit and compressible Euler limit, and the recent work Rachid \cite{Rachid2021}. As far as we are aware of, our paper is the first result to justify the Hilbert expansion for r-LAN equation and r-VML system. 

As for the well-posedness issue for fixed Knudsen number and Strouhal number, there are a huge number of literature. We list some closely related to this article. 
For the r-VML system, we refer to Strain-Guo \cite{Strain.Guo2004}, Yu \cite{Yu2009}, Yang-Yu \cite{Yang.Yu2012}, Liu-Zhao \cite{Liu.Zhao2014} and Xiao \cite{Xiao2015}. 
For the r-LAN equation, we refer to Hsiao-Yu \cite{Hsiao.Yu2006} and Yang-Yu \cite{Yang.Yu2010}. We also mention Guo-Strain \cite{Guo.Strain.2012} and some works in the non-relativistic framework: Villani \cite{Villani1996}, Guo \cite{Guo2002(=), Guo2003, Guo2012}, Strain \cite{Strain2006}, Duan-Strain \cite{Duan.Strain2011(=)}, Duan \cite{Duan2014}, Duan-Lei-Yang-Zhao \cite{Duan.Lei.Yang.Zhao2017}, Guo-Hwang-Jang-Ouyang \cite{Guo.Hwang.Jang.Ouyang2020}, Duan-Liu-Sakamoto-Strain \cite{Duan.Liu.Sakamoto.Strain2020} for Landau equation and Dong-Guo-Ouyang \cite{Dong.Guo.Ouyang2020} for Vlasov-Poisson-Landau (VPL) system.

Finally, we record some significant progress on the compressible fluid system. Sideris \cite{Sideris1985} justified the classical result on the compressible Euler equation that the solution might blow up even if the initial datum is small and irrotational. However, as a key observation, the electric field or the electromagnetic fields might help stabilize the system. Based on the Klein–Gordon effect, Guo \cite{Guo1998} and Germain-Masmoudi \cite{Germain.Masmoudi2014} constructed global classical solutions to the one-fluid Euler-Poisson system and Euler-Maxwell system, respectively.  Using the combination of normal-form method and vector-field method to capture the so-called ``null structure'', Guo-Ionescu-Pausader \cite{Guo.Ionescu.Pausader2016} justified the global well-posedness of 3D two-fluid Euler-Maxwell system, and the similar results were extended to treat 3D Euler-Poisson system, and 3D one-fluid/two-fluid relativistic Euler-Maxwell system in Guo-Ionescu-Pausader \cite{Guo.Ionescu.Pausader2014}. The 2D case was justified in Deng-Ionescu-Pausader \cite{Deng.Ionescu.Pausader2017}. More recently, the one-fluid Euler–Maxwell system in 3D with non-vanishing vorticity was studied in Ionescu-Lie \cite{Ionescu.Lie2018}.


\section{Formulation and Discussion} 

Without loss of generality, from now on, we will take the constants $c=e=m=k_B=1$.

\subsection{Hilbert Expansion for the Relativistic Vlasov-Maxwell-Landau System}

In this subsection, we will provide the Hilbert expansion of the r-VML system \eqref{main1} and \eqref{main2}, and introduce necessary notations.


We consider the Hilbert expansion with respect to small Knudsen number $\e$ and $k\geq 2$:
\begin{align}\label{expan}
\f=F+\sum_{n=1}^{2k-1}\e^nF_{n}+\e^kF^{\e}_{R},\quad E^{\e}=E+\sum_{n=1}^{2k-1}\e^n E_n+\e^kE^{\e}_R,\quad B^{\e}=B+\sum_{n=1}^{2k-1}\e^n B_n+\e^kB^{\e}_R.
\end{align}
To determine the coefficients $F_{n}(t, x, p)$, $E_n(t, x)$, $B_n(t, x)$ for $0\leq n\leq 2k-1$, we plug the formal expansions \eqref{expan} into equations \eqref{main1}-\eqref{main2} 
and equate the coefficients on both sides
in front of different powers of
the parameter $\e$ to obtain:\\
\ \\
\textbf{$\e^{-1}$-order:}
\begin{align}
    \c\big[F ,F \big]=0.
\end{align}
\ \\
\textbf{$\e^{0}$-order:}
\begin{align}
    &\dt F +\hat{p}\cdot\nabla_x F - \big(E +\hat{p} \times B  \big)\cdot\nabla_pF 
    =\;\mathcal {C}\big[F_{1},F \big]+\mathcal {C}\big[F ,F_{1}\big],\label{expan2}
\end{align} 
and
\begin{align}\label{expan2'}
    \left\{
    \begin{array}{l}
    \ds\dt E -  \nabla_x \times B  =4\pi  \int_{\mathbb R^3}\hat{p}F \ud p, \\\rule{0ex}{1.0em}
    \ds\dt B + \nabla_x \times E =0,\\\rule{0ex}{1.5em}
    \ds\nabla_x\cdot E =4\pi \Big( \overline{n}-\int_{\mathbb R^3} F  \ud p\Big), \\\rule{0ex}{1.0em}
    \ds\nabla_x\cdot B =0.
    \end{array}
    \right.
\end{align}
\ \\
\textbf{$\e^{n}$-order ($1\leq n<2k-1$):}
\begin{align}
    &\dt F_{n}+\hat{p}\cdot \nabla_xF_{n}- \big(E_n+\hat{p} \times B_n \big)\cdot\nabla_pF - \big(E +\hat{p} \times B  \big)\cdot\nabla_pF_{n}\\
    =&\sum_{\substack{i+j=n+1\\i,j\geq0}}\mathcal {C}\big[F_{i},F_{j}\big]+ \sum_{\substack{i+j=n\\i,j\geq1}}\big(E_i+\hat{p} \times B_i \big)\cdot\nabla_pF_{j}, \no
\end{align} 
and
\begin{align}
    \left\{
    \begin{array}{l}
    \ds\dt E_n-\nabla_x \times B_n=4\pi \int_{\mathbb R^3}\hat{p}F_{n}\ud p, \\\rule{0ex}{1em}
    \ds\dt  B_n+ \nabla_x \times E_n=0,\\\rule{0ex}{1.5em}
    \ds\nabla_x\cdot E_n= -4\pi\int_{\mathbb R^3} F_{n}\ud p,\\\rule{0ex}{1em}
    \ds\nabla_x\cdot B_n=0.
    \end{array}
    \right.
\end{align}
\ \\
\textbf{$\e^{2k-1}$-order:}
\begin{align}
\\
    &\dt F_{2k-1}+\hat{p}\cdot\nabla_x F_{2k-1}-  \big(E_{2k-1}+\hat{p} \times B_{2k-1} \big)\cdot\nabla_pF 
     -\big(E +\hat{p} \times B  \big)\cdot\nabla_pF_{2k-1}\no\\
    =&\sum_{\substack{i+j=2k\\i,j\geq1}}\mathcal {C}\big[F_{i},F_{j}\big]+ \sum_{\substack{i+j=2k-1\\i,j\geq1}}\big(E_i+\hat{p} \times B_i \big)\cdot\nabla_pF_{j}, \no
\end{align} 
and
\begin{align} 
    \left\{
    \begin{array}{l}
    \ds\dt E_{2k-1}-\nabla_x \times B_{2k-1}=4\pi \int_{\mathbb R^3}\hat{p} F_{2k-1}\ud p, \\\rule{0ex}{1em}
    \ds\dt  B_{2k-1}+ \nabla_x \times E_{2k-1}=0,\\\rule{0ex}{1.5em}
    \ds\nabla_x\cdot E_{2k-1}=-4\pi \int_{\mathbb R^3} F_{2k-1}\ud p, \\\rule{0ex}{1em}
    \ds\nabla_x\cdot B_{2k-1}=0.
    \end{array}
    \right.
\end{align}
\ \\
\textbf{Remainder equation:} The remainder $\big(F_R^{\e}, E_R^{\e}, B_R^{\e}\big)$ satisfies
\begin{align}
    & \dt F_{R}^{\e}+\hat{p}\cdot\nabla_x F_{R}^{\e}- \big(E_R^{\e}+\hat{p} \times B_R^{\e} \big)\cdot\nabla_pF  -\big(E +\hat{p} \times B  \big)\cdot\nabla_pF_{R}^{\e}\label{Remain}\\
    &-\frac{1}{\e}\left\{\mathcal{C}\Big[F_{R}^{\e},F \Big]
    +\mathcal{C}\Big[F ,F_{R}^{\e}\Big]\right\}\no\\
    =\;&\e^{k-1}\mathcal{C}\Big[F_{R}^{\e},F^{\e}_{R}\Big]+\sum_{i=1}^{2k-1}\e^{i-1}\left\{\mathcal{C}\Big[F_{R}^{\e},F_{i}\Big]+\mathcal{C}\Big[F_{i},F_{R}^{\e}\Big]\right\}+\e^k \Big(E_R^{\e}+\hat{p} \times B_R^{\e} \Big)\cdot\nabla_pF_{R}^{\e}\no\\
    &+\sum_{i=1}^{2k-1}\e^i \left\{\Big(E_i+\hat{p} \times B_i \Big)\cdot\nabla_pF_{R}^{\e}+\Big(E_R^{\e}+\hat{p} \times B_R^{\e} \Big)\cdot\nabla_pF_{i}\right\}+\e^{k}\ss,
\no
\end{align}
and
\begin{align} \label{remain}
    \left\{
    \begin{array}{l}
    \ds\dt E_R^{\e}-\nabla_x \times B_R^{\e}=4\pi \int_{\mathbb R^3} \hat{p} F^{\e}_{R}\ud p, \\\rule{0ex}{1em}
    \ds\dt  B_R^{\e}+ \nabla_x \times E_R^{\e}=0,\\\rule{0ex}{1.5em}
    \ds\nabla_x\cdot E_R^{\e}=-4\pi \int_{\mathbb R^3}F^{\e}_{R}\ud p,\\\rule{0ex}{1em}
    \ds\nabla_x\cdot  B_R^{\e}=0,
    \end{array}
    \right.
\end{align}
where
\begin{align}
\ss=&\sum_{\substack{i+j\geq 2k+1\\2\leq i,j\leq2k-1}}\e^{i+j-2k-1}\left\{\mathcal{C}\Big[F_{i},F_{j}\Big]+\mathcal{C}\Big[F_{i},F_{j}\Big]\right\}+\sum_{\substack{i+j\geq 2k\\1\leq i,j\leq2k-1}}\e^{i+j-2k}\Big(E_i+\hat{p} \times B_i \Big)\cdot\nabla_pF_{j}.
\end{align}

From \eqref{expan2}, we conclude that $F $ should be local Maxwellians:
\begin{equation}\label{maxwell}
F (t,x,p)=\m = \frac{n }{4 \pi T K_2(\gamma )} \exp\left\{\frac{u^{\mu}p_{\mu} }{T }\right\},
\end{equation} 
where $(n , u , T )$ is part of the solution to the relativistic Euler-Maxwell system \eqref{rem}.
The other coefficients $F_{n}(t, x, p)$, $E_n(t, x)$, $B_n(t, x)$ for $0\leq n\leq 2k-1$ can be derived in an inductive way (see Appendix). 

To prove Theorem \ref{main}, our main task is to solve \eqref{Remain}--\eqref{remain}.
Define $\fe $ as
\begin{equation}\label{L2}
\f_{R}:= \mh \fe .
\end{equation}
\eqref{Remain} and \eqref{remain} can be rewritten as
\begin{align}
\label{L20}
    &\Big\{\dt +\hat{p}\cdot\nabla_x - \big(E +\hat{p} \times B  \big)\cdot\nabla_p \Big\}\fe 
    +  \frac{u^0}{T }\hat{p}\mh \cdot E_R^{\e}\\
    &- \frac{u   }{ T }\mh\cdot\Big(E_R^{\e}+\hat{p} \times B_R^{\e} \Big)+\frac{1}{\e}\li \left[\fe \right]\no\\
    =&-\fe \m ^{-\frac{1}{2}}\Big\{\dt +\hat{p}\cdot\nabla_x - \big(E +\hat{p} \times B  \big)\cdot\nabla_p \Big\}\mh +\e^{k-1}\Gamma \left[\fe,\fe\right]\no\\
    &+\sum_{i=1}^{2k-1}\e^{i-1}\left\{\Gamma \left[\m ^{-\frac{1}{2}}F_i, \fe\right]+\Gamma \left[\fe, \m ^{-\frac{1}{2}}F_i\right]\right\}+\e^k\Big(E_R^{\e}+\hat{p} \times B_R^{\e} \Big)\cdot\nabla_p\fe \no\\
    &-  \e^k\frac{1 }{2 T }\Big(u^0\hat{p}-u \Big)\cdot\Big(E_R^{\e}+\hat{p} \times B_R^{\e} \Big)\fe \no\\
    &+ \sum_{i=1}^{2k-1}\e^i\left\{\Big(E_i+\hat{p} \times B_i \Big)\cdot\nabla_p\fe +\Big(E_R^{\e}+\hat{p} \times B_R^{\e} \Big)\cdot \m ^{-\frac{1}{2}}\nabla_pF_{i}\right\}\no\\
    &- \sum_{i=1}^{2k-1}\e^i\left\{\Big(E_i+\hat{p} \times B_i \Big)\cdot\frac{1}{2 T }\Big(u^0\hat{p}-u \Big)\fe \right\}+\e^{k}\sb,\no
\end{align}
and
\begin{align}\label{L201}
\left\{
\begin{array}{l}
\ds\dt E_R^{\e}-\nabla_x \times B_R^{\e}=4\pi \int_{\mathbb R^3}\hat{p} \mh \fe\ud p, \\\rule{0ex}{1.5em}
\dt  B_R^{\e}+ \nabla_x \times E_R^{\e}=0,\\\rule{0ex}{2em}
\ds\nabla_x\cdot E_R^{\e}=-4\pi\int_{\mathbb R^3}\mh  \fe\ud p, \\\rule{0ex}{1.5em} 
\nabla_x\cdot  B_R^{\e}=0.
\end{array}
\right.
\end{align}
Here $\sb=\m ^{-\frac{1}{2}}\ss$. 
The linearized collision operator $\li[f]$ and nonlinear collision operator ${ \Gamma}[f, g]$ are defined as follows:
\begin{align}\label{LAK}
\li[f]  :=&\m ^{-\frac{1}{2}}\left\{\mathcal{C}\Big[\mh f,\m \Big]
 +\mathcal{C}\Big[\m ,\mh f\Big]\right\}=:-\aa[f]-\kk[f],
\end{align}
and
\begin{align}\label{Gamma=}
\Gamma[f, g] &:=\m ^{-\frac{1}{2}}\mathcal{C}\Big[\mh f,\mh g\Big].
\end{align}
Note that the null space of the linearized operator $\li$ is given by 
\begin{align}
    \mathcal {N}=\mbox{span}\left\{\mh ,   p_i\mh  (1\leq i\leq3), p^{0}\mh \right\}.
\end{align}
Denote $\pk$ as the orthogonal projection from $L^2_p$ onto $\mathcal{N}$:
\begin{align}\label{macfe}
\pk [f]=\Big(a_{f}-\frac{\rho_{2}}{\rho_{1}}c_{f}\Big)\m^{\frac{1}{2}}+  b_{f}\cdot p \m ^{\frac{1}{2}}+ c_{f}p^{0}\m ^{\frac{1}{2}},
\end{align}
where $a_{f}$, $b_f$ and $c_f$ are coefficients which will  be written as $a,b, c$ when there is no confusion, and
\begin{align}\label{rho12}
   \rho_{1}:=&\int_{{\mathbb R}^3}\m \,\ud p=n u^0 ,\qquad
   \rho_{2}:=\int_{{\mathbb R}^3}p^0\m \,\ud p =\mathfrak{e} (u^0)^2+P|u|^2.
\end{align}

\subsubsection{Notation and Convention}\label{sec:notation} 

Throughout the paper, $C$ denotes a generic positive constant which may change line by line. The notation $A \lesssim B$ implies that there exists a positive constant $C$  such
that $A \leq CB$ holds uniformly over the range of parameters. The
notation $ A \approx B$ means $\frac{1}{\overline{C}}A\leq B\leq \overline{C}A$ for some constant $\overline{C}>1$. 

Let $\brv{\cdot,\cdot}$ denote the $L^2$ inner product in $p\in\r^3$ and $\br{\cdot,\cdot}$ the $L^2$ inner product in $(x,p)\in\r^3\times\r^3$:
\begin{align}
    \brv{f,g}&=\int_{\r^3_p}fg\ud p,\\
    \br{f,g}&=\iint_{\r^3_x\times\r^3_p}fg\ud p\ud x.
\end{align}
Let $\tbs{\,\cdot\,}$ denote the $L^2$ norm in $p\in\r^3$ and $\tnm{\,\cdot\,}$ the $L^2$ norm in $(x,p)\in\r^3\times\r^3$ :
\begin{align}
    \tbs{f}^2=\brv{f,f},\qquad
    \tnm{f}^2=\br{f,f}.
\end{align}
Note that for quantities related to $E$ or $B$ which do not depend on $p$, we also use $\tnm{\,\cdot\,}$ to denote the $L^2$ norm in $x\in\r^3$. Similarly, for $s=0, 1, 2,$ we define the Sobolev norms
\begin{align}
&\|f\|^2_{H^s}
=\sum_{|\alpha|=0}^s\iint_{\mathbb R^3_x\times\mathbb R^3_p}|\partial^{\alpha}_xf|^2\ud p\ud x,
\end{align}
where $\partial^{\alpha}_x=\partial^{\alpha_1}_{x_1}\partial^{\alpha_2}_{x_2}\partial^{\alpha_3}_{x_3}$ with $\alpha=(\alpha_1,\alpha_2,\alpha_3)$ and $\abs{\alpha}=\alpha_1+\alpha_2+\alpha_3$.
For $\m $ given in \eqref{maxwell}, we denote $\big(n , u , T \big)(t,x)$ as part of a solution to the relativistic Euler-Maxwell system \eqref{rem}-\eqref{rem'}, which was constructed in \cite{Guo.Ionescu.Pausader2014}
, and define the following $3\times 3$ matrix-type  collision frequency:
$$\sigma^{ij}(p)=\int_{{\mathbb R}^3} \Phi^{ij} \m (q)\,\ud q. $$
 To measure the dissipation of the linearized relativistic Landau collision, we define the inner product:
 \begin{align}
 (f,g)_{\sigma}=&\sum_{i,j=1}^3\int_{{\mathbb R}^3}\sigma^{ij}\partial_{p_i}f\partial_{p_j}g
 \,\ud p\label{norm}
 +\sum_{i,j=1}^3\frac{1}{4T ^2}\int_{{\mathbb R}^3}\sigma^{ij}\frac{p_i}{p^0}\frac{p_j}{p^0}fg\,\ud p .
  \end{align}
Denote the corresponding $\sigma$ norms:
\begin{align}
    \abss{f}^2=(f,f)_{\sigma}, \qquad
    \nms{f}^2=\int_{\r^3}\abss{f(x)}^2\ud x.
\end{align}
Similarly, for $s=0, 1, 2,$ we define the Sobolev $\sigma$ norms
\begin{align}
&\|f\|^2_{H^s_{\sigma}}=\sum_{|\alpha|=0}^s\int_{\mathbb R^3_x}|\partial^{\alpha}_xf|_{\sigma}^2\ud x.
\end{align}

\begin{remark}
According to \cite{Lemou2000}, all eigenvalues of $\sigma^{ij}(p)$ are positive and depend on $p$. Moreover, the eigenvalues converge to positive constants as $|p|\rightarrow\infty$. Then, for the $|\cdot|_{\sigma}$ norm defined above, we have 
\begin{align}\label{sgm}
   \frac{1}{T } \tbs{f}+\tbs{\nabla_pf}\ls\abss{f}\ls  \frac{1}{T } \tbs{f}+\tbs{\nabla_pf}
\end{align}
\end{remark}

Define the weight functions
\begin{align}\label{wpm}
   w^{\ell}=(p^0)^{2(\nc-\ell)}\exp\Big\{\frac{p^0}{5\ln(\ue+t)\tc}\Big\},\quad 0\leq \ell\leq 2,
\end{align}
where $\nc$ and $\tc$ are constants satisfying $\nc\geq3$ and
\begin{equation}\label{assump}
\tc\geq \sup_{t\in[0, \e^{-1/3}],x\in \mathbb R^3} T (t,x).
\end{equation}
It should be pointed out that the weight functions in \eqref{wpm} are  designed to make sure that
\begin{align}\label{smallw}
(w^{\ell})^2\m ^{\frac{1}{2}}\lesssim e^{-c_0p^0}, \qquad
(w^{\ell})^2(p^0)^{2\ell}\leq \frac{1}{2}\Big((w^{\ell})^2+(w^{0})^2\Big)  
\end{align}
for some small constant $c_0>0$.

Correspondingly, define the weighted norms
\begin{align}
    \tnm{f}_{w^{\ell}}&:=\btnm{w^{\ell}f},\qquad\nm{f}_{H^s_w}:=\sum_{\abs{\alpha}=0}^s\btnm{w^{\abs{\alpha}}\p^{\alpha}_x f} ,\label{weiH}\\
    \nm{f}_{w^{\ell},\sigma}&:=\bnm{w^{\ell}f}_{\sigma},\qquad \nm{f}_{H^s_{w,\sigma}}:=\sum_{\abs{\alpha}=0}^s\bnms{w^{\abs{\alpha}}\p^{\alpha}_x f}.
\end{align}

Denote
\begin{align}
    \ww(t)&:=\exp\left(\frac{1}{5\ln(e+t)\tc}\right),   \quad
    \yy(t):=-\frac{\ww'}{\ww}=\frac{1}{5[\ln(e+t)]^2(e+t)\tc},\label{aa 03}\\
    \zzz(t)&:=\sup_{x\in\mathbb R^3,0\leq \ell\leq 2}\left\{\babs{\nabla_{t,x}^{1+\ell}(n ,{u},T )}+\babs{\nabla_{t,x}^{\ell}(E ,B )}+\babs{\nabla_{t,x}^{\ell}{u}}\right\}.\no
\end{align}

\subsubsection{Key Proposition} 

Theorem \ref{main} follows naturally from the following proposition.

\begin{proposition}\label{result} Let $\f(0,x,p)\geq0$. Assume that $\big(n (t,x), u (t,x), T (t,x), E (t,x)$, $B (t,x)\big)$ is the global solution constructed in Theorem \ref{Euler-Maxwell}. 
Then for $k\geq3$ in the Hilbert expansion \eqref{expan} and $\fe =\m ^{-\frac{1}{2}}\f_{R}$ defined in \eqref{L2}, there exists a  $\e_0 > 0$  such
that for $0 \leq \e\leq\e_0$  and $0<t\leq\overline{t}$ with $\overline{t}=\e^{-1/3}$, if
\begin{align}
    \ee(0)\ls 1,
\end{align}
\eqref{L20} and \eqref{L201} admit a unique solution 
$\big(\fe,E^{\e}_R, B^{\e}_R\big)$ satisfying $\f(t,x,p)\geq0$ and
\begin{align}\label{thm2}
\sup_{0\leq  t\leq \overline{t}}\mathcal{E}(t)+\int_0^{\overline{t}}\mathcal{D}(s)\ud s\lesssim\mathcal{E}(0)+1,
\end{align}
where
\begin{align}\label{eed}
    \ee\simeq&\,\Big(\tnm{\fe}^2+\tnm{E^{\e}_ R}^2+\tnm{B^{\e}_ R}^2+\nmw{(\ik-\pk)[\fe]}^2\Big)\\
    &\,+\e\Big(\tnm{\nx\fe}^2+\tnm{\nx E^{\e}_ R}^2+\tnm{\nx B^{\e}_ R}^2+\tnmww{\nabla_x(\ik-\pk)[\fe]}^2\Big)\no\\
    &\,+\e^2\Big(\tnm{\nabla_x^2\fe}^2+\tnm{\nabla_x^2 E^{\e}_ R}^2+\tnm{\nabla_x^2B^{\e}_ R}^2+\e\tnmwww{ \nabla_x^2\fe}^2\Big), \no
\end{align}
and
\begin{align}\label{ddd}
    \dd\simeq&\,\Big(\e^{-1}\bnms{(\ik-\pk)[\fe]}^2+\e^{-1}\nmsw{(\ik-\pk)[\fe]}^2
    +\yy\tnmw{\sqrt{p^0}(\ik-\pk)[\fe]}^2\Big)\\
    &\,+\Big(\e\bnm{\nabla_x\pk[\fe]}^2+\bnms{(\ik-\pk)[\nabla_x\fe]}^2+\nmsww{\nabla_x(\ik-\pk)[\fe]}^2+\e\yy\tnmww{\sqrt{p^0}\nabla_x(\ik-\pk)[\fe]}^2\Big)\no\\
    &\,+\Big(\e^2\nm{\nabla_x^2\pk[\fe]}^2+\e\nms{(\ik-\pk)[\nabla_x^2\fe]}^2+\e^{2}\nmswww{\nabla_x^2(\ik-\pk)[\fe]}^2+\e^3\yy\tnmwww{\sqrt{p^0}\nabla_x^2(\ik-\pk)[\fe]}^2\Big).\no
\end{align}
\end{proposition}

\begin{remark}
By \eqref{expan}, the estimates \eqref{growth0} of the coefficients $F_{n}, E_n, B_n $ with $ (1\leq n\leq 2k-1)$ in Proposition \ref{fn}, and   \eqref{thm2}, we can obtain \eqref{mainlim}.
\end{remark}


\begin{remark}\label{remark 1}
In this paper, we will focus on deriving the a priori estimate \eqref{thm2} in Proposition \ref{result}. 
Then Theorem \ref{main} naturally follows from a standard iteration/fixed-point argument.
Based on the continuity argument (see \cite{Tao2006}), from now on, we will assume that
\begin{align}\label{rr 01}
     \sup_{0\leq t\leq \overline{t}}\ee(t)\ls \e^{-\frac{1}{2}},
\end{align}
and try to derive \eqref{thm2}.
\end{remark}



\begin{remark}
The irrotational assumption \eqref{irrotational} is necessary in the global well-posedness of Euler-Maxwell equation in \cite{Guo.Ionescu.Pausader2016}. Our proof does not rely on the irrotational assumption. Actually, as long as the fluid equation is well-posed and the solution enjoys proper time decay, our method should be able to justify the convergence.
\end{remark}


\subsection{Hilbert Expansion for the Relativistic Landau Equation}

In this subsection, we will derive the Hilbert expansion of the r-LAN equation \eqref{main1=}, and introduce necessary notations.


We consider the Hilbert expansion for small Knudsen number $\e$,
\begin{align}\label{expan=}
\f(t, x, p):=F+\sum_{n=1}^{2k-1}\e^nF_n(t, x, p)+\e^k \fr(t, x, p),
\end{align}
for some $k\geq2$. To determine the coefficients $F_n(t, x, p)$, we plug  \eqref{expan=} into \eqref{main1=} and
equate the coefficients on both sides of equation in front of different powers of
the parameter $\e$ to obtain:
\begin{align}
\e^{-1}:&\quad \c[F ,F ]=0,\nonumber\\
\e^0:&\quad\dt F +\hp\cdot\nx F =\c[F_1,F ]+\c[F ,F_1],\nonumber\\
 &\ldots\ldots\label{expan2=}\\
\e^n:&\quad \dt F_n+\hp\cdot \nx F_n=\sum_{\substack{i+j=n+1\\i,j\geq0}}\c[F_i,F_j], \nonumber\\
&\ldots\ldots\nonumber\\
\e^{2k-1}:&\quad \dt F_{2k-1}+\hp\cdot\nabla_x F_{2k-1}=\sum_{\substack{i+j=2k\\i,j\geq2}}\c[F_i,F_j].\nonumber
\end{align}
The remainder term $\fr$ satisfies the following equation:
\begin{align}\label{remain=}
&\;\dt\fr+\hp\cdot\nx \fr -\frac{1}{\e}\Big\{\c[\fr ,F ]+\c[F ,\fr ]\Big\}\\
 =&\;\e^{k-1}\c[\fr ,\fr ]+\sum_{i=1}^{2k-1}\e^{i-1}\Big\{\c[F_i, \fr]+\c[\fr , F_i]\Big\}+\ss,\no
\end{align}
where 
\begin{align}
    \ss:=\ds\sum_{\substack{i+j\geq 2k+1\\2\leq i,j\leq2k-1}}\e^{i+j-k}\c[F_i,F_j].
\end{align}
From the first equation in \eqref{expan2=}, we can obtain that $F $ should be a local Maxwellian:
\begin{equation}\label{r-maxwell=}
F (t,x,p)=\m(t,x,p):= \frac{n }{4 \pi T K_2(\gamma)} \exp\left\{\frac{u^{\mu}p_{\mu} }{T }\right\},
\end{equation}
where  
$(n , u, T )(t,x)$ is a solution to the relativistic Euler equations \eqref{re=}. 



We define $\fe$ as
\begin{equation}\label{decom}
\fr(t,x,p) :=\mh(t,x,p) \fe(t,x,p) .
\end{equation}
Then the remainder equation \eqref{remain=} can be rewritten as
\begin{align}\label{re-f=}
    &\;\dt\fe+\hp\cdot\nabla_x\fe 
    +\frac{1}{\e}\li[\fe] \\
    =&\;\e^{k-1}\Gamma[\fe,\fe] +\sum_{i=1}^{2k-1}\e^{i-1} \Big\{\Gamma\big[\mhh F_i,\fe\big]
    +\Gamma\big[\fe,\mhh F_i\big]\Big\}-\mhh\big(\dt\mh+\hp\cdot\nx\mh\big)\fe+\sb,\nonumber
\end{align}
where $\sb:=\mhh\ss$.

\subsubsection{Notation and Convention}

The notation here is mostly similar to those in Section \ref{sec:notation}. 

Define the weights
\begin{align}\label{tt 01=}
    w^{\ell}:=(p^0)^{2(\nc-\ell)}\exp\left(\frac{p^0}{5\ln(\ue+t)\tc}\right),\quad 0\leq \ell\leq 2,
\end{align}
where $\nc\geq3$ is a constant and $\tc$ is a constant satisfying
\begin{align}\label{tt 01'''=}
    \tc\geq \sup_{t\in[0, t_0], x\in \mathbb R^3} T (t,x),
\end{align}
where $t_0$ satisfies \eqref{assump=}. 

For the classical solution $\big(n (t,x), u(t,x), T (t,x)\big)$ to the relativistic Euler equations \eqref{re=}, denote $\ww$ and $\yy$ as in \eqref{aa 03} and
\begin{align}
    \zz&:=\sup_{0\leq t\leq t_0,x\in\mathbb R^3}\left\{\babs{\nabla_{t,x}(n ,u,T )}\frac{(1+T )u^0}{T ^2}\right\},\no\\
    \zzz&:=\sup_{0\leq t\leq t_0,x\in\mathbb R^3,1\leq \ell\leq 3}\left\{\babs{\nabla_{t,x}^{\ell}(n ,u,T )}\right\}.\no
\end{align}

\subsubsection{Key Proposition}

Theorem \ref{main=} follows naturally from the following proposition.

\begin{proposition}\label{result 2=} 
Let $F^{\e}(0,x,p)\geq0$, and  $F  =\mathbf{M}$ as in \eqref{r-maxwell=}. Assume $\big(n (t,x), u(t,x), T (t,x)\big)$ is a sufficiently small solution to the relativistic Euler equations \eqref{re=} satisfying 
\begin{align}\label{semp 3'}
    \sup_{0\leq t\leq t_0,x\in\mathbb R^3,1\leq \ell\leq 3}\left\{\babs{\nabla_{t,x}^{\ell}(n ,u,T )}\right\}\ll1,
\end{align}
and
\begin{align}\label{semp 3}
    \sup_{0\leq t\leq t_0,x\in\mathbb R^3}\left\{\babs{\nabla_{t,x}(n ,u,T )}\frac{(1+T )u^0}{T ^2}\right\}<\infty,
\end{align}
where $t_0>0$ fulfills
\begin{align}\label{assump=}
\frac{1}{10\tc(\ue+t_0)\big(\ln(\ue+t_0)\big)^2}\geq \sup_{0\leq t\leq t_0,x\in\mathbb R^3}\left\{\babs{\nabla_{t,x}(n ,u,T )}\frac{(1+T )u^0}{T ^2}\right\}
\end{align}
for $\tc$ defined in \eqref{tt 01'''=}.
Then the Hilbert expansion \eqref{expan=} with $F_n, 1\leq n\leq 2k-1,$ defined in \eqref{decom=}  holds for $k\geq3$, and for the remainder $\fe=\m^{-\frac{1}{2}}\fr$
satisfying \eqref{re-f=} , there exists a constant $\e_0 > 0$ such
that for $0<\e\leq\e_0$ and $0\leq t\leq t_0$, if
\begin{align}\label{semp 2=}
    \ee(0)\ls 1,
\end{align}
then there exists a solution $F^{\e}(t,x,p)\geq0$ to \eqref{main1=} satisfying
\begin{align}\label{thm2=}
&\sup_{0\leq  t\leq t_0}\ee(t)+\int_0^t\mathcal{D}(s)\ud s\leq \ee(0)+\e^{2k+3},
\end{align}
where
\begin{align}\label{eed=}
    \ee\sim&\,\Big(\tnm{\fe}^2+\tnmw{(\ik-\pk)[\fe]}^2\Big)\\
    &\,+\e\Big(\tnm{\nabla_x\fe}^2+\tnmww{\nabla_x(\ik-\pk)[\fe]}^2\Big)\no\\
    &\,+\e^2\Big(\tnm{\nabla_x^2\fe}^2+\e\tnmwww{ \nabla_x^2\fe}^2\Big),\no
\end{align}
and
\begin{align}\label{ddd=}
    \dd\sim
    &\,\Big(\e^{-1}\nms{(\ik-\pk)[\fe]}^2+\e^{-1}\nmsw{(\ik-\pk)[\fe]}^2
    +\yy\tnmw{\sqrt{p^0}(\ik-\pk)[\fe]}^2\Big)&\\
    &\,+\Big(\e\nm{\nabla_x\pk[\fe]}
    ^2+\nms{\nabla_x(\ik-\pk)[\fe]}^2+\nmsww{\nabla_x(\ik-\pk)[\fe]}^2+\e\yy\tnmww{\sqrt{p^0}\nabla_x(\ik-\pk)[\fe]}^2\Big)\no\\
    &\,+\Big(\e^2\nm{\nabla_x^2\pk[\fe]}
    ^2+\e\nms{\nabla_x^2(\ik-\pk)[\fe]}^2+\e^{2}\nmswww{\nabla_x^2(\ik-\pk)[\fe]}^2+\e^{3}\yy\tnmwww{\sqrt{p^0}\nabla_x^2\fe}^2\Big).\no
\end{align}
\end{proposition}


\begin{remark}
Theorem \ref{Eul-exi} and the additional assumption \eqref{assump=} actually dictate that for $0\leq t\leq t_0$
\begin{align}
    \zz\leq \frac{1}{2}\yy\ \ \text{and}\ \ \zzz\ll1.
\end{align}
This will play a key role in the energy estimates, since the solution to Euler equations does not have time decay.
\end{remark}

\begin{remark}\label{remark 1=}
In this paper, we will focus on deriving the a priori estimate \eqref{thm2=}. 
Then Proposition \ref{result 2=} naturally follows from a standard iteration/fixed-point argument. 
Based on the continuity argument (see \cite{Tao2006}), for the energy estimates in Section 8, we will assume that 
\begin{align}\label{rr 01=}
    \sup_{0\leq t\leq t_0}\ee(t)\ls \e^{-\frac{1}{2}},
\end{align}
and try to derive \eqref{thm2=}. Then in Section \ref{Sec:main-thm-pf}, we will in turn verify the validity of \eqref{rr 01=} with the help of \eqref{thm2=}.
\end{remark}

\subsection{Technical Overview}


In this paper, we will develop a new time-dependent energy method to study the Hilbert expansion of the Landau-type equation in the relativistic framework. This is inspired by Caflisch's pioneering work \cite{Caflisch1980}. It is well known that in the study of the Hilbert expansion of the Boltzmann/Landau-type equation, the main task is to solve the remainder term $\f_{R}=\m^{\frac{1}{2}}\fe$, and one of the most challenging  difficulty is from the linear term with one power moment growth
\begin{align}\label{rdiff}
    \m ^{-\frac{1}{2}}(t,x,v)\fe(t,x,p)\Big\{\dt +\hat{p}\cdot\nabla_x\Big\}{\mh(t,x,p)  }
\end{align}
in the relativistic frame, or the linear term with cubic velocity growth
\begin{align}\label{nordiff}
   \m ^{-\frac{1}{2}}(t,x,v)\fe(t,x,v) \Big\{\dt +v\cdot\nabla_x\Big\}{\mh (t,x,v)}
\end{align}
in the non-relativistic one. 






To tame the velocity growth, Caflisch decomposed the remainder $\f_{R}(t,x,v)$ into low- and high-velocity parts, which satisfy a coupled system and can be separately estimated via a weighted energy method.


This approach motivates us to design a time-dependent weight function 
\begin{align}\label{aa 04}
    w(t,x)=\exp\Big\{\frac{p^0}{5\ln(\ue+t)\tc}\Big\}.
\end{align}
Then this exponential momentum function generates an additional dissipation term to 
control the moment growth terms in \eqref{rdiff}
\begin{align}\label{exmp0}
  w^2\fe \partial_t\fe=\frac{1}{2}\frac{\ud}{\ud t}\big(w\fe\big)^2+\frac{1}{5(\ue+t)[\ln(\ue+t)]^2\tc}
  p^0(w\fe)^2.
\end{align}   
Correspondingly, the troublesome term in \eqref{rdiff} is roughly
\begin{align}\label{exmp1}
  w^2\fe   \m ^{-\frac{1}{2}}\fe\Big\{\dt +\hat{p}\cdot\nabla_x\Big\}\mh \leq \left\{\abs{\nabla_{t,x}(n ,u,T )}\frac{(1+T )u^0}{T ^2}\right\}p^0(w\fe)^2.
\end{align}
As long as
\begin{align}\label{exassump}
    \left\{\abs{\nabla_{t,x}(n ,u,T )}\frac{(1+T )u^0}{T ^2}\right\}< \frac{1}{5(\ue+t)[\ln(\ue+t)]^2\tc}
\end{align}
holds for $t,x$ under consideration, we can suppress the momentum growth in \eqref{rdiff}. Therefore, if 
$\abs{\nabla_{t,x}(n ,u,T )}$ is sufficiently small for all $x\in \mathbb{R}^3$, \eqref{exassump} holds locally in time;  if 
$\abs{\nabla_{t,x}(n ,u,T )}$ further enjoys suitably fast time decay,  \eqref{exassump} holds globally in time.


Since 1980s, time-dependent exponential weight functions have been widely used in the study of the collisional kinetic equations. In 1986, Ukai \cite{Ukai1986} introduced a weight function $w(t,v)\approx \exp\big\{(\alpha-\kappa t)\big(1+|v|^2\big)\big\}$ with $\alpha, \kappa>0, t\in \left[0,\frac{1}{2}\alpha\kappa^{-1}\right]$ to study the local well-posedness of the cutoff Boltzmann equation. Later, this technique was extended by AMUXY \cite{Alexandre.Morimoto.Ukai.Xu.Yang2010, Alexandre.Morimoto.Ukai.Xu.Yang2011ARMA,  Alexandre.Morimoto.Ukai.Xu.Yang2011KRM} for constructing local solutions to the non-cutoff Boltzmann equation in Soblev spaces. In these works, the weight function provides an extra gain of velocity weight at the expense of the loss of the decay
in the time-dependent Maxwellian. 

In the exploration of global classical solutions to the one-species VPB system for cutoff hard potentials and moderately soft potentials, to control the large velocity growth in the nonlinear term due to the Coulomb force,  Duan-Yang-Zhao \cite{Duan.Yang.Zhao2012,Duan.Yang.Zhao2013} introduced another type of weight function $w(t,v)\approx\exp\Big\{\frac{\lambda(1+|v|^2)}{(1+t)^{\vartheta}}\Big\}$,
where $\lambda, \vartheta$ are small positive constants.
By introducing a new time weighted energy framework, Xiao-Xiong-Zhao \cite{Xiao.Xiong.Zhao2013, Xiao.Xiong.Zhao2017} removed the so-called neutral condition assumption on the initial datum in previous work \cite{Duan.Yang.Zhao2013}, and extended this well-posedness result to the very soft potentials case.
 We point out that the nonlinear energy method and macro-micro decomposition technique employed in \cite{Xiao.Xiong.Zhao2013, Xiao.Xiong.Zhao2017}
play an essential role in the proof of the main results of this paper.
Recently, such techniques were further applied in constructing global classical solutions to the cutoff  VMB system, non-cutoff VMB system, and VML system \cite{Duan2014, Duan.Lei.Yang.Zhao2017, Fan.Lei.Zhao2018, Lei.Zhao2014, Wang2015}. 

More recently, a new weight function $w(t,v)\approx\exp\Big\{\big(q_1-q_2\int_0^t q_3(s)\, \ud s\big) (1+|v|)^2\Big\}$, with constants $q_1, q_2>0$ and $q_3$ being a dissipation energy functional, was used in Duan-Yang-Yu \cite{Duan.Yang.Yu2022} to justify the asymptotic convergence in Landau equation.


Technically, the introduction of weight function $w$ in \eqref{aa 04} brings multi-level complications. In order to handle the nonlinear term $\Gamma$, we have to control $L^{\infty}$ norm of $\fe$, which in turn requires spatial regularity up to $H^2$. The more derivatives hit $\m$, the more $p^0$ will be generated. Hence, we have to carefully design a hierarchy of weighted functions $w^{\ell}$ to control all kinds of interactions and nonlinear terms in the energy-dissipation structure. In particular, $\tc$ satisfies \eqref{assump} and \eqref{tt 01'''=} so that 
\begin{align*}
\frac{2p^0}{5\ln(\ue+t)\tc}+\frac{p^{\mu}u_{\mu}}{2T }<0,
\end{align*}
due to the smallness of $u$. This yields that for some constant $c_0>0$
\begin{align*}
    w^{2\ell}\mh&\ls \ue^{-c_0p^0},
\end{align*}
which helps to control the cross terms with both $w^{\ell}\mh$ and polynomial growth in $p^0$. 

Nevertheless, this hierarchy of weighted energy method produces new difficulties, especially from the linear collision operator term $\e^{-1}\li[\fe]$ and the macroscopic part $\pk[\fe]$. 

On the one hand, the linear collision operator term $\e^{-1}\li[\fe]$ and $\pk[\fe]$ do not commute with the spatial derivative operator $\nx$, and thus we have to bound the commutator $\jump{\li,\nx}$. This difficulty was also present in the third author's previous work with Guo \cite{Guo.Xiao2021}.
In the derivative estimate, we have
\begin{align} \label{example1}
  \br{\frac{1}{\e}\nx\li[\fe],\nx\fe}&=\br{ \frac{1}{\e}\li\Big[\nx(\ik-\pk)[\fe]\Big],(\ik-\pk)\nx[\fe]}-\br{\frac{1}{\e}\jump{\li,\nx}(\ik-\pk)[\fe], \nx\fe}\\
    &\geq \frac{\delta}{\e} \bnms{\nx(\ik-\pk)[\fe]}^2-\frac{C\zzz}{\e^2}\bnms{(\ik-\pk)[\fe]}^2-C\zzz\nm{\fe}^2_{H^1}.\no
\end{align}
Noting that only $\e^{-1}\bnms{(\ik-\pk)[\fe]}^2$ is included in the no-weight energy estimate,
this implies that we have to pay the cost of $\frac{1}{\e}$ for the derivative estimate $\bnm{\nx\fe}^2$. This is the very reason to include $\e\bnm{\nx\fe}^2$ and $\e^2\nm{\nx^2\fe}^2$ in $\mathcal{E}(t)$. 

On the other hand, in the weighted energy estimate
\begin{align}\label{example2}
    \br{\frac{1}{\e}\li[\fe],(w^0)^2\fe}&=\br{ \frac{1}{\e}\li\Big[(\ik-\pk)[\fe]\Big],(w^0)^2(\ik-\pk)[\fe]+(w^0)^2\pk[\fe]}\\
    &\geq \frac{\delta}{\e} \nms{w^0(\ik-\pk)[\fe]}^2-\frac{C}{\e}\nms{(\ik-\pk)[\fe]}^2-\frac{C}{\e}\nms{w^0\pk[\fe]}^2\no\\
    &\geq\frac{\delta}{\e} \nms{w^0(\ik-\pk)[\fe]}^2-\frac{C}{\e}\nms{(\ik-\pk)[\fe]}^2-\frac{C}{\e}\nm{\pk[\fe]}^2,\no
\end{align}
while $\nm{\pk[\fe]}^2$ cannot be controlled by the dissipation terms in $\dd$. Noting that due to the Maxwellian in the macroscopic part  $\pk[\fe]$, we naturally have $\tnm{w^0\pk[\fe]}\ls\tnm{\fe}$, and thus the weight function only takes effect for the microscopic part $(\ik-\pk)[\fe]$. Therefore, we may first apply the microscopic projection $(\ik-\pk)$ onto the $\fe$ equation \eqref{re-f=},
and directly estimate $w^0(\ik-\pk)[\fe]$ and $w^1\nx(\ik-\pk)[\fe]$. Then the  trouble term $C\e^{-1}\nm{\pk[\fe]}^2$  in \eqref{example2} would not appear.

However, this $(\ik-\pk)$ projection in turn generates another commutator $\jump{\pk,\hat{p}\cdot\nx}$, which 
may be controlled as
\begin{align}\label{semp 4}
  \Bbr{\jump{\pk,\hat{p}\cdot\nx}[\fe], w^0(\ik-\pk)[\fe]}\ls\frac{1}{\e} \nms{(\ik-\pk)[\fe]}^2+\e \nm{\nx\fe}^2.
\end{align}
The term $\e \nm{\nx\fe}^2$ cannot be controlled by the dissipation term $\e\bnms{(\ik-\pk)[\nx\fe]}^2$ in $\dd$. Similar to \eqref{semp 4}, in the weighted first-order derivative estimate, linear term $\e^2 \nm{\nx^2\fe}^2$ arises. This
reveals that the $(\ik-\pk)$ projection argument requires estimate of one more derivative (e.g. in order to bound $w^0(\ik-\pk)[\fe]$, we need the control of $\tnm{\nx\fe}$), and thus the microscopic projection cannot be applied to the highest-order derivatives. Hence, we have to directly perform the weighted energy estimate for $\nx^2\fe$, which in turn calls for $\e^{3}\nm{(w^2)^2\nx^2\fe}^2$ in $\mathcal{E}(t)$ and  leads to the trouble term $\e^{2}\nm{\nx^2\pk[\fe]}$ again.

To control the worrisome linear terms  $\e\bnm{\nabla_x\pk[\fe]}$ and $\e^2\nm{\nabla_x^2\pk[\fe]}$, we need to capture the macroscopic structure of the remainder equations \eqref{L20} and \eqref{re-f=}. The macroscopic dissipation estimates for $\e\bnm{\nabla_x\pk[\fe]}$ and $\e^2\nm{\nabla_x^2\pk[\fe]}$ are given in Section~\ref{sec: macro} and Section~\ref{Sec:macro-dissipation}. Motivated by \cite{Guo2006}, we give the proof combining the local conservation laws and the macroscopic equations. We write the macroscopic quantities \eqref{macfe},
and obtain the conservation law equations for $a^{\e}, b^{\e}, c^{\e}$.
Although these equations are very complicated, we only need to focus on main terms corresponding to the global Maxwellian case as in \cite{Liu.Zhao2014} and \cite{Strain.Guo2004} without the electromagnetic field. 



For the r-VML system, we discovered a new phenomenon related to the dissipation of the electromagnetic field. Through an intricate analysis of the macroscopic variables $\nm{\nabla_x\pk[\fe]}^2$, we conclude that $a^{\e}$, $E^{\e}_R$ and $\nx a^{\e}$, $\nabla_x E^{\e}_R$, $\nabla_x B^{\e}_R$ belongs to the dissipation $\mathcal{D}$. In the near-global-Maxwellian case (see \cite{Duan2014, Yang.Yu2012}), these dissipation terms are stronger than the energy $\mathcal{E}$. However, in our near-local-Maxwellian case, due to the Hilbert expansion, these dissipation terms are much weaker. Actually, we show that
\begin{align}
    \e\Big(\nm{a^{\e}}^2+\nm{E^{\e}_R}^2 \Big)
    +\e^2\Big(\nm{\nabla_xa^{\e}}^2
    +\nm{\nabla_x E^{\e}_R}^2
    +\nm{\nabla_x B^{\e}_R}^2\Big)\leq \e\ee,
\end{align}
which can be absorbed by $\ee$ after integration w.r.t. time $t$ for $t\in [0, \e^{-1/3}]$.

In addition, motivated by \cite{Guo2002} and \cite{Guo.Hwang.Jang.Ouyang2020}, we justify the positivity of the solution $F^{\e}$. We first perform a careful analysis of the construction of the initial data and prove that $F^{\e}(0)\geq0$. Then by analyzing the elliptic structure of the relativistic Landau operator, we show the validity of maximum principle and conclude that $F^{\e}(t)\geq0$ for all $t\geq 0$.


Compared with the $L^2$--$L^{\infty}$ framework as in \cite{Guo.Jang.Jiang2009,Guo.Jang.Jiang2010, Guo.Xiao2021}, our new method has several advantages.
Firstly, we don't require an explicit lower bound of the temperature $T $:
\begin{align}
    T_M<\max_{t,x}T (t,x)<2T_M
\end{align} 
for some constant $T_M>0$. This extra restriction on $T $ is a technical requirement in the $L^2$--$L^{\infty}$ method, and seems artificial from the physical viewpoint. Secondly, our method works for more general settings, including both the Landau-type and cutoff/non-cutoff Boltzmann-type equations in the relativistic frame. The $L^2$--$L^{\infty}$ framework heavily relies on the analysis of the characteristic, which fails for the presence of the  diffusion effect.


Finally, we briefly discuss the possible applications of our new method. 
First, it is hopeful to apply this method to the relativistic non-cutoff Boltzmann equation.
Due to the absence of momentum derivative estimate and weak dissipation of $\li$, the estimation of nonlinear terms related to the electromagnetic field would be critical. Then we can make use of our time-dependent exponential weight function and ideas in \cite{Xiao.Xiong.Zhao2013, Xiao.Xiong.Zhao2017}. 
Second, it is also very interesting to apply our method to deal with the bounded domain problem. Recently, Duan-Liu-Sakamoto-Strain \cite{Duan.Liu.Sakamoto.Strain2020} proved the global existence of mild solutions to the non-relativistic Landau equation and non-cutoff Boltzmann equation in solution spaces $X_t:=L^1_kL^{\infty}_{t} L^2_{v}$
for $x\in\mathbb{T}^3$, and $X_t:=L^1_{\overline{k}} L^{\infty}_{t} L^2_{x_1, v}$
for $x\in (-1,1)\times \mathbb{T}^2$, where $k$ and $\overline{k}$ are the corresponding variables after Fourier transformation w.r.t. $x$ and $\overline{x}=(x_2,x_3)$, respectively.
For the case $x\in\mathbb{T}^3$, replacing the present $H^2_x$ space  by $L^1_k$ space, the extension of our method can be expected. For the case $x\in (-1,1)\times \mathbb{T}^2$, while we are optimistic that the boundary layer analysis in \cite{Sone2002, Sone2007} can handle the specular-reflection boundary, the in-flow boundary with possible singularity might be more challenging.

This paper is organized as follows: in Section 3, we will present some preliminary lemmas regarding the linear and nonlinear relativistic Landau operators; in Section 4-6, we will prove the a priori estimate for the no-weight and weighted estimates as well as the macroscopic estimates of the r-VML system; in Section 7, we justify Proposition \ref{result}; finally, in Section 8, we consider the r-LAN equation and prove Proposition \ref{result 2=}.


\section{Preliminaries} \label{Sec:prelim}
In this part, we will write down explicit forms of  the operators $\aa, \kk$ and $\Gamma$ and further prove the coercive estimate of the linear collision operator $\li$ and trilinear estimates about  $\Gamma$ for both the r-LAN equation and the r-VML system. Weighted energy estimates for these collision operators will also be established. 

\begin{lemma}\label{ss 06}
For the local Maxwellian $\m$, it holds that
\begin{align}
    \abs{\mhh\dt\mh}+\abs{\mhh\nabla_x\mh}&\leq p^0\zz,\\
    \abs{\nabla_x^{\ell}\left(\mhh\dt\mh\right)}+\abs{\nabla_x^{\ell}\left(\mhh\nabla_x\mh\right)}&\ls \frac{1}{T ^2}\left(p^0\right)^{\ell+1}\zzz,\qquad \ell\geq1.
\end{align}
\end{lemma}
\begin{proof}
Direct computation and the assumption \eqref{exassump}  can justify this.
\end{proof}

\begin{lemma}\label{ss 07}
For the operators $\aa, \kk$ and $\Gamma$
, 
we have
\begin{align}\label{mathA}
\aa[f]=\partial_{p_i}\big(\sigma^{ij}\partial_{p_j}f\big)
-\frac{\sigma^{ij}}{4T ^2}\big(u^0\hat{p}_i-u_{i}\big)\big(u^0\hat{p}_j-u_{j}\big)f+\frac{1}{2 T }\partial_{p_i}\Big(\sigma^{ij}\big(u^0\hat{p}_j-u_{j}\big)\Big)f,
\end{align}
\begin{align}\label{mathK}
\kk[f]=&\left(\partial_{p_i}-\frac{u^0\hat{p}_i-u_i}{2T }\right)\int_{{\mathbb R}^3} \Phi^{ij} (p,q) \mh(p)\mh(q)\left(\frac{-u^0\hat{q}_j+u_j}{2T }f(q)-\partial_{q_j}f(q)\right)\ud q ,
\end{align}
and
\begin{align}\label{Gamma}
\Gamma [f,g]=&\left(\partial_{p_i}-\frac{u^0\hat{p}_i-u_i}{2T }\right)\int_{{\mathbb R}^3} \Phi^{ij} (p,q) \mh(q)\Big(\partial_{p_j}f(p)g(q)-f(p)\partial_{q_j}g(q)\Big)\ud q .
\end{align}
\end{lemma}

\begin{proof} 
This corresponds to \cite[Lemma 6]{Strain.Guo2004}. We first prove \eqref{mathA}. From the definition of the operator $\aa$ and \eqref{Phi}, we have
\begin{align}
\aa[f]&=\mhh (p)\partial_{p_i}\int_{{\mathbb R}^3} \Phi^{ij} (p,q) \left(\partial_{p_j}\big[\mh f\big](p)\mathbf{M}(q)-\big[\mh f\big](p)
\partial_{q_j}\mathbf{M}(q)\big]\right)\ud q\\
&=\mhh (p)\partial_{p_i}\int_{{\mathbb R}^3} \Phi^{ij} (p,q) \mh (p)\mathbf{M}(q)\left(\left(-\frac{u^0\hat{p}_j-u_j}{2T }+\frac{u^0\hat{q}_j-u_j}{T }\right)f(p)+\partial_{p_j}f(p)\right)\ud q\nonumber\\
&=\mhh (p)\partial_{p_i}\int_{{\mathbb R}^3} \Phi^{ij} (p,q) \mh (p)\mathbf{M}(q)\left(\frac{u^0\hat{p}_j-u_j}{2T }f(p)+\partial_{p_j}f(p)\right)\ud q\nonumber\\
&=\partial_{p_i}\bigg(\sigma^{ij}(p)\left(\frac{u^0\hat{p}_j-u_j}{2T }f(p)+\partial_{p_j}f(p)\right)\bigg)
-\sigma^{ij}(p)\frac{u^0\hat{p}_i-u_i}{2T }\left(\partial_{p_j}f(p)+\frac{u^0\hat{p}_j-u_j}{2T }\right) \nonumber\\
&=\partial_{p_i}\Big(\sigma^{ij}\partial_{p_j}f\Big)
-\frac{\sigma^{ij}}{4T ^2}\big(u^0\hat{p}_i-u_{i}\big)
\big(u^0\hat{p}_j-u_{j}\big)f+\frac{1}{2T }\partial_{p_i}\Big(\sigma^{ij}\big(u^0\hat{p}_j-u_{j}\big)\Big)f\nonumber.
\end{align}
Similarly, for \eqref{mathK}, we can obtain that
\begin{align}
\kk [f]&=\mhh (p)\partial_{p_i}\int_{{\mathbb R}^3} \Phi^{ij} (p,q) \left(\partial_{p_j}\mathbf{M}(p)\big[\mh f\big](q)-\mathbf{M}(p)
\partial_{q_j}\big[\mh f\big](q)\right)\ud q\\
&=\mhh (p)\partial_{p_i}\int_{{\mathbb R}^3} \Phi^{ij} (p,q) \mh (q)\mathbf{M}(p)\left(\left(-\frac{u^0\hat{p}_j-u_j}{T }+\frac{u^0\hat{q}_j-u_j}{2T }\right)f(q)-\partial_{q_j}f(q)\right)\ud q\nonumber\\
&=\mhh (p)\partial_{p_i}\int_{{\mathbb R}^3} \Phi^{ij} (p,q) \mh (q)\mathbf{M}(p)\left(\frac{-u^0\hat{p}_j+u_j}{2T }f(q)-\partial_{q_j}f(q)\right)\ud q\nonumber\\
&=\left(\partial_{p_i}-\frac{u^0\hat{p}_i-u_i}{2T }\right)\int_{{\mathbb R}^3} \Phi^{ij} (p,q)\mh (p) \mh (q)\left(\frac{-u^0\hat{q}_j+u_j}{2T }f(q)-\partial_{q_j}f(q)\right)\ud q .\nonumber
\end{align}
For \eqref{Gamma}, we use  \eqref{Phi} again to have
\begin{align}
\Gamma [f,g]&=\mhh (p)\partial_{p_i}\int_{{\mathbb R}^3} \Phi^{ij} (p,q) \left(\partial_{p_j}\big[\mh f\big](p)\big[\mh g\big](q)-\big[\mh f\big](p)
\partial_{q_j}\big[\mh f\big](q)\big]\right)\ud q\\
&=\mhh (p)\partial_{p_i}\int_{{\mathbb R}^3} \Phi^{ij} (p,q)\mh (p)\mh (q)\Big(
\partial_{p_j}f(p)g(q)-f(p)\partial_{q_j}g(q)\Big)\ud q\nonumber\\
&\quad+\mhh (p)\partial_{p_i}\int_{{\mathbb R}^3} \Phi^{ij} (p,q) \mh (p)\mh (q)f(p)g(q)\left(-\frac{u^0\hat{p}_j-u_j}{2T }+\frac{u^0\hat{q}_j-u_j}{2T }\right)\ud q\nonumber\\
&=\mhh (p)\partial_{p_i}\int_{{\mathbb R}^3} \Phi^{ij} (p,q)\mh (p)\mh (q)\Big(
\partial_{p_j}f(p)g(q)-f(p)\partial_{q_j}g(q)\Big)\ud q\nonumber\\
&=\left(\partial_{p_i}-\frac{u^0\hat{p}_i-u_i}{2T }\right)\int_{{\mathbb R}^3} \Phi^{ij} (p,q) \mh (q)\Big(\partial_{p_j}f(p)g(q)-f(p)\partial_{q_j}g(q)\Big)\ud q .\nonumber
\end{align}
\end{proof}

\begin{remark}
From Lemma \ref{ss 07}, we know that when taking $x_i$ derivatives on $\li$ and $\Gamma$, although there will be $p$ or $q$  popped out from $\m$ and $\mh$, they can be absorbed by $\m$ or $\mh$.
\end{remark}

\begin{lemma}\label{ss 01}
The linearized collision operator $\li$ is self-adjoint in $L^2$. It satisfies
\begin{align}\label{coerc}
    \brv{\li[f],f}\gs \abss{(\ik-\pk)[f]}^2.
\end{align}
\end{lemma}

\begin{proof}
Using Lemma \ref{ss 07}, compared with \cite[Lemma 6]{Strain.Guo2004}, for any large constant $R$, it holds that
\begin{align}\label{ppt0}
\brv{\li[f],f}&\gtrsim  |(\ik-\pk)[f]|^2_{\tilde{\sigma}},
\end{align}
where the norm $|\cdot|_{\tilde{\sigma}}$ is defined as
\begin{align}\label{norm-tilde}
 |f|^2_{\tilde{\sigma}}=&\sum_{i,j=1}^3\int_{{\mathbb R}^3}\sigma^{ij}\partial_{p_i}f\partial_{p_j}f\,\ud p
 +\sum_{i,j=1}^3\frac{1}{4T ^2}\int_{{\mathbb R}^3}\sigma^{ij}\big(u^0\hat{p}_i-u_i\big)\big(u^0\hat{p}_j-u_j\big)|f|^2\,\ud p.
\end{align}
Now we show the equivalence of the norm $|\cdot|_{\sigma}$ and $|\cdot|_{\tilde{\sigma}}$ under the smallness assumption of $u$. By the simple inequality
\begin{align}
    (A-B)^2\geq \frac12 A^2-B^2,
\end{align}
we have
\begin{align}\label{sigma0}
&\sum_{i,j=1}^3 \sigma^{ij}\big(u^0\hat{p}_i-u_i\big)\big(u^0\hat{p}_j-u_j\big)\geq \sum_{i,j=1}^3\frac12\big(u^0\big)^2\sigma^{ij}\hat{p}_i\hat{p}_j-\sum_{i,j=1}^3\sigma^{ij}u_iu_j.
  \end{align}
Combining \eqref{norm-tilde} and \eqref{sigma0}, we use the the smallness assumption of $u$ to obtain
\begin{align*}
|f|^2_{\tilde{\sigma}}&\geq\sum_{i,j=1}^3\int_{{\mathbb R}^3}\sigma^{ij}\partial_{p_i}f\partial_{p_j}f\,dp
 +\sum_{i,j=1}^3\frac{(u^0\big)^2}{8T ^2}\int_{{\mathbb R}^3}\sigma^{ij}\hat{p}_i\hat{p}_j|f|^2\,\ud p-\sum_{i,j=1}^3\frac{1}{4T ^2}\int_{{\mathbb R}^3}\sigma^{ij}u_iu_j|f|^2\,\ud p\\
 &\geq \frac18|f|_{\sigma}^2-\frac{C}{T ^2}\|u\|_{L^{\infty}_{t,x}}^2 \tbs{f}^2
 \gtrsim \tbs{\nabla_pf}^2+\tbs{f}^2\gtrsim|f|_{\sigma}^2.\nonumber
\end{align*}
On the other hand, from \eqref{sgm}, we have $|f|_{\sigma}\lesssim |f|_{\tilde{\sigma}}$. \eqref{coerc} follows from \eqref{ppt0} and the above inequality.
\end{proof}

\begin{lemma}\label{ss 02}
The nonlinear collision operator $\Gamma$ satisfies
\begin{align*}
    \babs{\brv{\Gamma[f,g]+\Gamma[g,f],h}}\ls \big(\tbs{f}\abss{g}+\tbs{g}\abss{f}\big)\abss{(\ik-\pk)[h]}.
\end{align*}
\end{lemma}
\begin{proof}
The proof is similar to \cite[Theorem 4]{Strain.Guo2004}, so we omit it here. Compared with \cite[Theorem 4]{Strain.Guo2004}, we need the smallness of $\nm{u}_{L^{\infty}_{t,x}}$ to handle the term 
$-\frac{u \hat{p}_i-u_i}{2T }$ in \eqref{Gamma}.
\end{proof}

\begin{lemma}\label{ss 04}
For the weight functions defined in \eqref{wpm}, it holds that
\begin{align*}
    \brv{(w^{\ell})^2\li[f],f}\gs \abss{(w^{\ell})f}^2-C\abss{f}^2.
\end{align*}
\end{lemma}
\begin{proof}
We split $\li$ as $-\aa$ and $-\kk$ and use the expression of $\aa$ in \eqref{mathA} to integrate by parts w.r.t. $p$ to have
\begin{align}\label{wL}
\brv{\li[f],(w^{\ell})^2 f}=&-\brv{\partial_{i}\big(\sigma^{ij}\partial_jf\big), (w^{\ell})^2f}
+\brv{\frac{\sigma^{ij}}{4T ^2}\big(u^0\hat{p}_i-u_{i}\big)\big(u^0\hat{p}_j-u_{j}\big)f, (w^{\ell})^2f}\\
&-\frac{1}{2T }\brv{\partial_{i}\Big(\sigma^{ij}\big(u^0\hat{p}_j-u_{j}\big)\Big)f, (w^{\ell})^2f}-\brv{ \kk[f], (w^{\ell})^2f}\nonumber\\
=&\brv{\sigma^{ij}\partial_jf, (w^{\ell})^2\partial_{i}f}
+\brv{\frac{\sigma^{ij}}{4T ^2}\big(u^0\hat{p}_i-u_{i}\big)\big(u^0\hat{p}_j-u_{j}\big)f, (w^{\ell})^2f}\nonumber\\
&-\frac{1}{2T }\brv{\partial_{i}\Big(\sigma^{ij}\big(u^0\hat{p}_j-u_{j}\big)\Big)f, (w^{\ell})^2f}-\brv{\sigma^{ij}\partial_jf, \partial_{i}((w^{\ell})^2)f}-\brv{\kk[f],(w^{\ell})^2 f}.\nonumber
\end{align}
Now we estimate the terms in the R.H.S. of the second equal sign in \eqref{wL}.
From \cite[Lemma 5]{Strain.Guo2004}, we know
\begin{equation}
\big|\nabla_{p}^k\sigma^{ij}(p)\big|\lesssim (p^0)^{-k},
\end{equation}
for any integer $k\geq0$. Then for any large constant $R$, we have
\begin{align}
&\,\frac{1}{2 T }\abs{\brv{\partial_{i}\Big(\sigma^{ij}\big(u^0\hat{p}_j-u_{j}\big)\Big)f, (w^{\ell})^2f}}
\lesssim \frac{1}{2 T } \int_{{\mathbb R}^3}\frac{(w^{\ell})^2}{p^0}|f|^2\,\ud p\\
=&\,\frac{1}{2 T } \int_{p^0\leq R}\frac{(w^{\ell})^2}{p^0}|f|^2\,\ud p+\frac{1}{2 T } \int_{p^0> R}\frac{(w^{\ell})^2}{p^0}|f|^2\,\ud p
\lesssim \frac{C_R}{T }\tbs{f}^2+\frac{1}{RT }\tbs{(w^{\ell})f}^2. \no
\end{align}
Noting that 
\begin{align}\label{diff}
\partial_{i}((w^{\ell})^2)=\frac{4(\nc-\ell)}{p^0}\hat{p}_i(w^{\ell})^2+\frac{2\hat{p}_i}{5T_c\ln(\ue+t)}(w^{\ell})^2,
\end{align}
we use Cauchy's inequality to have
\begin{align}
&\;\abs{\brv{\sigma^{ij}\partial_jf, \partial_{i}[(w^{\ell})^2]f}}\\
\leq&\;  C\int_{{\mathbb R}^3}\frac{(w^{\ell})^2}{p^0}\Big(|\nabla_pf|^2+|f|^2\Big)\,\ud p+\abs{\int_{{\mathbb R}^3}\frac{2(w^{\ell})^2}{5\tc\ln(\ue+t)}\Big(\sigma^{ij}\hat{p}_i\hat{p}_j|f|^2\Big)^{\frac{1}{2}}\Big(\sigma^{ij}\partial_if\partial_jf\Big)^{\frac{1}{2}}\,\ud p}\no\\
\leq &\;C_R|f|^2_{\sigma}+\frac{C}{R}|(w^{\ell})f|^2_{\sigma}+\frac{1}{2}\int_{{\mathbb R}^3}(w^{\ell})^2\sigma^{ij}\partial_if\partial_jf\,\ud p+\frac{2}{25\tc^2\ln^2(\ue+t)}\int_{{\mathbb R}^3}(w^{\ell})^2\sigma^{ij}\hat{p}_i\hat{p}_j|f|^2\,\ud p.\no
\end{align}
For the term $\brv{\kk[f],(w^{\ell})^2 f}$, we integrate with respect to $p$ and use \eqref{mathK} and \eqref{diff} to have
\begin{align}
\\
\abs{\brv{\kk[f],(w^{\ell})^2 f}}\lesssim &\int_{{\mathbb R}^3} \int_{{\mathbb R}^3} |\Phi (p,q)| \mh(p)\mh(q)(w^{\ell})^2\Big(|\nabla_pf(p)|+\frac{1}{T_c }|f(p)|\Big)\no\\
&\times\Big(|\nabla_qf(q)|+\frac{1}{T }|f(q)|\Big)\,\ud p\ud q\no\\
\lesssim&\, \Big(\int_{{\mathbb R}^3}|\Phi (p,q)|^2\mathbf{M}(q)\,\ud q\Big)^{\frac{1}{2}}  \Big(\int_{{\mathbb R}^3}\Big(|\nabla_qf(q)|^2+\frac{1}{T ^2}|f(q)|^2\Big)\,\ud q\Big)^{\frac{1}{2}}\no \\
&\times \Big(\int_{{\mathbb R}^3}(w^{\ell})^4\mathbf{M}(p)\,\ud p\Big)^{\frac{1}{2}} \Big(\int_{{\mathbb R}^3}\Big(|\nabla_pf(p)|^2+\frac{1}{T_c ^2}|f(p)|^2\Big)\,\ud q\Big)^{\frac{1}{2}} \no\\
\lesssim &\,\Big(\int_{{\mathbb R}^3}|\Phi (p,q)|^2\mathbf{M}(q)\,\ud p\Big)^{\frac{1}{2}}\abss{f}^2.\no
\end{align}
Here we used \eqref{smallw} to deduce that
\begin{align}
    \int_{{\mathbb R}^3}(w^{\ell})^4\mathbf{M}(p)\,\ud p\lesssim 1.
\end{align}
Then, as in \cite{Strain.Guo2004}, we can split ${\mathbb R}^3_q$ into two regions $A$ and $B$. From \cite[Lemma 2]{Strain.Guo2004}, we have
\begin{align}\label{PhiM}
\int_{{\mathbb R}^3}|\Phi (p,q)|^2\mathbf{M}(q)\,\ud q&=\int_{q\in A}|\Phi (p,q)|^2\mathbf{M}(q)\,\ud q+\int_{q\in B}|\Phi (p,q)|^2\mathbf{M}(q)\,\ud q\\
&\lesssim \int_{{\mathbb R}^3}\big(1+|p-q|^{-2}\big)\mathbf{M}^{1/2}(q)\,\ud q\lesssim 1.\nonumber
\end{align}
Then we can further bound $\abs{\brv{\kk[f],(w^{\ell})^2 f}}$ by $|f|_{\sigma}^2$.

Collecting the above estimates in \eqref{wL}, we use \eqref{assump}, \eqref{tt 01'''=} and \eqref{sigma0} to get
\begin{align}
&\;\brv{\li[f],(w^{\ell})^2 f}\\
\geq&\;\frac12\brv{(w^{\ell})^2\sigma^{ij}\partial_jf, \partial_{i}f}-C_R|f|^2_{\sigma}-C\left(\frac{1}{R}+\nm{u}_{L^{\infty}_{t,x}}^2\right)|(w^{\ell})f|^2_{\sigma}\no\\
&+\left(\frac{\big(u^0\big)^2}{8T ^2}-\frac{2}{25\tc^2\ln^2(\ue+t)}\right)
\int_{{\mathbb R}^3}(w^{\ell})^2\sigma^{ij}\hat{p}_i\hat{p}_j|f|^2\,\ud p\no\\
\geq&\;\frac12\brv{(w^{\ell})^2\sigma^{ij}\partial_jf, \partial_{i}f}-C_R|f|^2_{\sigma}-C\left(\frac{1}{R}+\nm{u}_{L^{\infty}_{t,x}}^2\right)|(w^{\ell})f|^2_{\sigma}+\frac{9\big(u^0\big)^2}{200T ^2}
\int_{{\mathbb R}^3}(w^{\ell})^2\sigma^{ij}\hat{p}_i\hat{p}_j|f|^2\,\ud p\no\\ 
\gtrsim&\; |(w^{\ell})f|^2_{\sigma} -C_R|f|^2_{\sigma}\no \nonumber
\end{align}
by choosing $R$ large enough. 
\end{proof}

\begin{lemma}\label{ss 05}
For the weight functions defined in \eqref{wpm}, it holds that
\begin{align}\label{Weight1} 
    \brv{\Gamma[f,g],(w^{\ell})^2h}\ls \Big(\tbs{(w^{\ell})f}\abss{g}+\tbs{g}\abss{(w^{\ell})f}\Big)\abss{(w^{\ell})h}.
\end{align}
\end{lemma}
\begin{proof}
From \eqref{Gamma}, we integrate by parts with respect to $p$ and use \eqref{diff} to get
\begin{align}\label{WGamma}
&\abs{\brv{\Gamma[f,g],(w^{\ell})^2 h}}\\
=&\abs{\brv{\int_{{\mathbb R}^3} \Phi^{ij} (p,q) \mh(q)\Big(\partial_{p_j}f(p)g(q)-f(p)\partial_{q_j}g(q)\Big)\ud q,\left(\partial_{p_i}-\frac{u \hat{p}_i-u_i}{2 T }\right)((w^{\ell})^2 h)}}\nonumber\\
\lesssim& \iint_{\r^3\times\r^3} \Phi^{ij} (p,q) \mh(q)\abs{\partial_{p_j}f(p)g(q)-f(p)\partial_{q_j}g(q)}(w^{\ell})^2 \Big(\frac{|h|}{T_c}+|\partial_{p_i}h|\Big)\,\ud p\ud q.\nonumber
\end{align}
By H\"{o}lder's inequality, we can use \eqref{PhiM} to further estimate \eqref{WGamma} as
\begin{align}
\abs{\brv{\Gamma[f,g],(w^{\ell})^2 h}}
&\lesssim | \Phi(p,q) \mh|_{L^2_q}\Big(|(w^{\ell})\partial_{p_j}f|_{L^2}|g|_{L^2}+|(w^{\ell})f|_{L^2}|\partial_{q_j}g|_{L^2}\Big) |(w^{\ell})h|_{\sigma}\\
&\lesssim \Big(|(w^{\ell})f|_{\sigma}\tbs{g}+\tbs{(w^{\ell})f}|g|_{\sigma}\Big)|(w^{\ell})h|_{\sigma}.\nonumber
\end{align}
\end{proof}

\section{No-Weight Energy Estimates} \label{l2}

In this section, we derive the $L^2$ energy estimates for the remainders $\big(\fe, E_R^{\e}, B_R^{\e}\big)$.

\subsection{Basic \texorpdfstring{$L^2$}{} Estimates} \label{l21}

We first perform the simplest $L^2$ energy estimate for the remainders. 

\begin{proposition}\label{L2ener}
For the remainders $\big(\fe, E^{\e}_ R, B^{\e}_ R\big)$, it holds that
\begin{align}\label{L2ener01}
&\,\frac{\ud}{\ud t}\bigg(\tnm{\sqrt{\frac{4\pi T }{u^0}}\fe}^2+\tnm{E^{\e}_R}^2+\tnm{B^{\e}_ R}^2\bigg)+\frac{\delta}{\e}\nms{(\ik-\pk)[\fe]}^2\\
\lesssim &\, \left[(1+t)^{-\beta_0}+\e^{\frac{1}{3}}\right]\mathcal{E}+\e^2\mathcal{D} +\e^{2k+1}(1+t)^{4k+2}
+\e^{k}(1+t)^{2k}\sqrt{\mathcal{E}}.\no
\end{align}
\end{proposition}

\begin{proof}
From \eqref{L201}, we have
\begin{align}\label{L2max}
\frac{1}{2}\frac{\ud}{\ud t}\Big(\tnm{E_R^{\e}}^2+\tnm{B_R^{\e}}^2\Big)=4\pi\br{\hat{p}\mh \fe, E_R^{\e}}.
\end{align}
We take the $L^2$ inner product of \eqref{L20} with $\frac{4\pi T }{u^0}\fe$ and use \eqref{coerc}, \eqref{L2max} to have
\begin{align}\label{L202}
    & \frac{1}{2}\frac{\ud}{\ud t}\bigg(\tnm{\sqrt{\frac{4\pi T }{u^0}}\fe}^2+\tnm{E_R^{\e}}^2+\tnm{B_R^{\e}}^2\bigg)+\frac{\delta}{\e}\nms{(\ik-\pk)[\fe]}^2\\
    \leq&\abs{\br{ \m ^{-\frac{1}{2}}\fe\big[\dt +\hat{p}\cdot\nabla_x- \big(E +\hat{p} \times B  \big)\cdot\nabla_p \big]\mh ,\frac{4\pi T }{u^0}\fe}}\no\\
    &+\abs{\frac{1}{2}\br{\left\{\big(\dt +\hat{p}\cdot\nabla_x\big)\left[\frac{T }{u^0}\right]\right\}\fe, 4\pi \fe}}\no\\
    &+\abs{\br{ -u \mh \cdot\big(E_R^{\e}+\hat{p} \times B_R^{\e} \big), \frac{4\pi}{u^0} \fe}}+\abs{\e^{k-1}\br{\Gamma\left[\fe,\fe\right],\frac{4\pi T }{u^0}\fe}}\no\\
    &+\abs{\sum_{i=1}^{2k-1}\e^{i-1}\br{\Gamma\left[ \m ^{-\frac{1}{2}}F_i, \fe\right]+\Gamma\left[\fe,  \m ^{-\frac{1}{2}}F_i\right],\frac{4\pi T }{u^0}\fe}}\no\\
    &+\abs{\e^k\br{\big(u^0\hat{p}-u\big)\cdot\big(E_R^{\e}+\hat{p} \times B_R^{\e} \big)\fe,\frac{2\pi \fe}{u^0}}}\no\\
    &+\abs{\sum_{i=1}^{2k-1}\e^i\br{\big(E_i+\hat{p} \times B_i \big)\cdot\nabla_p\fe +\big(E_R^{\e}+\hat{p} \times B_R^{\e} \big)\cdot \m ^{-\frac{1}{2}}\nabla_pF_{i},\frac{4\pi T }{u^0}\fe }}\no\\
    &+\abs{\sum_{i=1}^{2k-1}\e^i\br{\big(E_i+\hat{p} \times B_i \big)\cdot\big(u^0\hat{p}-u \big)\fe ,\frac{2\pi}{u^0}\fe }}+
    \abs{\e^{k}\br{\sb ,\frac{4\pi T }{u^0}\fe}}.\no
\end{align}
Now we estimate each term on the R.H.S. of \eqref{L202}.

\textbf{First Term on the R.H.S. of \eqref{L202}:}
Note that 
\begin{align*}
\m^{-\frac12}\left\{\dt +\hat{p}\cdot\nabla_x- \big(E +\hat{p} \times B  \big)\cdot\nabla_p \right\}\mh 
\end{align*}
are the first-order  polynomials of $p$.
Then, for a given sufficiently small positive constant $\kappa$, we  have
\begin{align}\label{L200}
&\abs{\br{ \m ^{-\frac{1}{2}}\fe \big[\dt +\hat{p}\cdot\nabla_x- \big(E +\hat{p} \times B  \big)\cdot\nabla_p \big]\mh ,\frac{4\pi T }{u^0}\fe }}\\
\lesssim& \Big(\nm{E }_{L^{\infty}}+\nm{B }_{L^{\infty}}\Big)
\nm{\fe}^2+\|\nabla_x(n ,u,T )\|_{L^{\infty}}\nm{\sqrt{p^0}\fe}^2\no\\
\ls&\zzz\bigg(\nm{\sqrt{p^0}\pk[\fe]}^2+\int_{\r^3}\int_{p^0\leq\e^{-1}\kappa}p^0\abs{(\ik-\pk)[\fe]}^2+\int_{\r^3}\int_{p^0\geq\e^{-1}\kappa}p^0\abs{(\ik-\pk)[\fe]}^2\bigg)\no\\
        \ls&(1+t)^{-\beta_0
        }\tnm{\fe}^2+o(1)\e^{-1}\nms{(\ik-\pk)[\fe]}^2+\e\nmsw{(\ik-\pk)[\fe]}^2.\no
\end{align}
Here we have used \eqref{decay} and $\e (w^0)^2\gs\e (p^0)^2\gs p^0$ for $p^0\gs\e^{-1}$.

\begin{remark}
The decay estimate $\zzz\ls (1+t)^{-\beta_0}$ is crucial here.
\end{remark}

\textbf{Second and Third Terms on the R.H.S. of \eqref{L202}:}
From \eqref{decay} again, we have
\begin{align*}
&\abs{\frac{1}{2}\br{\left\{\big(\dt +\hat{p}\cdot\nabla_x\big)\left[\frac{T }{u^0}\right]\right\}\fe , 4\pi \fe }} +\abs{\br{ -u  \mh \cdot\big(E_R^{\e}+\hat{p} \times B_R^{\e} \big), \frac{4\pi}{u^0} \fe }}\\
\lesssim&\|\nabla_x(n ,u,T )\|_{L^{\infty}}
\nm{\fe}^2+\|u\|_{L^{\infty}}\nm{\fe}\Big(\|E^{\e}_R\|+\|B^{\e}_R\|\Big)\no\\
\lesssim&o(1)\e^{-1}\nms{(\ik-\pk)[\fe]}^2+(1+t)^{-\beta_0}\Big(\nm{\fe}^2+\nm{E_R^{\e}}^2+\nm{B_R^{\e}}^2\Big).\no
\end{align*}

\textbf{Fourth Term on the R.H.S. of \eqref{L202}:}
We use Lemma \ref{ss 02} and \eqref{rr 01} to estimate it as
\begin{align}
&\e^{k-1}\abs{\br{\Gamma\left[\fe,\fe\right],\frac{4\pi T }{u^0}\fe}}
=\e^{k-1}\abs{\br{\Gamma\left[\fe,\fe\right],\frac{4\pi T }{u^0}(\ik-\pk)[\fe]}}\\
\ls&\e^{k-1}\int_{x\in\r^3}\tbs{\fe}\abss{\fe}\abss{(\ik-\pk)[\fe]}
    \ls\e^{k-1}\nm{\fe}_{H^2}\Big(\nms{(\ik-\pk)[\fe]}+\nms{\pk[\fe]}\Big)\nms{(\ik-\pk)[\fe]}\no\\
    \ls &\nms{(\ik-\pk)[\fe]}^2+\e\nms{\pk[\fe]}^2\ls \nms{(\ik-\pk)[\fe]}^2+\e\tnm{\fe}^2 .\no
\end{align}

\textbf{Fifth Term on the R.H.S. of \eqref{L202}:}
Similarly, considering that $F_i$ decay fast in $p$ by \eqref{growth0} in Proposition \ref{fn} and $t\leq \overline{t}=\e^{-1/3}$, we have
\begin{align*}
&\abs{\sum_{i=1}^{2k-1}\e^{i-1}\br{\Gamma\left[ \m ^{-\frac{1}{2}}F_i, \fe\right]+\Gamma\left[\fe,  \m ^{-\frac{1}{2}}F_i\right],\frac{4\pi T }{u^0}\fe}}\\
\ls& \sum_{i=1}^{2k-1}\e^{i-1}\Big(\nm{\m ^{-\frac{1}{2}} F_i}_{L^{\infty}_x L^2_p}\nms{\fe}+\Big\|\big|\m ^{-\frac{1}{2}}F_i\big|_{\sigma}\Big\|_{L^{\infty}_x}\nm{\fe}\Big)\nms{(\ik-\pk)[\fe]}\\
\ls& \sum_{i=1}^{2k-1}\e^{i-1}(1+t)^i\nms{\fe}\nms{(\ik-\pk)[\fe]}
\ls o(1)\e^{-1}\nms{(\ik-\pk)[\fe]}^2+\e(1+t)^2\nms{\fe}^2\no\\
\ls& o(1)\e^{-1}\nms{(\ik-\pk)[\fe]}^2+\e(1+t)^2\nms{\pk[\fe]}^2
    \ls o(1)\e^{-1}\nms{(\ik-\pk)[\fe]}^2+\e^{\frac{1}{3}}\tnm{\fe}^2.\no
\end{align*}

\textbf{Sixth Term on the R.H.S. of \eqref{L202}:}
According to the assumption \eqref{rr 01}, its  upper bound is
\begin{align*}
    \e^k\Big(\nm{E_R^{\e}}_{H^2}+\nm{B_R^{\e}}_{H^2}\Big) \nm{\fe}^2\ls \e\nm{\fe}^2.
\end{align*}

\textbf{Seventh Term on the R.H.S. of \eqref{L202}:}
We use \eqref{growth0} in Proposition \ref{fn} to obtain
\begin{align*}
&\abs{\sum_{i=1}^{2k-1}\e^i\br{\big(E_i+\hat{p} \times B_i \big)\cdot\nabla_p\fe +\big(E_R^{\e}+\hat{p} \times B_R^{\e} \big)\cdot \m ^{-\frac{1}{2}}\nabla_pF_{i},\frac{4\pi T }{u^0}\fe }}\\
=&\sum_{i=1}^{2k-1}\e^i\abs{\br{ \big(E_R^{\e}+\hat{p} \times B_R^{\e} \big)\cdot \m ^{-\frac{1}{2}}\nabla_pF_{i} ,\frac{4\pi T }{u^0}\fe }}\no\\
\lesssim&\sum_{i=1}^{2k-1}\e^{i}\nm{ \m ^{-\frac{1}{2}}\nabla_p F_i}_{L^{\infty}_x L^2_p}\Big(\nm{E_R^{\e}}+\nm{B_R^{\e}}\Big)
\nm{\fe}\lesssim\sum_{i=1}^{2k-1}\e^{i}(1+t)^i\big(\nm{\fe}^2+\nm{E_R^{\e}}^2+\nm{B_R^{\e}}^2\big)\no\\
\lesssim&\e(1+t)\big(\nm{\fe}^2+\nm{E_R^{\e}}^2+\nm{B_R^{\e}}^2\big)\lesssim \e^{\frac{2}{3}}\big(\nm{\fe}^2+\nm{E_R^{\e}}^2+\nm{B_R^{\e}}^2\big).\no
\end{align*}

\textbf{Eighth and Ninth Terms on the R.H.S. of \eqref{L202}:}
Similarly, we estimate the last two terms as
\begin{align*}
&\abs{\sum_{i=1}^{2k-1}\e^i\br{\big(E_i+\hat{p} \times B_i \big)\cdot\big(u^0\hat{p}-u \big)\fe ,\frac{2\pi}{u^0}\fe }}\\
\lesssim&\sum_{i=1}^{2k-1}\e^{i}\Big(\nm{E_i}_{L^{\infty}}+\nm{B_i}_{L^{\infty}}\Big)
\nm{\fe}^2\lesssim\sum_{i=1}^{2k-1}\e^{i}(1+t)^i\nm{\fe}^2\ls \e^{\frac{2}{3}}\nm{\fe}^2\no
\end{align*}
for $t\leq \overline{t}=\e^{-1/3}$, and
\begin{align}\label{s0}
\e^{k}\abs{\br{\sb ,\frac{4\pi T }{u^0}\fe}}
\leq& o(1)\e^{-1}\nms{(\ik-\pk)[\fe]}^2+C\e^{2k+1}\sum_{\substack{i+j\geq 2k+1\\2\leq i,j\leq2k-1}}\e^{2(i+j-2k-1)}(1+t)^{2(i+j)}\\
&+C\sum_{\substack{i+j\geq 2k\\1\leq i,j\leq2k-1}}\e^{i+j-k}(1+t)^{i+j}\nm{\fe}\no\\
\leq& o(1)\e^{-1}\nms{(\ik-\pk)[\fe]}^2+C\e^{2k+1}(1+t)^{4k+2}
+C\e^{k}(1+t)^{2k}\nm{\fe}.\no
\end{align}

\textbf{Summary:}
We collect these estimates in \eqref{L202} to obtain \eqref{L2ener01}.
\end{proof}

\subsection{First-Order Derivatives Estimates}\label{h1f}

In this part, we continue to perform the $L^2$ energy estimates for the first-order derivatives of the remainders  $(\fe, E^{\e}_ R, B^{\e}_ R)$. To this end, we first
apply $\partial^{\alpha}_x (1\leq |\alpha|\leq 2)$ to \eqref{L20} to have
\begin{align}\label{alpha}
    & \partial^{\alpha}_x\bigg(\Big\{\dt +\hat{p}\cdot\nabla_x- \big(E +\hat{p} \times B  \big)\cdot\nabla_p \Big\}\left[\fe \right]\bigg)
    +\partial^{\alpha}_x\bigg(\frac{u^0}{T }\hat{p}\mh \cdot E_R^{\e}\bigg)\\
    &-\partial^{\alpha}_x\bigg(\frac{u }{ T } \mh \cdot
    \big(E_R^{\e}+\hat{p} \times B_R^{\e} \big)\bigg)+\frac{\partial^{\alpha}_x\li [\fe ]}{\e}\no\\
    =&-\partial^{\alpha}_x\bigg( \m ^{-\frac{1}{2}}\fe 
    \Big\{\dt +\hat{p}\cdot\nabla_x-\big(E +\hat{p} \times B  \big)\cdot\nabla_p \Big\}
    \mh \bigg)+\e^{k-1}\partial^{\alpha}_x\Gamma \left[\fe,\fe\right]\no\\
    &+\sum_{i=1}^{2k-1}\e^{i-1}\partial^{\alpha}_x\bigg(\Gamma 
    \left[ \m ^{-\frac{1}{2}}F_i, \fe\right]+\Gamma \left[\fe,  \m ^{-\frac{1}{2}}F_i\right]\bigg)
    + \e^k\partial^{\alpha}_x\bigg(\big(E_R^{\e}+\hat{p} \times B_R^{\e} \big)\cdot\nabla_p\fe \bigg)\no\\
    &- \e^k\partial^{\alpha}_x\bigg(\frac{1 }{2 T }\big(u^0\hat{p}-u \big)\cdot\big(E_R^{\e}+\hat{p} \times B_R^{\e} \big)\fe \bigg)\no\\
    &+ \sum_{i=1}^{2k-1}\e^i\partial^{\alpha}_x\bigg(\big(E_i+\hat{p} \times B_i \big)\cdot\nabla_p\fe +\big(E_R^{\e}+\hat{p} \times B_R^{\e} \big)\cdot \m ^{-\frac{1}{2}}\nabla_pF_{i}\bigg)\no\\
    &- \sum_{i=1}^{2k-1}\e^i\partial^{\alpha}_x\bigg(\big(E_i+\hat{p} \times B_i \big)\cdot\frac{1}{2 T }\big(u^0\hat{p}-u \big)\fe \bigg)+\e^{k}\partial^{\alpha}_x\sb .
\no
\end{align}


\begin{proposition}\label{H1x}
For the remainders $\Big(\fe, E^{\e}_ R, B^{\e}_ R\Big)$, it holds that
\begin{align}\label{H1ener01}
&\e\bigg[\frac{\ud}{\ud t}\bigg(\nm{\sqrt{\frac{4\pi T }{u^0}}\nabla_x\fe}^2+\nm{\nabla_xE^{\e}_ R}^2+\nm{\nabla_xB^{\e}_ R]}^2\bigg)+\frac{\delta}{\e}\nms{\nabla_x(\ik-\pk)[\fe]}^2\bigg]\\
\lesssim &\Big[(1+t)^{-\beta_0}+\e^{\frac{1}{3}}\Big]\mathcal{E}+ \overline{\e}_0\e^{-1}\nms{(\ik-\pk)[\fe]}^2+\e\mathcal{D} +\e^{2k+2}(1+t)^{4k+2}
+\e^{k+1}(1+t)^{2k}\sqrt{\mathcal{E}}.\no
\end{align}
\end{proposition}

\begin{proof} 
From \eqref{L201}, we have
\begin{align}\label{H1max}
    \frac{1}{2}\frac{\ud}{\ud t}\Big(\nm{\partial_xE_R^{\e}}^2+\nm{\partial_x B_R^{\e}}^2\Big)=4\pi\br{\partial_x\Big( \hat{p}\mh \fe\Big), \partial_xE_R^{\e}}.
\end{align}
Here and below we use $\partial_x$ to denote $\partial_{x_i}, 1\leq i \leq 3,$ for simplicity. Note that \begin{align*}
    \jump{\li,\partial_x}[\fe]&=\li[\partial_x\fe]-\partial_x\Big(\li[\fe]\Big)=\li\Big[(\ik-\pk)[\partial_x\fe]\Big]-\partial_x\Big(\li\Big[(\ik-\pk)[\fe]\Big]\Big)\\
    &=-\li\Big[\jump{\pk,\partial_x}[\fe]\Big]+\jump{\li,\partial_x}\big[(\ik-\pk)[\fe]\big].\no
\end{align*}
Hence, naturally we have
\begin{align*}
    \e^{-1}\abs{\br{\jump{\li,\partial_x}[\fe],\frac{4\pi T }{u^0 }\partial_x\fe}}\ls& \e^{-1}\abs{\br{\li\Big[\jump{\pk,\partial_x}[\fe]\Big],\frac{4\pi T }{u^0 }\partial_x\fe}}\\
    &+\e^{-1}\abs{\br{\jump{\li,\partial_x}\big[(\ik-\pk)[\fe]\big],\frac{4\pi T }{u^0 }\partial_x\fe}}.\no
\end{align*}
For the first term, we have
\begin{align*}
    \e^{-1}\abs{\br{\li\Big[\jump{\pk,\partial_x}[\fe]\Big],\frac{4\pi T }{u^0 }\partial_x\fe}}=&\e^{-1}\abs{\br{\li\Big[\jump{\pk,\partial_x}[\fe]\Big],\frac{4\pi T }{u^0 }(\ik-\pk)[\partial_x\fe]}}.
\end{align*}
Note that $\jump{\pk,\partial_x}$ only contains terms that $\partial_x$ hits the Maxwellian but not $\fe$. Hence, we have
\begin{align*}
   \e^{-1}\abs{\br{\li\Big[\jump{\pk,\partial_x}[\fe]\Big],\frac{4\pi T }{u^0 }\partial_x\fe}}
   \ls& o(1)\e^{-1}\nms{(\ik-\pk)[\partial_x\fe]}^2+\e^{-1}\zzz\tnm{\fe}^2 \no\\
    \ls& o(1)\e^{-1}\nms{\partial_x(\ik-\pk)[\fe]}^2+\overline{\e}_0\e^{-1}(1+t)^{-\beta_0}\tnm{\fe}^2\no.
\end{align*}
For the second term, since $\jump{\li,\partial_x}$ indicates that $\partial_x$ only hits the Maxwellian in $\li$ but not on $(\ik-\pk)[\fe]$, we directly bound
\begin{align*}
    &\e^{-1}\abs{\br{\jump{\li,\partial_x}\big[(\ik-\pk)[\fe]\big],\frac{4\pi T }{u^0 }\partial_x\fe}}
    \ls \e^{-1}\zzz \nms{(\ik-\pk)[\fe]}\nms{\partial_x\fe}
   \\
    \ls&(1+t)^{-\beta_0}\overline{\e}_0\tnm{\nabla_x\fe}^2+o(1)\e^{-1}\nms{(\ik-\pk)[\partial_x\fe]}^2+\overline{\e}_0\e^{-2}\nms{(\ik-\pk)[\fe]}^2.\no
\end{align*}
Noting 
\begin{align}\label{xp}
    \nms{(\ik-\pk)[\partial_x\fe]}\geq \nms{\partial_x(\ik-\pk)[\fe]}-\nms{\jump{\pk,\partial_x}[\fe]}\geq \nms{\partial_x(\ik-\pk)[\fe]}-\zzz\tnm{\fe},
\end{align}
In total, we have 
\begin{align}\label{H1L}
  & \e^{-1}\br{ \partial_x\li\left[\fe\right],\frac{4\pi  T }{u^0 }\partial_x\fe}\\
   \geq& \e^{-1}\br{ \li\left[\partial_x\fe\right],\frac{4\pi T }{u^0 }\partial_x\fe}-\e^{-1} \abs{\br{\jump{\li,\partial_x}[\fe],\frac{4\pi T }{u^0 }\partial_x\fe}}\no\\
   \geq
    &\delta \e^{-1}\nms{(\ik-\pk)[\partial_x\fe]}^2-C\overline{\e}_0\e^{-2}\zzz\nms{(\ik-\pk)[\fe]}^2-C\zzz\tnm{\nabla_x\fe}^2-\e^{-1}\zzz\tnm{\fe}^2.\no\\
     \geq
    &\delta\e^{-1}\nms{\partial_x(\ik-\pk)[\fe]}^2-C\overline{\e}_0\e^{-2}\nms{(\ik-\pk)[\fe]}^2-C(1+t)^{-\beta_0}\overline{\e}_0\big(\tnm{\nabla_x\fe}^2+\e^{-1}\tnm{\fe}^2\big).\no
\end{align}
Taking $|\alpha|=1$ in \eqref{alpha} and denoting $\partial_x$ as $\partial_{x_i}, i=1,2,3$ for convenience, we multiply the equation by
$\frac{4\pi T }{u^0}\partial_x\fe $, and use \eqref{H1max}, \eqref{H1L} to have
\begin{align}\label{H10}
    & \frac{1}{2}\frac{\ud}{\ud t}\bigg(\nm{\sqrt{\frac{4\pi T }{u^0}}\partial_x\fe}^2+\nm{\partial_xE^{\e}_ R}^2+\nm{ \partial_xB^{\e}_ R]}^2\bigg)+\frac{\delta}{\e}\nms{\partial_x(\ik-\pk)[\fe]}^2\\
    \leq&\abs{\br{\partial_x\bigg(\frac{u^0}{T }\hat{p}\mh \bigg)\cdot E_R^{\e}, \frac{4\pi T }{u^0 }\partial_x\fe }}+4\pi\abs{\br{\Big( \hat{p}\big(\partial_x\mh \big)\fe\Big), \partial_xE_R^{\e}}}\no\\
    &+\abs{\br{\partial_x\bigg(\Big\{\big(E +\hat{p} \times B  \big)\cdot\nabla_p\Big\}\left[\fe \right]\bigg),\frac{4\pi T }{u^0 }\partial_x\fe }}\no\\
    &+\abs{\br{\partial_x\bigg( \m ^{-\frac{1}{2}}\fe \Big\{\dt +\hat{p}\cdot\nabla_x -\big(E +\hat{p} \times B  \big)\cdot\nabla_p \Big\}\left[\mh \right]\bigg),\frac{4\pi T }{u^0 }\partial_x\fe  }}\no\\
    &+\abs{\frac{1}{2}\br{(\dt +\hat{p}\cdot\nabla_x)\left[\frac{T }{u^0}\right],4\pi \abs{\partial_x\fe }^2}}\no\\
    &+ \abs{\br{\partial_x \bigg(u  \mh \cdot\big(E_R^{\e}+\hat{p} \times B_R^{\e} \big)\bigg), \frac{4\pi T }{u^0 }\partial_x \fe }}+\e^{k-1}\abs{\br{\partial_x\Gamma\left[\fe,\fe\right],\frac{4\pi T }{u^0}\partial_x\fe}}\no\\
    &+\abs{\sum_{i=1}^{2k-1}\e^{i-1}\br{\partial_x\Gamma\left[ \m ^{-\frac{1}{2}}F_i, \fe\right]+\partial_x\Gamma\left[\fe,  \m ^{-\frac{1}{2}}F_i\right],\frac{4\pi T }{u^0}\partial_x\fe}}\no\\
    &+\e^k\abs{\br{\partial_x\bigg(\frac{\big(u^0\hat{p}-u \big)}{2T }\cdot\big(E_R^{\e}+\hat{p} \times B_R^{\e} \big)\fe \bigg),\frac{4\pi T }{u^0 }\partial_x\fe  }}\no\\
    &+\e^k\abs{\br{\partial_x\bigg(\big(E_R^{\e}+\hat{p} \times B_R^{\e} \big)\cdot\nabla_p\fe \bigg),\frac{4\pi T }{u^0 }\partial_x\fe  }}\no\\
    &+\abs{\sum_{i=1}^{2k-1}\e^i\br{\partial_x\bigg(\big(E_i+\hat{p} \times B_i \big)\cdot\nabla_p\fe +\big(E_R^{\e}+\hat{p} \times B_R^{\e} \big)\cdot \m ^{-\frac{1}{2}}\nabla_pF_{i}\bigg),\frac{4\pi T }{u^0 }\partial_x\fe}}\no\\
    &+\abs{\sum_{i=1}^{2k-1}\e^i\br{\partial_x\bigg(\big(E_i+\hat{p} \times B_i \big)\cdot\frac{\big(u^0\hat{p}-u \big)}{2T }\fe \bigg),\frac{4\pi T }{u^0 }\partial_x\fe  }}\no\\
    &+ \e^{k}\abs{\br{ \partial_x\sb ,\frac{4\pi T }{u^0}\partial_x\fe}}+C\overline{\e}_0\Big[\e^{-2}\nms{(\ik-\pk)[\fe]}^2+(1+t)^{-\beta_0}\big(\tnm{\nabla_x\fe}^2+\e^{-1}\tnm{\fe}^2\big)\Big].\no
\end{align}
Now we estimate each term on the R.H.S. of \eqref{H10}. 

\textbf{First Two Terms on the R.H.S. of \eqref{H10}:}
The two terms  can be bounded by
\begin{align*}
    C\|\nabla_x(n ,u,T )\|_{L^{\infty}}\Big(\nm{\fe}_{H^1}^2+\nm{E_R^{\e}}_{H^1}^2\Big)\lesssim (1+t)^{-\beta_0}\Big(\nm{\fe}_{H^1}^2+\nm{E_R^{\e}}_{H^1}^2\Big).
\end{align*}

\textbf{Third Term on the R.H.S. of \eqref{H10}:}
For this term, its upper bound is
\begin{align*}
\Big(\nm{\nabla_xE }_{L^{\infty}}+\nm{\nabla_xB ]}_{L^{\infty}}\Big)\nm{\partial_x\fe}\nms{\fe}\lesssim (1+t)^{-\beta_0}\overline{\e}_0\Big(\nm{\fe}^2_{H^1}+\nms{(\ik-\pk)[\fe]}^2\Big).
\end{align*}

\textbf{Fourth Term on the R.H.S. of \eqref{H10}:}
Using  $w^0\gs (p^0)^3$, and noticing that for $p^0\gs\e^{-1}$, we have $\e (w^1)^2\gs \e(p^0)^2\gs p^0$. Then,  for $\kappa$ sufficiently small, we have
\begin{align*}
&\abs{\br{\partial_x\bigg( \m ^{-\frac{1}{2}}\fe\Big\{\dt +\hat{p}\cdot\nabla_x- \big(E +\hat{p} \times B  \big)\cdot\nabla_p \Big\}\left[\mh \right]\bigg),\frac{4\pi T }{u^0 }\partial_x\fe  }}\\
 \lesssim& \zzz\br{p^0\partial_x\fe,\partial_x\fe}+\zzz\big|\br{(p^0)^2\fe,\partial_x\fe}\big|\lesssim \zzz\tnm{\sqrt{p^0}\partial_x\fe}^2+\zzz\tnm{\sqrt{p^0}p^0\fe}^2\no\\
\ls&\zzz\bigg(\nm{\sqrt{p^0}\partial_x\pk[\fe]}^2+\int_{\r^3}\int_{p^0\leq\e^{-1}\kappa}p^0\abs{\partial_x(\ik-\pk)[\fe]}^2+\int_{\r^3}\int_{p^0\geq\e^{-1}\kappa}p^0\abs{\partial_x(\ik-\pk)[\fe]}^2\bigg)\no\\
    &+\zzz\bigg(\nm{\sqrt{p^0}p^0\pk[\fe]}^2+\int_{\r^3}\int_{\r^3}(p^0)^3\abs{(\ik-\pk)[\fe]}^2\bigg)\no\\
    \ls&\zzz\tnm{\partial_x\fe}^2+o(1)\e^{-1}\nms{\partial_x(\ik-\pk)[\fe]}^2+\e\tnmww{\partial_x(\ik-\pk)[\fe]}^2+\zzz\tnm{\fe}^2+\zzz\nmsw{(\ik-\pk)[\fe]}^2\no\\
    \ls&o(1)\e^{-1}\nms{\partial_x(\ik-\pk)[\fe]}^2+\e\tnmww{\nabla_x(\ik-\pk)[\fe]}^2+\overline{\e}_0\nmsw{(\ik-\pk)[\fe]}^2+(1+t)^{-\beta_0}\tnm{\fe}^2_{H^1}.\no
\end{align*}

\textbf{Fifth and Sixth Terms on the R.H.S. of \eqref{H10}:}
They can be bounded by
\begin{align*}
 C\zzz\bigg(\nm{\partial_x\fe}^2+\Big(\nm{E_R^{\e}}_{H^1}+\nm{B_R^{\e}}_{H^1}\Big)\nm{\partial_x\fe}\bigg)   \leq C(1+t)^{-\beta_0}\big(\nm{\partial_x\fe}^2+\nm{E_R^{\e}}_{H^1}^2+\nm{B_R^{\e}}^2_{H^1}\big).
\end{align*}

\textbf{Seventh Term on the R.H.S. of \eqref{H10}:} Using Lemma \ref{ss 02}  for $p$ integral, $(\infty,2,2)$ or $(4,4,2)$ for $x$ integral, and Sobolev embedding, we have 
\begin{align*}
    &\abs{\e^{k-1}\br{\partial_x\Gamma[\fe,\fe],\frac{4\pi T }{u^0 }\partial_x\fe}}\\
    \ls&\e^{k-1}\int_{x\in\r^3}\Big[\Big(\tbs{\fe}\abss{\nabla_x\fe}+\abss{\fe}\tbs{\nabla_x\fe}\Big)\abss{(\ik-\pk)[\nabla_x\fe]}+\zzz\tbs{\fe}\abss{\fe}\abss{\nx\fe}\Big]\no\\
    \ls&\e^{k-1}\nm{\fe}_{H^2}\nm{\fe}_{H^1_{\sigma}}\nms{(\ik-\pk)[\nabla_x\fe]}+\zzz\e^{k-1}\nm{\fe}_{H^2}\nms{\fe}\nms{\nx\fe}\no\\
    \ls&\e^{\frac{1}{2}}\nm{\fe}_{H^1_{\sigma}}\nms{(\ik-\pk)[\nabla_x\fe]}+\e^{\frac{1}{2}}\zzz\nms{\fe}\nms{\nx\fe}.\no
\end{align*}
Noting 
\begin{align*}
    \nm{\fe}_{H^1_{\sigma}}\leq& \nm{(\ik-\pk)[\fe]}_{H^1_{\sigma}}+\nm{\pk[\fe]}_{H^1_{\sigma}}\ls \nm{(\ik-\pk)[\fe]}_{H^1_{\sigma}}+\nm{\fe}_{H^1},
\end{align*}
 we use \eqref{xp} to have
\begin{align*}
    &\abs{\e^{k-1}\br{\partial_x\Gamma[\fe,\fe],\frac{4\pi T }{u^0 }\partial_x\fe}}\\
    \ls&\e^{\frac{1}{2}}\Big(\nm{(\ik-\pk)[\fe]}_{H^1_{\sigma}}+\nm{\fe}_{H^1}\Big)\Big(\nms{\partial_x(\ik-\pk)[\fe]}+\zzz\tnm{\fe}\Big)\no\\
    &+\e^{\frac{1}{2}}\zzz\Big(\nms{(\ik-\pk)[\fe]}+\tnm{\fe}\Big)\Big(\nms{(\ik-\pk)[\nx\fe]}+\tnm{\nx\fe}\Big)\no\\
    \ls&\nm{(\ik-\pk)[\fe]}^2_{H^1_{\sigma}}+(1+t)^{-\beta_0}\overline{\e}_0\Big(\nm{\fe}^2+\e\nm{\nabla_x\fe}^2\Big).\no
\end{align*}

\textbf{Eighth Term on the R.H.S. of \eqref{H10}:}
 Similarly, considering that $F_i$ decay fast in $p$, we know
\begin{align*}
  &\abs{\sum_{i=1}^{2k-1}\e^{i-1} \br{\partial_x\Gamma\left[\mhh F_i,\fe\right]
    +\partial_x\Gamma\left[\fe,\mhh F_i\right],\frac{4\pi T }{u^0 }\partial_x\fe}}\\
    \ls& o(1)\e^{-1}\nms{(\ik-\pk)[\partial_x\fe]}^2+\e (1+t)^2\nm{\fe}_{H^1_{\sigma}}^2\no\\
    &+\zzz(1+t)\Big(\nm{\fe}+\nms{(\ik-\pk)[\fe]}\Big)\Big(\nm{\fe}_{H^1}+\nms{\partial_x(\ik-\pk)[\fe]}\Big)\no\\
    \ls&o(1)\e^{-1}\nms{\partial_x(\ik-\pk)[\fe]}^2+\overline{\e}_0\Big[\e^{-1}\nms{(\ik-\pk)[\fe]}^2+\e^{\frac{1}{3}}\nm{\fe}_{H^1}^2+\e^{-\frac{1}{3}}\nm{\fe}^2\Big].\no
\end{align*}

\textbf{Other Terms on the R.H.S. of \eqref{H10}:} Similar to the corresponding estimates in Proposition \ref{l21},
these terms can be conrolled by
\begin{align*}
&o(1)\e^{-1}\nm{(\ik-\pk)[\fe]}^2_{H^1_{\sigma}}+C\e^{-1}(1+t)^{-\beta_0}\nm{\fe}^2+\e^{\frac{2}{3}}\Big(\nm{E_R^{\e}}_{H^1}^2+\nm{B_R^{\e}}_{H^1}^2+\nm{\nabla_x\fe}^2\Big)\\
&+\e^{2k+1}(1+t)^{4k+2}
+\e^{k}(1+t)^{2k}\nm{\nabla_x\fe}.\no
\end{align*}
\textbf{Summary:}
By collecting all the above estimates in \eqref{H10}, we multiply the resulting inequality by $\e$ to derive \eqref{H1ener01}.
\end{proof}

\subsection{Second-Order Derivatives Estimates}\label{41}

In this subsection, we proceed to the $L^2$ estimate of $\nabla_x^2\Big(\fe, E^{\e}_ R, B^{\e}_R\Big)$.

\begin{proposition}\label{H2x}
For the remainders $(\fe, E^{\e}_ R, B^{\e}_ R)$, it holds that
\begin{align}\label{H2ener01}
&\e^2\frac{\ud}{\ud t}\bigg(\nm{\sqrt{\frac{4\pi T }{u^0}}\nabla_x^2\fe}^2+\nm{\nabla_x^2E^{\e}_ R}^2+\nm{\nabla_x^2B^{\e}_ R]}^2\bigg)+\frac{\delta}{\e}\nms{\nabla_x^2(\ik-\pk)[\fe]}^2\Big]\\
\lesssim & \Big[(1+t)^{-\beta_0}+\e^{\frac{1}{3}}\Big]\mathcal{E}+\overline{\e}_0\nm{(\ik-\pk)[\fe]}^2_{H^1_{\sigma}}+\e\mathcal{D} +\e^{2k+3}(1+t)^{4k+2}
+\e^{k+2}(1+t)^{2k}\sqrt{\mathcal{E}}.\no
\end{align}
\end{proposition}

\begin{proof} 
From \eqref{L201}, we have
\begin{align}\label{H2max}
\frac{1}{2}\frac{\ud}{\ud t}\Big(\nm{\partial_x^2E_R^{\e}}^2+\nm{\partial_x^2 B_R^{\e}}^2\Big)=4\pi\br{\partial_x^2\Big( \hat{p}\mh \fe\Big), \partial_x^2E_R^{\e}}.
\end{align}
Here and below we  denote $\partial_x^2$ as $\partial_{x_i}\partial_{x_j}$ with $i,j=1,2,3$ for convenience.
Using a similar argument as in Proposition \ref{H1x}, we have
\begin{align}\label{xx1}
  \e^{-1}  \abs{\br{\jump{\li,\partial_x^2}[\fe],\frac{4\pi T }{u^0 }\partial_x^2\fe}}\leq& \e^{-1}\abs{\br{\li\Big[\jump{\pk,\partial_x^2}[\fe]\Big],\frac{4\pi T }{u^0 }\partial_x^2\fe}}\\
    &+\e^{-1}\abs{\br{\jump{\li,\partial_x^2}\big[(\ik-\pk)[\fe]\big],\frac{4\pi T }{u^0 }\partial_x^2\fe}}.\no
\end{align}
For the first term in \eqref{xx1}, we have
\begin{align*}
    \e^{-1}\abs{\br{\li\Big[\jump{\pk,\partial_x^2}[\fe]\Big],\frac{4\pi T }{u^0 }\partial_x^2\fe}}=\e^{-1}\abs{\br{\li\Big[\jump{\pk,\partial_x^2}[\fe]\Big],(\ik-\pk)[\partial_x^2\fe]}}.
\end{align*}
Note that $\jump{\pk,\partial_x^2}$ only contains terms that $\partial_x$ hits $\fe$ at most once. Hence, we have
\begin{align*}
    \e^{-1}\abs{\br{\li\Big[\jump{\pk,\partial_x^2}[\fe]\Big],\frac{4\pi T }{u^0 }\partial_x^2\fe}}&\ls o(1)\e^{-1}\nms{(\ik-\pk)[\partial_x^2\fe]}^2+\e^{-1}\zzz\nm{\fe}_{H^1}^2\\
    &\ls o(1)\e^{-1}\nms{\partial_x^2(\ik-\pk)[\fe]}^2+\e^{-1}\overline{\e}_0(1+t)^{-\beta_0}\nm{\fe}_{H^1}^2.\no
\end{align*}
For the second term in \eqref{xx1}, since $\jump{\li,\partial_x^2}$ indicates that $\partial_x$ at most hits $(\ik-\pk)[\fe]$ once, we directly bound it as follows
\begin{align*}
    &\e^{-1}\abs{\br{\jump{\li,\partial_x^2}\big[(\ik-\pk)[\fe]\big],\partial_x^2\fe}}\\
    \ls& \e^{-1}\zzz\big(\abss{(\ik-\pk)[\fe]}\abss{\partial_x^2\fe}+\abss{\partial_x(\ik-\pk)[\fe]}\abss{\partial_x^2\fe}\big)\no\\
    \ls&o(1)\e^{-1}\nms{\partial_x^2(\ik-\pk)[\fe]}^2+(1+t)^{-\beta_0}\tnm{\nabla_x^2\fe}^2+\e^{-2}\overline{\e}_0\nm{(\ik-\pk)[\fe]}^2_{H^1_{\sigma}}.\no
\end{align*}
In total, we have 
\begin{align}\label{xx11}
   &    \e^{-1}\br{ \partial_x^2\li\left[\fe\right],\frac{4\pi T }{u^0 }\partial_x^2\fe}\\
   \geq&    \e^{-1}\br{ \li\left[\partial_x^2\fe\right],\frac{4\pi T }{u^0 }\partial_x^2\fe}-     \e^{-1}\abs{\br{\jump{\li,\partial_x^2}[\fe],\frac{4\pi T }{u^0 }\partial_x^2\fe}}\no\\
    \geq&\delta\e^{-1}\nms{\partial_x^2(\ik-\pk)[\fe]}^2-\e^{-2}\overline{\e}_0\nm{(\ik-\pk)[\fe]}^2_{H^1_{\sigma}}-(1+t)^{-\beta_0}\big(\tnm{\nabla_x^2\fe}^2+\e^{-1}\nm{\fe}_{H^1}^2\big).\no
\end{align}

Taking $|\alpha|=2$ in \eqref{alpha}, we multiply the equation by
$\frac{4\pi T }{u^0}\partial_x^2\fe $, and further use \eqref{H2max}, \eqref{xx11} to  obtain
\begin{align}\label{2x0}
    & \frac{1}{2}\frac{\ud}{\ud t}\bigg(\nm{\sqrt{\frac{4\pi T }{u^0}}\partial_x^2\fe}^2+\nm{\partial_x^2E_R^{\e}}^2+\nm{\partial_x^2B_R^{\e}}^2\bigg)
    +\frac{\delta}{\e}\nms{\partial_x^2(\ik-\pk)[\fe]}^2\\
    \leq& \abs{\sum_{|\alpha|<2}\br{\partial_x^{2-|\alpha|}\bigg(\frac{u^0}{T }\hat{p}\mh \bigg)\cdot \partial_x^{\alpha}E_R^{\e}, \frac{4\pi T }{u^0 }\partial_x^2\fe }}\no\\
    &+4\pi\abs{\sum_{|\alpha|<2}\br{  \hat{p}\big(\partial_x^{2-|\alpha|}\mh \big)\partial_x^{\alpha}\fe, \partial_x^2E_R^{\e}}}\no\\
    &+\abs{\br{\partial_x^2\bigg(\big(E +\hat{p} \times B  \big)\cdot\nabla_p\fe \bigg),\frac{4\pi T }{u^0 }\partial_x^2\fe }}\no\\
    &+\abs{\br{\partial_x^2\bigg( \m ^{-\frac{1}{2}}\fe \Big\{\dt +\hat{p}\cdot\nabla_x- \big(E +\hat{p} \times B  \big)\cdot\nabla_p \Big\}\left[\mh \right]\bigg),\frac{4\pi T }{u^0 }\partial_x^2\fe  }}\no\\
    &+\frac{1}{2}\abs{\br{\big(\dt +\hat{p}\cdot\nabla_x\big)\left[\frac{T }{u^0}\right],4\pi \abs{\partial_x^2\fe }^2}}\no\\
    &+ \abs{\br{\partial_x^2 \bigg(u  \mh \cdot\big(\hat{p} \times B_R^{\e} \big)\bigg), \frac{4\pi T }{u^0 }\partial_x^2 \fe }}+\e^{k-1}\abs{\br{\partial_x^2\Gamma\left[\fe,\fe\right],\frac{4\pi T }{u^0}\partial_x^2\fe}}\no\\
    &+\abs{\sum_{i=1}^{2k-1}\e^{i-1}\br{\partial_x^2\Gamma\left[ \m ^{-\frac{1}{2}}F_i, \fe\right]+\partial_x^2\Gamma\left[\fe,  \m ^{-\frac{1}{2}}F_i\right],\frac{4\pi T }{u^0}\partial_x^2\fe}}\no\\
    &+\e^k\abs{\br{\partial_x^2\bigg(\frac{u^0\hat{p}-u }{2 T }\cdot\big(E_R^{\e}+\hat{p} \times B_R^{\e} \big)\fe \bigg),\frac{4\pi T }{u^0 }\partial_x^2\fe  }}\no\\
    &+\e^k\abs{\br{\partial_x^2\bigg(\big(E_R^{\e}+\hat{p} \times B_R^{\e} \big)\cdot\nabla_p\fe \bigg),\frac{4\pi T }{u^0 }\partial_x^2\fe  }}\no\\
    &+\abs{\sum_{i=1}^{2k-1}\e^i\br{\partial_x^2\bigg(\big(E_i+\hat{p} \times B_i \big)\cdot\nabla_p\fe +\big(E_R^{\e}+\hat{p} \times B_R^{\e} \big)\cdot \m ^{-\frac{1}{2}}\nabla_pF_{i}\bigg),\frac{4\pi T }{u^0 }\partial_x^2\fe}}\no\\
    &+\abs{\sum_{i=1}^{2k-1}\e^i\br{\partial_x^2\bigg(\big(E_i+\hat{p} \times B_i \big)\cdot\frac{u^0\hat{p}-u }{2 T }\fe \bigg),\frac{4\pi T }{u^0 }\partial_x^2\fe  }}\no\\
    &+ \e^{k}\abs{\br{ \partial_x^2\sb ,\frac{4\pi T }{u^0}\partial_x^2\fe}}+C\e^{-2}\overline{\e}_0\nm{(\ik-\pk)[\fe]}^2_{H^1_{\sigma}}\no\\
    &+C(1+t)^{-\beta_0}\big(\tnm{\nabla_x^2\fe}^2+\e^{-1}\nm{\fe}_{H^1}^2\big).
\no
\end{align}
Now we estimate each term on the R.H.S. of \eqref{2x0}. 

\textbf{First Two Terms on the R.H.S. of \eqref{2x0}:}
The two terms can be bounded by
\begin{align*}
    C\zzz\Big(\nm{\fe}_{H^2}^2+\nm{E_R^{\e}}_{H^2}^2\Big)\lesssim (1+t)^{-\beta_0}\Big(\nm{\fe}_{H^2}^2+\nm{E_R^{\e}}_{H^2}^2\Big).
\end{align*}

\textbf{Third Term on the R.H.S. of \eqref{2x0}:}
We can obtain
\begin{align*}
&\abs{\br{\partial_x^2\bigg(\big(E +\hat{p} \times B  \big)\cdot\nabla_p\fe \bigg),\frac{4\pi T }{u^0 }\partial_x^2\fe }}\\
\lesssim& \zzz\nm{\partial_x^2\fe}\nm{\nabla_x\fe}_{H^1_{\sigma}}\lesssim (1+t)^{-\beta_0}\overline{\e}_0\Big(\nm{\partial_x^2\fe}^2+\nm{\nabla_x(\ik-\pk)[\fe]}^2_{H^1_{\sigma}}+\nm{\fe}^2_{H^2}\Big).\no
\end{align*}

\textbf{Fourth Term on the R.H.S. of \eqref{2x0}:} Noting that $\e (w^2)^2\gs\e (p^0)^2\gs p^0$ for $p^0\gs\e^{-1}$, we use 
similar arguments in \eqref{L200} to have
\begin{align*}
&\abs{\br{\partial_x^2\bigg( \m ^{-\frac{1}{2}}\fe \Big\{\dt +\hat{p}\cdot\nabla_x- \big(E +\hat{p} \times B  \big)\cdot\nabla_p \Big\}\left[\mh \right]\bigg),\frac{4\pi T }{u^0 }\partial_x^2\fe  }}\\
\ls&\zzz \br{(p^0)^3|\fe|+(p^0)^2|\partial_x\fe|+ p^0|\partial_x^2\fe|,|\partial_x^2\fe|}\no\\
    \ls& \zzz\Big(\tnm{\sqrt{p^0}\partial_x^2(\ik-\pk)[\fe]}^2+\tnm{\sqrt{p^0}\partial_x^2\pk[\fe]}^2\Big)+\zzz\nmsw{\fe}^2+\zzz\nmsww{\partial_x\fe}^2\no\\
    \ls&o(1)\e^{-1}\nms{\partial_x^2(\ik-\pk)[\fe]}^2+\e\zzz\nmwww{\partial_x^2(\ik-\pk)[\fe]}^2+\zzz\tnm{\partial_x^2\fe}^2\no\\
    &+\zzz\nms{w_0(\ik-\pk)[\fe]}^2+\zzz\nms{w_1\partial_x(\ik-\pk)[\fe]}^2+\zzz\tnm{\fe}^2_{H^1}\no\\
    \ls&o(1)\e^{-1}\nms{\partial_x^2(\ik-\pk)[\fe]}^2+\e\overline{\e}_0\nmwww{\nabla_x^2(\ik-\pk)[\fe]}^2+(1+t)^{-\beta_0}\tnm{\nabla_x^2\fe}^2\no\\
    &+\overline{\e}_0\Big(\nms{w_0(\ik-\pk)[\fe]}^2+\nms{w_1\nabla_x(\ik-\pk)[\fe]}^2\Big)+(1+t)^{-\beta_0}\tnm{\fe}^2_{H^1}.\no
\end{align*}

\textbf{Fifth and Sixth Terms on the R.H.S. of \eqref{2x0}:}
The upper bound of the two terms is
\begin{align*}
  C\zzz \Big(\nm{\partial_x^2\fe}^2+\nm{E_R^{\e}}_{H^2}^2+\nm{B_R^{\e}}_{H^2}^2\Big) \ls (1+t)^{-\beta_0}\Big(\nm{\partial_x^2\fe}^2+\nm{E_R^{\e}}^2_{H^2}+\nm{B_R^{\e}}_{H^2}^2\Big).
\end{align*}

\textbf{Seventh Term on the R.H.S. of \eqref{2x0}:}
Using Lemma \ref{ss 02} for $p$ integral, $(\infty,2,2), (2,\infty,2)$ or $(4,4,2)$ for $x$ integral, and Sobolev embedding, we have 
\begin{align*}
    &\abs{\e^{k-1}\br{\partial_x^2\Gamma[\fe,\fe],\frac{4\pi T }{u^0}\partial_x^2\fe}}\\
    \ls&\abs{\e^{k-1}\br{\Gamma[\partial_x^2\fe,\fe]+\Gamma[\fe,\partial_x^2\fe]+\Gamma[\partial_x\fe,\partial_x\fe],\frac{4\pi T }{u^0}(\ik-\pk)[\partial_x^2\fe]}}\no\\
    &+\zzz\abs{\e^{k-1}\br{\Gamma[\partial_x\fe,\fe]+\Gamma[\fe,\partial_x\fe],\frac{4\pi T }{u^0}\partial_x^2\fe}}+\zzz\abs{\e^{k-1}\br{\Gamma[\fe,\fe],\frac{4\pi T }{u^0}\partial_x^2\fe}}\no\\
    \ls&\e^{k-1}\int_{x\in\r^3}\Big[\Big(\abss{\partial_x^2\fe}\tbs{\fe}+\tbs{\partial_x^2\fe}\abss{\fe}+\abss{\partial_x\fe}\tbs{\partial_x\fe}\Big)\abss{(\ik-\pk)[\partial_x^2\fe]}\no\\
    &+\zzz\Big(\abss{\partial_x\fe}\tbs{\fe}+\tbs{\partial_x\fe}\abss{\fe}+\abss{\fe}\tbs{\fe}\Big)\abss{\partial_x^2\fe}\Big]\no\\
    \ls&\e^{k-1}\nm{\fe}_{H^2}\nm{\fe}_{H^2_{\sigma}}\nms{(\ik-\pk)[\nabla_x^2\fe]}+\e^{k-1}\zzz\nm{\fe}_{H^2}\nm{\fe}_{H^1_{\sigma}}\nms{\partial_x^2\fe}\no\\
    \ls&\e^{\frac{1}{2}}\nm{\fe}_{H^2_{\sigma}}\Big(\nms{\partial_x^2(\ik-\pk)[\fe]}+\zzz\nm{\fe}_{H^1}\Big)+\e^{\frac{1}{2}}\zzz\nm{\fe}_{H^1_{\sigma}}\nms{\partial_x^2\fe}\no\\
    \ls& \overline{\e}_0\nm{(\ik-\pk)[\fe]}^2_{H^2_{\sigma}}+\Big[(1+t)^{-\beta_0}+\e\Big]\nm{\fe}^2_{H^2}.\no
\end{align*}

\textbf{Eighth Term on the R.H.S. of \eqref{2x0}:}
Similar to the seventh term,  it can be bounded by 
\begin{align*}
     & o(1)\e^{-1}\nms{(\ik-\pk)[\partial_x^2\fe]}^2+\e (1+t)^2\nm{\fe}_{H^2_{\sigma}}^2\no\\
    &+\zzz(1+t)\Big(\nm{\fe}_{H^1}+\nm{(\ik-\pk)[\fe]}_{H^1_{\sigma}}\Big)\Big(\nm{\fe}_{H^2}+\nms{\partial_x^2(\ik-\pk)[\fe]}\Big)\no\\
    \ls&o(1)\e^{-1}\nms{\partial_x^2(\ik-\pk)[\fe]}^2+\overline{\e}_0\e^{-1}\nm{(\ik-\pk)[\fe]}_{H^1_{\sigma}}^2+\e^{\frac{1}{3}}\nm{\fe}_{H^2}^2+\overline{\e}_0\e^{-\frac{1}{3}}\nm{\fe}_{H^1}^2.\no
\end{align*}

\textbf{Other Terms on the R.H.S. of \eqref{2x0}:} As in Proposition \ref{H1x}, we have the following upper bound of
these terms:
\begin{align*}
&o(1)\e^{-1}\nm{\partial_x(\ik-\pk)[\fe]}^2_{H^1_{\sigma}}+C\e^{-1}(1+t)^{-\beta_0}\nm{\fe}^2_{H^1}+C\e^{\frac{2}{3}}\Big(\nm{\nabla_x^2\fe}^2+\nm{E_R^{\e}}^2_{H^2}+\nm{B_R^{\e}}^2_{H^2}\Big)\\
&+C\e^{2k+1}(1+t)^{4k+2}
+C\e^{k}(1+t)^{2k}\nm{\nabla_x^2\fe}.
\end{align*}

\textbf{Summary:}
We collect the above estimates in \eqref{2x0}, and multiply the resulting inequality by $\e^2$ to derive \eqref{H2ener01}.
\end{proof}


\section{Weighted Energy Estimates} \label{l22}
In  this section, we are devoted to the weighted energy estimates of the remainder term $\fe$.

\subsection{Weighted Estimate}


Apply microscopic projection $(\ik-\pk )$ onto \eqref{L20} to have
\begin{align}\label{wL20}
    &\left\{\dt +\hat{p}\cdot\nabla_x -\big(E +\hat{p} \times B  \big)\cdot\nabla_p \right\}(\ik-\pk )\left[\fe \right] \\
    &- (\ik-\pk )\Big[\big(E_R^{\e}+\hat{p} \times B_R^{\e} \big)\cdot\frac{-u^0\hat{p}+u   }{ T }\mh\Big]+\frac{\li [\fe ]}{\e}\no\\
    =&- \m ^{-\frac{1}{2}}(\ik-\pk )\left[\fe \right]\left\{\dt +\hat{p}\cdot\nabla_x- \big(E +\hat{p} \times B  \big)\cdot\nabla_p\right\}\mh +\e^{k-1}\Gamma \left[\fe,\fe\right]\no\\
    &+\sum_{i=1}^{2k-1}\e^{i-1}\left(\Gamma \left[ \m ^{-\frac{1}{2}}F_i, \fe\right]+\Gamma \left[\fe,  \m ^{-\frac{1}{2}}F_i\right]\right)+ \e^k\big(E_R^{\e}+\hat{p} \times B_R^{\e} \big)\cdot\nabla_p(\ik-\pk )\left[\fe \right]\no\\
    &- \e^k\frac{1 }{2 T }\big(u^0\hat{p}-u \big)\cdot\big(E_R^{\e}+\hat{p} \times B_R^{\e} \big)(\ik-\pk )\left[\fe \right]+ \sum_{i=1}^{2k-1}\e^i\big(E_i+\hat{p} \times B_i \big)\cdot\nabla_p(\ik-\pk )\left[\fe \right]\no\\
    &+ \sum_{i=1}^{2k-1}\e^i(\ik-\pk )\Big(\big(E_R^{\e}+\hat{p} \times B_R^{\e} \big)\cdot \m ^{-\frac{1}{2}}\nabla_pF_{i}\Big)\Big]\no\\
    &- \sum_{i=1}^{2k-1}\e^i\Big[\big(E_i+\hat{p} \times B_i \big)\cdot\frac{1}{2 T }\big(u^0\hat{p}-u \big)(\ik-\pk )\left[\fe \right]\Big]+\e^{k}(\ik-\pk)[\sb]+\jump{{\bf P},\tau_{B}}\fe ,\no
\end{align}
where $\jump{{\bf P},\tau_{B}}={\bf P}\tau_{B}-\tau_{B}{\bf P}$ denotes the commutator of two operators ${\bf P}$ and $\tau_{B}$:
\begin{align*}
\tau_{B}:=&\dt +\hat{p}\cdot\nabla_x- \big(E +\hat{p} \times B  \big)\cdot\nabla_p\\
&+ \m ^{-\frac{1}{2}}\left\{\dt +\hat{p}\cdot\nabla_x- \big(E +\hat{p} \times B  \big)\cdot\nabla_p \right\}\mh \no\\
&+ \e^k\frac{1 }{2 T }\big(u^0\hat{p}-u \big)\cdot\big(E_R^{\e}+\hat{p} \times B_R^{\e} \big)-\big(E_R^{\e}+\hat{p} \times B_R^{\e} \big)\cdot\nabla_p\no\\
&-\sum_{i=1}^{2k-1}\e^i\left\{\big(E_i+\hat{p} \times B_i \big)\cdot\Big[\nabla_p-\frac{1}{2 T }\big(u^0\hat{p}-u \big)\right\}.\no
\end{align*}

\begin{proposition}\label{wL2ener}
For the remainders $\Big(\fe, E^{\e}_ R, B^{\e}_ R\Big)$, it holds that
\begin{align}\label{wL2f01}
&\frac{\ud}{\ud t}\nmw{(\ik-\pk)[\fe]}^2+\frac{\delta}{\e}\nmsw{(\ik-\pk)[\fe]}^2+Y\nmw{\sqrt{p^0}(\ik-\pk)[\fe]}^2\\
\lesssim&\frac{1}{\e}\nms{(\ik-\pk)[\fe]}^2+\e\nm{\nabla_x\fe}^2+\e^{\frac{1}{3}}\tnm{\fe}^2 +\e\big(\mathcal{E}+\mathcal{D}\big)+\e^{2k+1}(1+t)^{4k+2}.
\no
\end{align}
\end{proposition}

\begin{proof}
Noting
\begin{align}\label{w0f}
(w^0)^2(\ik-\pk )[\fe ]\cdot \partial_t\{ (\ik-\pk )[\fe ]\}&=\frac{1}{2}\partial_t|w^0(\ik-\pk )[\fe ]|^2+p^0Y(w^0)^2|(\ik-\pk )[\fe ]|^2,
\end{align}
we take the $L^2$ inner product of \eqref{wL20} with  $(w^0)^2(\ik-\pk )[\fe ]$ and use Lemma \ref{ss 04}
 to have
\begin{align}\label{wL2f1}
    & \frac{1}{2}\frac{\ud}{\ud t}\nmw{(\ik-\pk)[\fe]}^2+\frac{\delta}{\e}\nmsw{(\ik-\pk)[\fe]}^2+Y\nmw{\sqrt{p^0}(\ik-\pk)[\fe]}^2\\
    \leq&\frac{C}{\e}\nms{(\ik-\pk)[\fe]}^2+\abs{\br{\Big(E +\hat{p} \times B  \Big)\cdot\nabla_p(\ik-\pk )\left[\fe \right],(w^0)^2(\ik-\pk )\left[\fe \right]}}\no\\
    &+\abs{\br{(\ik-\pk )\Big[\Big(E_R^{\e}+\hat{p} \times B_R^{\e} \Big)\cdot\frac{-u^0\hat{p}+u  }{ T }\mh \Big],(w^0)^2(\ik-\pk )\left[\fe \right]}}\no\\
    &+\abs{\br{ \m ^{-\frac{1}{2}}\big\{\dt +\hat{p}\cdot\nabla_x- \big(E +\hat{p} \times B  \big)\cdot\nabla_p \big\}\left[\mh \right],\abs{w^0(\ik-\pk )\left[\fe \right]}^2}}\no\\
    &+\abs{\e^{k-1}\br{\Gamma\left[\fe,\fe\right],(w^0)^2(\ik-\pk)[\fe]}}\no\\
    &+\abs{\sum_{i=1}^{2k-1}\e^{i-1}\br{\Gamma\left[ \m ^{-\frac{1}{2}}F_i, \fe\right]+\Gamma\left[\fe,  \m ^{-\frac{1}{2}}F_i\right],(w^0)^2(\ik-\pk)[\fe]}}\no\\
    &+\abs{\e^k\br{\frac{\big(u^0\hat{p}-u \big)}{2T }\cdot\Big(E_R^{\e}+\hat{p} \times B_R^{\e} \Big),\abs{w^0(\ik-\pk )\left[\fe \right]}^2}}\no\\
    &+\abs{\br{\Big(E_R^{\e}+\hat{p} \times B_R^{\e} \Big)\cdot\nabla_p\fe ,(w^0)^2(\ik-\pk )\left[\fe \right]}}\no\\
    &+\abs{\sum_{i=1}^{2k-1}\e^i\br{\big(E_i+\hat{p} \times B_i \big)\cdot\nabla_p(\ik-\pk )\left[\fe \right],(w^0)^2(\ik-\pk )\left[\fe \right]}}\no\\
    &+\abs{\sum_{i=1}^{2k-1}\e^i\br{(\ik-\pk )\left[\big(E_R^{\e}+\hat{p} \times B_R^{\e} \big)\cdot \m ^{-\frac{1}{2}}\nabla_pF_{i}\right],(w^0)^2(\ik-\pk )\left[\fe \right]}}\no\\
    &+\abs{\sum_{i=1}^{2k-1}\e^i\br{ \left(\big(E_i+\hat{p} \times B_i \big)\cdot\frac{u^0\hat{p}-u }{2 T }\right),\abs{w^0(\ik-\pk )\left[\fe \right]}^2}}\no\\
    &+\abs{\e^{k}\br{(\ik-\pk)[\sb] ,(w^0)^2(\ik-\pk)[\fe]}}+
    \abs{\br{\jump{{\bf P},\tau_{B}}\left[\fe \right],(w^0)^2(\ik-\pk)\left[\fe \right]}}.
\no
\end{align}

\textbf{Second Term on the R.H.S. of \eqref{wL2f1}:}
It is bounded by
\begin{align*}
    C\Big(\nm{E }_{L^{\infty}}+\nm{B }_{L^{\infty}}\Big)\nmw{(\ik-\pk)[\fe]}^2\lesssim (1+t)^{-\beta_0}\nmw{(\ik-\pk)[\fe]}^2.
\end{align*}

\textbf{Third Term on the R.H.S. of \eqref{wL2f1}:}
Noting \eqref{smallw}, we bound it by
\begin{align*}
C\Big(\nm{E_R^{\e}}+\nm{B_R^{\e}}\Big)\|(\ik-\pk)[\fe]\|\lesssim \e^{-1}\|(\ik-\pk)[\fe]\|^2+\e\big(\nm{E_R^{\e}}^2+\nm{B_R^{\e}}^2\big).
\end{align*}

\textbf{Fourth Term on the R.H.S. of \eqref{wL2f1}:}
By the smallness of $\overline{\e}_0$, it holds that
\begin{align}\label{zzyy}
    \zzz\ls \overline{\e}_0(1+t)^{-\beta_0}\ll \yy.
\end{align}
Then, \eqref{exassump} holds true uniformly and we have
\begin{align}\label{e1}
&\abs{\br{ \m ^{-\frac{1}{2}}\big\{\dt +\hat{p}\cdot\nabla_x -\big(E +\hat{p} \times B  \big)\cdot\nabla_p \big\}\left[\mh \right],\abs{w^0(\ik-\pk )\left[\fe \right]}^2}}\\
\ls& \zzz\Big(\nmw{\sqrt{p^0}(\ik-\pk)[\fe]}^2+\nmsw{(\ik-\pk)[\fe]}^2\Big)\no\\
\leq& \frac{Y}{2}\nmw{\sqrt{p^0}(\ik-\pk)[\fe]}^2+C\overline{\e}_0\nmsw{(\ik-\pk)[\fe]}^2\no.
\end{align}

\textbf{Fifth Term on the R.H.S. of \eqref{wL2f1}:}
Using \eqref{Weight1} in Lemma \ref{ss 05} and Sobolev embedding, it can be controlled by
\begin{align*}
    &\e^{k-1}\Big(\nm{\fe}_{H^2}\nmsw{\fe}+\nm{\fe}_{H^2_{\sigma}}\nm{\fe}_{w^0}\Big)\nmsw{(\ik-\pk)[\fe]}\\
    \ls&  \Big\{\e^{\frac{1}{2}}\Big(\nmsw{(\ik-\pk)[\fe]}+\nmsw{\pk[\fe]}\Big)+\e\Big(\nm{(\ik-\pk)[\fe]}_{H^2_{\sigma}}+\nm{\pk[\fe]}_{H^2_{\sigma}}\Big)\Big\}\nmsw{(\ik-\pk)[\fe]}\no\\
    \ls&o(1)\e^{-1}\nmsw{(\ik-\pk)[\fe]}^2+\e^3\nm{(\ik-\pk)[\fe]}_{H^2_{\sigma}}^2+\e^3\nm{\fe}_{H^2}^2+\e^2\nm{\fe}^2.\no
\end{align*}

\textbf{Sixth Term on the R.H.S. of \eqref{wL2f1}:}
Similarly,  we use \eqref{smallw}, \eqref{Weight1} and \eqref{growth0} in Proposition \ref{fn} to obtain that for $t\leq \overline{t}=\e^{-1/3}$,
\begin{align*}
&\sum_{i=1}^{2k-1}\e^{i-1}\abs{\br{\Gamma\left[ \m ^{-\frac{1}{2}}F_i, \fe\right]+\Gamma\left[\fe,  \m ^{-\frac{1}{2}}F_i\right],(w^0)^2(\ik-\pk)[\fe]}}\\
\lesssim& \sum_{i=1}^{2k-1}\e^{i-1}\Big\|\big|w^0 \m ^{-\frac{1}{2}}F_i\big|_{\sigma}\Big\|_{L^{\infty}_x}\nmsw{\fe}\nmsw{(\ik-\pk)[\fe]}^2\no\\
\ls& o(1)\e^{-1}\nmsw{(\ik-\pk)[\fe]}^2+\e(1+t)^2\nmsw{\fe}^2\no\\
    \ls& o(1)\e^{-1}\nmsw{(\ik-\pk)[\fe]}^2+\e^{\frac{1}{3}}\tnm{\fe}^2.\no
\end{align*}

\textbf{Seventh and eighth Terms on the R.H.S. of \eqref{wL2f1}:}
Its upper bound is
\begin{align*}
    C\e^k\Big(\nm{E_R^{\e}}+\nm{B_R^{\e}}\Big)_{H^2}\nmsw{(\ik-\pk)[\fe]}^2\ls\e\nmsw{(\ik-\pk)[\fe]}^2.
\end{align*}

\textbf{Last Term on the R.H.S. of \eqref{wL2f1}:} For this term, we have
\begin{align*}
&\abs{\br{\jump{{\bf P},\tau_{B}}\left[\fe \right],(w^0)^2(\ik-\pk)\left[\fe \right]}}
\lesssim\e^{-1}\nms{(\ik-\pk)[\fe]}^2+ \e\Big(
\nm{E_R^{\e}}^2+\nm{B_R^{\e}}^2+\nm{\fe}^2_{H^1}\Big).
\end{align*}

\textbf{Other Terms on the R.H.S. of \eqref{wL2f1}:}
Similarly, these terms can be bounded by
\begin{align*}
& o(1)\e^{-1}\nmsw{(\ik-\pk)[\fe]}^2+C\e \Big(\nm{\fe}^2+\nm{E_R^{\e}}^2+\nm{B_R^{\e}}^2\Big)+ C\e^{2k+1}(1+t)^{4k+2}.
\end{align*}

\textbf{Summary:}
We collect the above estimates in \eqref{wL2f1} to derive \eqref{wL2f01}.
\end{proof}


\subsection{Weighted First-Order Derivatives Estimates}\label{31}

In this subsection, we proceed the weighted  $L^2$ estimate of $\nabla_x\fe$.

\begin{proposition}\label{H1p}
For the remainders $\Big(\fe, E^{\e}_ R, B^{\e}_R\Big)$, it holds that
\begin{align}\label{wH1f12}
& \e\left(\frac{\ud}{\ud t}\nmww{\nabla_x(\ik-\pk)[\fe]}^2+\frac{\delta}{\e}\nmsww{\nabla_x(\ik-\pk)[\fe]}^2+Y\nmww{\sqrt{p^0}\nabla_x(\ik-\pk)[\fe]}^2\right)\\
\ls&\Big((1+t)^{-\beta_0}+\e^{\frac{1}{3}}\Big)\mathcal{E}+\e^2\nm{\nabla_x^2\fe}^2+\nms{\nabla_x(\ik-\pk)[\fe]}^2+\e\mathcal{D}+C\e^{2k+2}(1+t)^{4k+2}.\no
\end{align}
\end{proposition}

\begin{proof}
Noting that $\jump{\li,\nabla_x}$ only contains terms that hit $\li$, it holds that
\begin{align*}
    &\e^{-1}\abs{\br{\jump{\li,\partial_x}\Big[(\ik-\pk)[\fe]\Big],w_1^2\partial_x(\ik-\pk)[\fe]}}\\
    \ls&\e^{-1}\zzz\nmsw{(\ik-\pk)[\fe]}\nmsww{\partial_x(\ik-\pk)[\fe]}\no\\
    \ls&o(1)\e^{-1}\nmsww{\partial_x(\ik-\pk)[\fe]}^2+\overline{\e}_0\e^{-1}\nmsw{(\ik-\pk)[\fe]}^2.\no
\end{align*}
Then, by Lemma \ref{ss 04} and \eqref{xp}, we have
\begin{align}\label{wxf}
  &\e^{-1}\br{ \partial_x\li\left[(\ik-\pk)[\fe]\right],w_1^2\partial_x(\ik-\pk)[\fe]}\\
   \geq&\e^{-1}\br{ \li\left[\partial_x(\ik-\pk)[\fe]\right],w_1^2\partial_x(\ik-\pk)[\fe]}-\e^{-1} \abs{\br{\jump{\li,\partial_x}(\ik-\pk)[\fe],w_1^2 \partial_x(\ik-\pk)[\fe]}}\no\\
   \geq
    &\delta \e^{-1}\nmsww{\partial_x(\ik-\pk)[\fe]}^2-C\e^{-1}\nms{\partial_x(\ik-\pk)[\fe]}^2-C\overline{\e}_0\e^{-1}\nmsw{(\ik-\pk)[\fe]}^2\no.
\end{align}
Apply $\partial_x$  to \eqref{wL20} and take the $L^2$ inner product of the resulting equation with $(w^1)^2\partial_x(\ik-\pk )[\fe ]$. Then, by similar arguments as in \eqref{w0f}, we use \eqref{wxf}  to have
\begin{align}\label{wH1f1}
    & \frac{1}{2}\frac{\ud}{\ud t}\nmww{\partial_x(\ik-\pk)[\fe]}^2+\frac{\delta}{\e}\nmsww{\partial_x(\ik-\pk)[\fe]}^2+Y\nmww{\sqrt{p^0}\partial_x(\ik-\pk)[\fe]}^2\\
    \leq&C\e^{-1}\Big(\nms{\partial_x(\ik-\pk)[\fe]}^2+\overline{\e}_0\nmsw{(\ik-\pk)[\fe]}^2\Big)\no\\
    &+\abs{\br{\partial_x\bigg(\big(E +\hat{p} \times B  \big)\cdot\nabla_p(\ik-\pk )\left[\fe \right]\bigg),(w^1)^2\partial_x(\ik-\pk )\left[\fe \right]}}\no\\
    &+\abs{\br{\partial_x(\ik-\pk)\Big[\bigg(\big(E_R^{\e}+\hat{p} \times B_R^{\e} \big)\cdot\frac{-u^0\hat{p}+u  }{ T }\mh \bigg)\Big],(w^1)^2\partial_x(\ik-\pk )\left[\fe \right]}}\no\\
    &+\abs{\br{\partial_x\bigg((\ik-\pk )\left[\fe \right] \m ^{-\frac{1}{2}}\big\{\dt +\hat{p}\cdot\nabla_x -\big(E +\hat{p} \times B  \big)\cdot\nabla_p\big\}\left[\mh \right]\bigg),(w^1)^2\partial_x(\ik-\pk )\left[\fe \right]}}\no\\
    &+\e^{k-1}\abs{\br{\partial_x\Gamma\left[\fe,\fe\right],(w^1)^2\partial_x(\ik-\pk)[\fe]}}\no\\
    &+\abs{\sum_{i=1}^{2k-1}\e^{i-1}\br{\partial_x\Gamma\left[ \m ^{-\frac{1}{2}}F_i, \fe\right]+\partial_x\Gamma\left[\fe,  \m ^{-\frac{1}{2}}F_i\right],(w^1)^2\partial_x(\ik-\pk)[\fe]}}\no\\
    &+\e^k\abs{\br{\partial_x\bigg(\frac{\big(u^0\hat{p}-u \big)}{2T }\cdot
    \big(E_R^{\e}+\hat{p} \times B_R^{\e} \big)(\ik-\pk )\left[\fe \right]\bigg),(w^1)^2\partial_x(\ik-\pk )\left[\fe \right]}}\no\\
    &+\e^k\abs{\br{\partial_x\bigg(
    \big(E_R^{\e}+\hat{p} \times B_R^{\e} \big)\cdot\nabla_p(\ik-\pk )\left[\fe \right]\bigg),(w^1)^2\partial_x(\ik-\pk )\left[\fe \right]}}\no\\
    &+\abs{\sum_{i=1}^{2k-1}\e^i\br{ \partial_x\bigg(\big(E_i+\hat{p} \times B_i \big)\cdot\nabla_p(\ik-\pk )\left[\fe \right]\bigg),(w^1)^2\partial_x(\ik-\pk )\left[\fe \right]}}\no\\
    &+\abs{\sum_{i=1}^{2k-1}\e^i\br{ \partial_x\bigg((\ik-\pk )\Big[\big(E_R^{\e}+\hat{p} \times B_R^{\e} \big)\cdot \m ^{-\frac{1}{2}}\nabla_pF_{i}\Big]\bigg),(w^1)^2\partial_x(\ik-\pk )\left[\fe \right]}}\no\\
    &+\abs{\sum_{i=1}^{2k-1}\e^i\br{ \partial_x\bigg(\big(E_i+\hat{p} \times B_i \big)\cdot\frac{u^0\hat{p}-u }{2 T }(\ik-\pk )\left[\fe \right]\bigg),(w^1)^2\partial_x(\ik-\pk )\left[\fe \right]}}\no\\
    &+ \e^{k}\abs{\br{\partial_x(\ik-\pk)[\sb] ,(w^1)^2\partial_x(\ik-\pk)[\fe]}}+ \abs{\br{\partial_x\jump{{\bf P},\tau_{B}}\fe,(w^1)^2\partial_x(\ik-\pk)[\fe]}}.
\no
\end{align}
Now we estimate each term on the R.H.S. of \eqref{wH1f1}.

\textbf{Second Term on the R.H.S. of \eqref{wH1f1}:}
It can be bounded by
\begin{align*}
&C \Big(\nm{\nabla_xE }_{W^{1,\infty}}+\nm{\nabla_xB }_{W^{1,\infty}}\Big)\Big(\nmww{\partial_x(\ik-\pk)[\fe]}^2+\nmsw{(\ik-\pk)[\fe]}^2\Big)\\
\lesssim& (1+t)^{-\beta_0}\overline{\e}_0\Big(\nmww{\partial_x(\ik-\pk)[\fe]}^2+\nmsw{(\ik-\pk)[\fe]}^2\Big).\no
\end{align*}

\textbf{Third Term on the R.H.S. of \eqref{wH1f1}:}
It can be controlled by
\begin{align*}
   C\e^{-1}\nms{\partial_x(\ik-\pk)[\fe]}^2+C\e
    \Big(\nm{E_R^{\e}}^2_{H^1}+\nm{B_R^{\e}}^2_{H^1}\Big).
\end{align*}

\textbf{Fourth Term on the R.H.S. of \eqref{wH1f1}:}
Similar to \eqref{e1}, we bound it by
\begin{align*}
&C\zzz\nmww{\sqrt{p^0}\partial_x(\ik-\pk)[\fe]}^2+C \zzz\nmw{(\ik-\pk)[\fe]}^2\\
\leq &\frac{Y}{2}\nmww{\sqrt{p^0}\partial_x(\ik-\pk)[\fe]}^2+(1+t)^{-\beta_0}\overline{\e}_0\nmsw{(\ik-\pk)[\fe]}^2.\no
\end{align*}

\textbf{Fifth Term on the R.H.S. of \eqref{wH1f1}:}
Using Lemma \ref{ss 05} for $p$ integral, $(\infty,2,2)$, $(2,\infty,2)$ or $(4,4,2)$ for $x$ integral, Sobolev embedding and \eqref{smallw}, we have 
\begin{align*}
    &\abs{\br{\e^{k-1}\partial_x\Gamma[\fe,\fe],(w^1)^2\partial_x(\ik-\pk)[\fe]}}\\
    \ls&\e^{k-1}\int_{x\in\r^3}\Big[\tbs{w^1\partial_x\fe}\abss{\fe}+\abss{w^1\partial_x\fe}\tbs{\fe}+\tbs{w_1\fe}\abss{\partial_x\fe}+\abss{w^1\fe}\tbs{\partial_x\fe}\no\\
    &+\zzz\big(\tbs{w^1\fe}\abss{\fe}+\abss{w^1\fe}\tbs{\fe}\big)\Big]\abss{w^1\partial_x(\ik-\pk)[\fe]}\no\\
    \ls&\e^{k-1}\Big(\nm{\fe}_{H^2}\nm{w^1\fe}_{H^1_{\sigma}}+\nm{\fe}_{H^2_{\sigma}}\nm{w^1\fe}_{H^1}\Big)\nmsww{\partial_x(\ik-\pk)[\fe]}\no\\
    \ls& \Big(\e^{\frac{1}{2}}\nm{w^1\fe}_{H^1_{\sigma}}+\e\nm{\fe}_{H^2_{\sigma}}\Big)\nmsww{\partial_x(\ik-\pk)[\fe]}\no\\
    \ls&o(1)\e^{-1}\nmsww{\partial_x(\ik-\pk)[\fe]}^2+\e^2\nm{w^1\fe}_{H^1_{\sigma}}^2+\e^3\nm{\fe}_{H^2_{\sigma}}^2\no\\
    \ls&o(1)\e^{-1}\nmsww{\partial_x(\ik-\pk)[\fe]}^2+\e^2\nmsw{(\ik-\pk)[\fe]}^2\\
    &+\e^3\nm{(\ik-\pk)[\fe]}_{H^2_{\sigma}}^2+\e^3\nm{\fe}_{H^2}^2+\e^2\nm{\fe}_{H^1}^2.\no
\end{align*}

\textbf{Sixth Term on the R.H.S. of \eqref{wH1f1}:}
Similarly, for $t\in[0, \e^{-1/3}]$, we use \eqref{Weight1} and \eqref{growth0} in Proposition \ref{fn} to bound it by
\begin{align}
&\sum_{i=1}^{2k-1}\e^{i-1}(1+t)^{i}\nmsww{\partial_x(\ik-\pk)[\fe]}\nm{\fe}_{H^1_{w,\sigma}}\\
\ls &o(1)\e^{-1}\nm{\partial_x(\ik-\pk)[\fe]}_{w^1,\sigma}^2+C\sum_{i=1}^{2k-1}\e^{2i-1}(1+t)^{2i}\nm{\fe}_{H^1_{w,\sigma}}^2\no\\
\ls& o(1)\e^{-1}\nm{(\ik-\pk)[\fe]}_{H^1_{w,\sigma}}^2+\e^{\frac{1}{3}}\nm{\fe}^2_{H^1}.\no
\end{align}


\textbf{Other Terms on the R.H.S. of \eqref{wH1f1}:}
Similarly, these terms can be controlled by
\begin{align*}
o(1)\e^{-1}\nm{(\ik-\pk)[\fe]}_{H^1_{w,\sigma}}^2+ C\e\Big(\nm{E_R^{\e}}_{H^1}^2+\nm{B_R^{\e}}^2_{H^1}+\nm{\fe}^2_{H^2}\Big)+C\e^{2k+1}(1+t)^{4k+2}.
\end{align*}

\textbf{Summary:}
We collect the above estimates in \eqref{wH1f1} and multiply it by $\e$  to obtain \eqref{wH1f12}.

\end{proof}

\subsection{Weighted Second-Order Derivatives Estimates}

In this subsection, we proceed the weighted  $L^2$ estimate of  $\nabla_x^2\fe$.

\begin{proposition}\label{H2p}
For the remainders $\Big(\fe, E^{\e}_ R, B^{\e}_ R\Big)$, it holds that
\begin{align}\label{w2x001}
& \e^{3}\left(\frac{\ud}{\ud t}\nmwww{\nabla_x^2\fe}^2
+Y\nmwww{\sqrt{p^0}\nabla_x^2\fe}^2
+\frac{\delta}{\e}\nmswww{\nabla_x^2(\ik-\pk)[\fe]}^2\right)\\
\lesssim & \e^2\nm{\nabla_x^2\fe}^2+\e\big(\ee+\dd\big)+\e^{2k+4}(1+t)^{4k+2}.
\no
\end{align}
\end{proposition}

\begin{proof}
Noting that $\jump{\li,\partial_x^2}[\fe]$ contains terms that $\partial_x$ hits $\fe$ at most once, we have
\begin{align*}
    &\e^{-1}\abs{\br{\jump{\li,\partial_x^2}[\fe],(w^2)^2\partial_x^2\fe}}\\
    \ls&o(1)\e^{-1}\nmswww{\partial_x^2(\ik-\pk)[\fe]}^2+\overline{\e}_0\e^{-1}\nm{(\ik-\pk)[\fe]}^2_{H^1_{w,\sigma}}+(1+t)^{-\beta_0}\e^{-1}\nm{\fe}_{H^2}^2.\no
\end{align*}
Then we use Lemma \ref{ss 04} to further obtain
\begin{align}
  &\e^{-1}\br{\partial_x^2 \li[\fe],(w^2)^2\partial_x^2\fe}\geq \e^{-1}\br{ \li[\partial_x^2\fe],(w^2)^2\partial_x^2\fe} -\e^{-1}\abs{\br{\jump{\li,\partial_x^2}[\fe],(w^2)^2\partial_x^2\fe}}\\
    \geq&\delta\e^{-1}\nmswww{\partial_x^2(\ik-\pk)[\fe]}^2-C\overline{\e}_0\e^{-1}\nm{(\ik-\pk)[\fe]}^2_{H^1_{w,\sigma}}-C\e^{-1}\nm{\fe}_{H^2}^2.\no
\end{align}

Now we take $|\alpha|=2$ in \eqref{alpha} and multiply the equation by
$(w^2)^2\partial_x^2\fe $ to get
\begin{align}\label{w2x0}
    & \frac{1}{2}\frac{\ud}{\ud t}\nmwww{\partial_x^2\fe}^2
    +Y\nmwww{\sqrt{p^0}\partial_x^2\fe}^2
    +\frac{\delta}{\e}\nmswww{\partial_x^2(\ik-\pk)[\fe]}^2\\
    \leq& \abs{\br{\partial_x^2\bigg(\mh  \big(E_R^{\e}+\hat{p} \times B_R^{\e} \big)\cdot \frac{u^0\hat{p}-u }{T }\bigg), (w^2)^2\partial_x^2\fe }}\no\\
    &+\abs{\br{\partial_x^2\bigg(\big(E +\hat{p} \times B  \big)\cdot\nabla_p\fe \bigg),(w^2)^2\partial_x^2\fe }}\no\\
    &+\abs{\br{\partial_x^2\bigg( \m ^{-\frac{1}{2}}\fe\Big\{\dt +\hat{p}\cdot\nabla_x-\big(E +\hat{p} \times B  \big)\cdot\nabla_p \Big\}\left[\mh \right]\bigg),(w^2)^2\partial_x^2\fe  }}\no\\
    &+\e^{k-1}\abs{\br{\partial_x^2\Gamma\left[\fe,\fe\right],(w^2)^2\partial_x^2\fe}}\no\\
    &+\abs{\sum_{i=1}^{2k-1}\e^{i-1}\br{\partial_x^2\Gamma\left[ \m ^{-\frac{1}{2}}F_i, \fe\right]+\partial_x^2\Gamma\left[\fe,  \m ^{-\frac{1}{2}}F_i\right],(w^2)^2\partial_x^2\fe}}\no\\
    &+\e^k\abs{\br{\partial_x^2\bigg(\frac{\big(u^0\hat{p}-u \big)}{2T }\cdot\big(E_R^{\e}+\hat{p} \times B_R^{\e} \big)\fe \bigg),(w^2)^2\partial_x^2\fe  }}\no\\
    &+\e^k\abs{\br{ \partial_x^2\bigg(\big(E_R^{\e}+\hat{p} \times B_R^{\e} \big)\cdot\nabla_p\fe \bigg),(w^2)^2\partial_x^2\fe  }}\no\\
    &+\abs{\sum_{i=1}^{2k-1}\e^i\br{ \partial_x^2\bigg(\big(E_i+\hat{p} \times B_i \big)\cdot\nabla_p\fe \bigg),(w^2)^2\partial_x^2\fe }}\no\\
    &+\abs{\sum_{i=1}^{2k-1}\e^i\br{ \partial_x^2\bigg(\big(E_R^{\e}+\hat{p} \times B_R^{\e} \big)\cdot \m ^{-\frac{1}{2}}\nabla_pF_{i}\bigg),(w^2)^2\partial_x^2\fe }}\no\\
    &+\abs{\sum_{i=1}^{2k-1}\e^i\Big|\br{ \partial_x^2\bigg(\big(E_i+\hat{p} \times B_i \big)\cdot\frac{\big(u^0\hat{p}-u \big)}{2T }\fe \bigg),(w^2)^2\partial_x^2\fe  }}\no\\
    &+ \e^{k}\abs{\br{ \partial_x^2\sb ,(w^2)^2\partial_x^2\fe}}+C\overline{\e}_0\e^{-1}\nm{(\ik-\pk)[\fe]}^2_{H^1_{w,\sigma}}+C\e^{-1}\nm{\fe}_{H^2}^2.
\no
\end{align}
Now we treat the terms in the R.H.S. of \eqref{w2x0}. 

\textbf{First Term on the R.H.S. of \eqref{w2x0}:}
By \eqref{smallw}, it can be bounded by
\begin{align*}
&C\Big(\nm{\fe}_{H^2}^2+\nm{E_R^{\e}}_{H^2}^2+\nm{B_R^{\e}}_{H^2}^2\Big).
\end{align*}

\textbf{Second Term on the R.H.S. of \eqref{w2x0}:}
Its upper bound is
\begin{align*}
 \zzz\nm{\fe}^2_{H^2_{w,\sigma}}\lesssim (1+t)^{-\beta_0}\overline{\e}_0\Big(\nm{(\ik-\pk)[\fe]}^2_{H^2_{w,\sigma}}+\nm{\fe}^2_{H^2}\Big).
\end{align*}

\textbf{Third Term on the R.H.S. of \eqref{w2x0}:}
Similar to \eqref{L200}, it can be bounded by 
\begin{align*}
& \zzz \nmwww{\sqrt{p^0}\partial_x^2\fe}^2+ \zzz
\nm{\fe}^2_{H^2_{w,\sigma}}\\
\leq& \frac{\yy}{2} \nmwww{\sqrt{p^0}\partial_x^2\fe}^2+C(1+t)^{-\beta_0}\overline{\e}_0\Big(\nm{(\ik-\pk)[\fe]}^2_{H^2_{w,\sigma}}+\nm{\fe}^2_{H^2}\Big).\no
\end{align*}

\textbf{Fourth Term on the R.H.S. of \eqref{w2x0}:}
We use \eqref{smallw}, \eqref{Weight1} in Lemma \ref{ss 05} and Sobolev's inequalities to have
\begin{align*}
    &\abs{\br{\e^{k-1}\partial_x^2\Gamma[\fe,\fe],(w^2)^2\partial_x^2\fe}}\\
    \ls&\abs{\e^{k-1}\br{\Gamma[\partial_x^2\fe,\fe]+\Gamma[\fe,\partial_x^2\fe]+\Gamma[\partial_x\fe,\partial_x\fe],(w^2)^2\partial_x^2\fe}}\no\\
    &+\zzz\abs{\e^{k-1}\br{\Gamma[\partial_x\fe,\fe]+\Gamma[\fe,\partial_{x}\fe],(w^2)^2\partial_x^2\fe}}+\zzz\abs{\e^{k-1}\br{\Gamma[\fe,\fe],(w^2)^2\partial_x^2\fe}}\no\\
    \ls&\e^{k-1}\int_{x\in\r^3}\Big[\abss{w^2\partial_x^2\fe}\tbs{\fe}+\tbs{w^2\partial_x^2\fe}\abss{\fe}+\abss{w^2\fe}\tbs{\partial_x^2\fe}+\tbs{w^2\fe}\abss{\partial_x^2\fe}\no\\
    &+\abss{w^2\partial_x\fe}\tbs{\partial_x\fe}+\tbs{w^2\partial_x\fe}\abss{\partial_x\fe}+\zzz\Big(\abss{w^2\partial_x\fe}\tbs{\fe}+\tbs{w^2\partial_x\fe}\abss{\fe}\no\\
    &+\abss{w^2\fe}\tbs{\partial_x\fe}+\tbs{w^2\fe}\abss{\partial_x\fe}+\abss{w^2\fe}\tbs{\fe}+\tbs{w^2\fe}\abss{\fe}\Big)\Big]\abss{w^2\partial_x^2\fe}\no\\
    \ls&\e^{k-1}\Big(\nm{\fe}_{H^2}\nm{w^2\fe}_{H^2_{\sigma}}+\nm{\fe}_{H^2_{\sigma}}\nm{w^2\fe}_{H^2}\Big)\nmswww{\partial_x^2\fe}\no\\
    \ls& \Big(\e^{\frac{1}{2}}\nm{w^2\fe}_{H^2_{\sigma}}+\nm{\fe}_{H^2_{\sigma}}\Big)\nmswww{\partial_x^2\fe}\no\\
    \ls&\nm{(\ik-\pk)[\fe]}^2_{H^2_{w,\sigma}}+\nm{\fe}_{H^2}^2.\no
\end{align*}

\textbf{Fifth Term on the R.H.S. of \eqref{w2x0}:}
Similarly,  Similarly,  we use \eqref{smallw}, \eqref{Weight1} and \eqref{growth0} in Proposition \ref{fn} to obtain that for $t\leq \overline{t}=\e^{-1/3}$,
\begin{align*}
&\abs{\sum_{i=1}^{2k-1}\e^{i-1}\br{\partial_x^2\Gamma\left[ \m ^{-\frac{1}{2}}F_i, \fe\right]+\partial_x^2\Gamma\left[\fe, \m ^{-\frac{1}{2}}F_i\right],(w^2)^2\partial_x^2\fe}}\\
\lesssim& \sum_{i=1}^{2k-1}\e^{i-1}\left(\nm{w^2 \m ^{-\frac{1}{2}}F_i}_{W^{2,\infty}_xL^2_p}\nm{\fe}_{H^2_{w,\sigma}}
+\Big\|\big|w^2 \m ^{-\frac{1}{2}}F_i\big|_{\sigma}\Big\|_{W^{2,\infty}_x}\nm{\fe}_{H^2_{w}}\right)\nms{w^2\fe}^2\no\\
\ls& (1+t)\nm{\fe}_{H^2_{w,\sigma}}^2
\ls o(1)\e^{-1}\nm{(\ik-\pk)[\fe]}^2_{H^2_{w,\sigma}}+\e^{-\frac{1}{3}}\nm{\fe}^2_{H^2}\no.
 \end{align*}

\textbf{Other Terms on the R.H.S. of \eqref{w2x0}:}
Similarly, these terms can be controlled by
\begin{align*}
&o(1)\e^{-1}\nm{(\ik-\pk)[\fe]}^2_{H^2_{w,\sigma}}+C\e^{-1}\nm{\fe}^2_{H^2}\\
&+C\e\Big(\nm{E_R^{\e}}^2_{H^2}+\nm{B_R^{\e}}^2_{H^2}\Big)+C\e^{2k+1}(1+t)^{4k+2}.\no
\end{align*}

\textbf{Summary:}
We collect the above estimates in \eqref{w2x0} and multiply the resulting inequality by $\e^3$ to derive \eqref{w2x001}.
\end{proof}


\section{Macroscopic Estimates and Electromagnetic Dissipation}\label{sec: macro}

In this section, we study the macroscopic estimates of $\fe$ and electromagnetic dissipation. With these estimates and the estimates obtained in the previous two Sections, we can finally close the whole energy estimate.

\subsection{Macroscopic Estimates}

To capture the dissipation of the macroscopic part of $\fe$ which can be seen as a perturbation around a local Maxwellian,
as in \eqref{macfe}, we write $\pk [\fe]$ as
\begin{align}
\pk [\fe]=\Big(a^{\e}-\frac{\rho_{2}}{\rho_{1}}c^{\e}\Big)\m ^{\frac{1}{2}}+ b^{\e}\cdot p\m ^{\frac{1}{2}}+ c^{\e}p^{0}\m ^{\frac{1}{2}},
\end{align}
where $\rho_{1}$ and $\rho_{2}$ are defined in \eqref{rho12}.
\begin{proposition}\label{md00}
There are two functionals $\mathcal{E}^{mac}_i$ for $i=1,2$ satisfying
\begin{align}\label{i11}
   \ee^{mac}_i\ls \tnm{\nabla_x^{i-1}\fe}\tnm{\nabla_x^i\fe}, 
\end{align}
such that 
\begin{align}\label{semp 5}
   & -\frac{\ud}{\ud t}\Big(\e\ee^{mac}_1+\e^2\ee^{mac}_2\Big)+\e\Big(\tnm{\nabla_x\pk[\fe]}^2+\tnm{a^{\e}}^2+\tnm{\big(\nabla_x\cdot E^{\e}_R\big)}^2\Big)\\
   &+\e^2\Big(\tnm{\nabla_x^2\pk[\fe]}^2+\tnm{\nabla_xa^{\e}}^2+\tnm{\nabla_x\big(\nabla_x\cdot E^{\e}_R\big)}^2\Big)\no\\
    \ls& \e^{-1}\nms{(\ik-\pk)[\fe]}^2 +\nms{\nabla_x(\ik-\pk)[\fe]}^2+\e\nms{\nabla_x^2(\ik-\pk)[\fe]}^2+\e^{\frac{2}{3}}\big(\ee+\dd)+\e^{k+1}(1+t)^{2k+1}.\no
\end{align}
\end{proposition} 

\begin{proof} 
 Motivated by \cite[Lemma 6.1]{Guo2006}, we will prove this proposition by two key ingredients: local conservation laws and the macroscopic equations of $\fe$.
For convenience, we write \eqref{L20} as 
\begin{align}\label{re-f1}
    &\dt \fe +\hat{p}\cdot\nabla_x\fe 
   +  \frac{u^0}{T }\hat{p}\mh \cdot E_R^{\e}+\frac{1}{\e}\li \left[\fe \right]
     =\overline{h}^{\e},
\end{align}
where
\begin{align}
   \overline{h}^{\e}=&\Big(E +\hat{p} \times B  \Big)\cdot\nabla_p\fe + \frac{u   }{ T }\mh\cdot\Big(E_R^{\e}+\hat{p} \times B_R^{\e} \Big)\\
    &- \m ^{-\frac{1}{2}}\fe \Big\{\dt +\hat{p}\cdot\nabla_x- \big(E +\hat{p} \times B  \big)\cdot\nabla_p \Big\}\mh +\e^{k-1}\Gamma \left[\fe,\fe\right]\no\\
    &+\sum_{i=1}^{2k-1}\e^{i-1}\left\{\Gamma \left[ \m ^{-\frac{1}{2}}F_i, \fe\right]+\Gamma \left[\fe,  \m ^{-\frac{1}{2}}F_i\right]\right\}+ e \e^k\Big(E_R^{\e}+\hat{p} \times B_R^{\e} \Big)\cdot\nabla_p\fe \no\\
    &- \e^k\frac{1 }{2 T }\Big(u^0\hat{p}-u \Big)\cdot\Big(E_R^{\e}+\hat{p} \times B_R^{\e} \Big)\fe \no\\
    &+ \sum_{i=1}^{2k-1}\e^i\left\{\Big(E_i+\hat{p} \times B_i \Big)\cdot\nabla_p\fe +\Big(E_R^{\e}+\hat{p} \times B_R^{\e} \Big)\cdot \m ^{-\frac{1}{2}}\nabla_pF_{i}\right\}\no\\
    &- \sum_{i=1}^{2k-1}\e^i\left\{\Big(E_i+\hat{p} \times B_i \Big)\cdot\frac{1}{2 T }\Big(u^0\hat{p}-u \Big)\fe \right\}+\e^{k}\sb,\no
\end{align}

\textbf{Local conservation laws:}
Firstly, we derive the local conservation laws of $a^{\e}, b^{\e}, c^{\e}$. Note that
\begin{align}
    \frac{\rho_{2}}{\rho_{1}}=\frac{\en (u^0)^2+P|u|^2}{n u^0}.
\end{align}
Projecting \eqref{re-f1} onto the null space $\mathcal{N}$, similar to the derivation of \eqref{number},  \eqref{moment} and \eqref{energy},
we can obtain
\begin{align}
   & n u^0\partial_ta^{\e}+P\nabla_x\cdot b^{\e}\\
   =&\Xi_{1}\big[a^{\e},b^{\e},c^{\e}\big]-\nabla_x\cdot\int_{\mathbb R^3}  \hat{p}\mh (\ik-\pk )[\fe ]\,\ud p+\int_{\mathbb R^3}  \mh \overline{h}^{\e}\,\ud p,\no\\
   &\frac{n u^0K_3(\gamma )}{\gamma  K_2(\gamma )}\partial_t b^{\e}+P\nabla_x\Big(a^{\e}-\frac{\en u^0}{n}c^{\e}\Big)+\frac{nu^0K_3(\gamma )}{\gamma  K_2(\gamma )}\nabla_xc^{\e} + n u^0 E_R^{\e}\\
   =&\Xi_{2}\big[a^{\e},b^{\e},c^{\e}\big]-\nabla_x\cdot\int_{\mathbb R^3}  \hat{p}p\mh (\ik-\pk )[\fe ]\,\ud p +\int_{\mathbb R^3}  p\mh \overline{h}^{\e}\,\ud p,\no\\
    &\en (u^0)^2\partial_t \Big(a^{\e}-\frac{\en u^0}{n}c^{\e}\Big)+\frac{n(u^0)^2\big[3K_3(\gamma )+\gamma K_2(\gamma )\big]}{\gamma  K_2(\gamma )}\partial_tc^{\e}+\frac{nu^0K_3(\gamma )}{\gamma  K_2(\gamma )}\nabla_x\cdot b^{\e}\\
    =& \Xi_{3}\big[a^{\e},b^{\e},c^{\e}\big]+\int_{\mathbb R^3}  p^0\mh \overline{h}^{\e}\,\ud p.\no
   \end{align}
where $\Xi_{j}\big[a^{\e},b^{\e},c^{\e}\big]$ for $j=1,2,3$ denotes a combination of linear terms of $a^{\e}, b^{\e}, c^{\e}$ with coefficients $\nabla_{t,x}(n ,u,T )$,  and derivatives of $a^{\e}, b^{\e}, c^{\e}$ with coefficient $u$. Since they are small perturbations and thus will not affect the estimates, we will ignore the details for clarity.



Noting that
\begin{align}
    P=\frac{n}{\gamma },\qquad \en =n\left(\frac{K_3(\gamma )}{K_2(\gamma )}-\frac{1}{\gamma }\right)=n\left(\frac{K_1(\gamma )}{K_2(\gamma )}+\frac{3}{\gamma }\right),
\end{align}
we can further write the above system as 
\begin{align}
    &n u^0\partial_ta^{\e}+\frac{n}{\gamma }\nabla_x\cdot b^{\e}\label{lcl0}\\
    &\qquad=\Xi_{1}\big[a^{\e},b^{\e},c^{\e}\big]-\nabla_x\cdot\int_{\mathbb R^3} \hat{p}\mh (\ik-\pk )[\fe ]\,\ud p+\int_{\mathbb R^3}  \mh \overline{h}^{\e}\,\ud p,\no\\
    &n\left(\frac{K_1(\gamma )}{\gamma  K_2(\gamma )}+\frac{4}{\gamma ^2}\right)\partial_t b^{\e}+\frac{n}{\gamma }\nabla_x a^{\e}
    +\frac{n}{\gamma ^2}\nabla_xc^{\e}+ n u^0 E_R^{\e} \label{lcl00} \\
    &\qquad=\Xi_{2}\big[a^{\e},b^{\e},c^{\e}\big]-\nabla_x\cdot\int_{\mathbb R^3}  \hat{p}p\mh (\ik-\pk )[\fe ]\,\ud p +\int_{\mathbb R^3}  p\mh \overline{h}^{\e}\,\ud p,\no\\
    &n\left(-\frac{K_1^2(\gamma )}{K_2^2(\gamma )}-\frac{3}{\gamma }\frac{K_1(\gamma )}{K_2(\gamma )}+1+\frac{3}{\gamma^2}\right)\partial_tc^{\e}+\frac{n}{\gamma  ^2}\nabla_x\cdot b^{\e} \label{lcl000}\\
    & \qquad=\Xi_{3}\big[a^{\e},b^{\e},c^{\e}\big]+\int_{\mathbb R^3}  p^0\mh \overline{h}^{\e}\,\ud p\no\\
    &\qquad-u^0\left(\frac{K_1(\gamma )}{K_2(\gamma )}+\frac{3}{\gamma }\right)\left(-\nabla_x\cdot\int_{\mathbb R^3}  \hat{p}\mh (\ik-\pk )[\fe ]\,\ud p+\int_{\mathbb R^3}  \mh \overline{h}^{\e}\,\ud p\right).\no
\end{align}
This system fully describes the evolution of $a^{\e}$, $b^{\e}$ and $c^{\e}$.

\textbf{Macroscopic equations:}
Secondly, we turn to the macroscopic equations of $\fe$. Splitting $\fe$ as the macroscopic part $\pk[\fe]$ and the microscopic $(\ik-\pk)[\fe]$ part in \eqref{re-f1}, we have
\begin{align}\label{semp 6}
    \bigg\{\partial_t\Big(a^{\e}-\frac{\rho_{2}}{\rho_{1}}c^{\e}\Big)+p\cdot \partial_tb^{\e}+p^0 \partial_t c^{\e}\bigg\}\mh \\+\hat{p}\cdot\bigg\{\nabla_x\Big(a^{\e}-\frac{\rho_{2}}{\rho_{1}}c^{\e}\Big)+ \nabla_x b^{\e}\cdot p+p^0 \nabla_x c^{\e}\bigg\}\mh + \frac{u^0}{T }\hat{p}\mh \cdot E_R^{\e}&=\ell^{\e}+h^{\e},\no
  \end{align}
where
\begin{align*}
   \ell^{\e}:=& -\big(\partial_t+\hp \cdot\nabla_x\big)\Big[(\ik-\pk )[\fe ]\Big]-\frac{1}{\e}\li[\fe ],\\
   h^{\e}:=&- \left\{\Big(a^{\e}-\frac{\rho_{2}}{\rho_{1}}c^{\e}\Big)+p\cdot b^{\e}+p^0 c^{\e}\right\}\big(\partial_t+\hp\cdot\nabla_x\big)\mh +\overline{h}^{\e}. 
\end{align*}
For fixed $t,x$, we compare the coefficients in front of 
\begin{align}\label{exbasis}
    \left\{\m ^{\frac{1}{2}}, p_i\m ^{\frac{1}{2}}, p^{0}\m ^{\frac{1}{2}},\frac{p_i}{p^0}\m^{\frac{1}{2}}, \frac{p_ip_j}{p^0}\m ^{\frac{1}{2}}\right\},\quad 1\leq i,j\leq 3,
\end{align}
on both sides of \eqref{semp 6} and get the following macroscopic equations:
\begin{align}
    \partial_t a^{\e}- \frac{\rho_{2}}{\rho_{1}}\partial_t c^{\e}&= \ell_{a}^{\e}+h^{\e}_{a}+\partial_t\Big(\frac{\rho_{2}}{\rho_{1}}\Big) c^{\e},\label{macabc}\\
    \partial_t b^{\e}_i+\partial_i c^{\e}&=\ell_{bi}^{\e}+h^{\e}_{bi},\no\\
    \partial_t c^{\e}&=\ell_{c}^{\e}+h^{\e}_{c},\label{macat}\\
    \partial_i a^{\e}- \frac{\rho_{2}}{\rho_{1}}\partial_i c^{\e}+ \frac{u^0}{T } E_{R,i}^{\e}&= \ell_{ai}^{\e}+h^{\e}_{ai}+\partial_i\Big(\frac{\rho_{2}}{\rho_{1}}\Big) c^{\e},\label{macaci}\\
    \partial_i b^{\e}_i&= \ell_{ii}^{\e}+h^{\e}_{ii},\no\\
    \partial_i b_j^{\e}+\partial_j b_i^{\e}&= \ell_{ij}^{\e}+h^{\e}_{ij}.\qquad i\neq j.\no
\end{align}
Here $\ell_a^{\e}, h_a^{\e}$, $\ell_{bi}^{\e}, h^{\e}_{bi}$,  $\ell_{c}^{\e}, h^{\e}_{c}$,  $\ell_{ai}^{\e}, h^{\e}_{ai}$,  $\ell_{ii}^{\e}, h^{\e}_{ii}$ and  $\ell_{ij}^{\e}, h^{\e}_{ij}$ take the form
\begin{align*}
    (\ell^{\e}, \zeta)\ \ \text{and}\ \  (h^{\e}, \zeta),
\end{align*}
where $\zeta$ is linear combinations of vectors in \eqref{exbasis}.
Combing \eqref{macabc} and    \eqref{macat} yields
\begin{align}\label{at}
    \partial_t a^{\e}=\ell_{a}^{\e}+h^{\e}_{a}+ \frac{\rho_{2}}{\rho_{1}}\Big( \ell_{c}^{\e}+h^{\e}_{c}\Big)+\partial_t\Big(\frac{\rho_{2}}{\rho_{1}}\Big) c^{\e}.
\end{align}

For $m=0,1$, we have the following estimates:
\begin{align}
    &\nm{\nabla_x^mh_a^{\e}}+\nm{\nabla_x^mh_{bi}^{\e}}+\nm{\nabla_x^mh_{c}^{\e}}+\nm{\nabla_x^mh_{ai}^{\e}}+\nm{\nabla_x^mh_{ii}^{\e}}+\nm{\nabla_x^mh_{ij}^{\e}}\label{h01}\\
    \ls& \zzz\Big(\nm{E^{\e}_{R}}_{H^m}+\nm{B^{\e}_{R}}_{H^m}\Big)+\Big(\zzz+\e^{\frac{1}{2}}\Big)\nm{\fe}_{H^m_{\sigma}}\no\\
    &+\sum_{l=1}^{2k-1}\e^{l-1}\left(\big\|\mhh F_l\big\|_{W^{m,\infty}_xL^2_p}\nm{\fe}_{H^m_{\sigma}}+\Big\||\mhh F_l|_{\sigma}\Big\|_{W^{m,\infty}_x}\nm{\fe}_{H^m}\right)\no\\
    &+\sum_{l=1}^{2k-1}\e^{l}(1+t)^{l}\left(\nm{\fe}_{H^m_{\sigma}}+\nm{E^{\e}_{R}}_{H^m}+\nm{B^{\e}_{R}}_{H^m}\right)+\e^{k}(1+t)^{2k+1}\no\\
    \ls& \e^{-\frac{1}{3}}\big(\nm{(\ik-\pk)[\fe]}_{H^m_{\sigma}}+\nm{\fe}_{H^m}\big)+\Big[(1+t)^{-\beta_0}+\e^{\frac{2}{3}}\Big]\Big(\nm{E^{\e}_{R}}_{H^m}+\nm{ B^{\e}_{R}}_{H^m}\Big)+\e^{k}(1+t)^{2k+1},\no
\end{align}
for $t\in [0, \e^{-1/3}]$.
For brevity, we only give the estimate of $\nm{\nabla_x a^{\e}}$ and $\nm{a^{\e}}$. The estimates w.r.t. $b^{\e}$ and $c^{\e}$ in \eqref{semp 5} can be derived similarly as in \cite[Lemma 6.1]{Guo2006}.
From  \eqref{macaci}, we have
\begin{align}\label{apm}
  &-\Delta a^{\e}+\frac{\rho_{2}}{\rho_{1}}\Delta  c^{\e}- \frac{u^0}{T } \nabla_x\cdot E_{R}^{\e}\\
  =& -\sum_{i=1}^3\partial_i\big(\ell_{ai}^{\e}+h^{\e}_{ai}\big)-\nabla_x\cdot \Big[c^{\e}\nabla_x \Big(\frac{\rho_{2}}{\rho_{1}}\Big)\Big]-\nabla_x\Big(\frac{u^0}{T } \Big) \cdot E_{R}^{\e}.\no
\end{align}
On the other hand, by \eqref{L201}, we have
\begin{align}\label{Ediv}
    \nabla_x\cdot E_R^{\e}&=-4\pi \int_{\mathbb R^3} \m^{\frac{1}{2} }\fe \ud p\\
    &=-4\pi  n a^{\e}-4\pi n a^{\e}(u^0-1)-4\pi(\en+P)u^0 u\cdot b\no\\
    &=-4\pi \overline{n}a^{\e}-4\pi  \big( n -\overline{n}\big)a^{\e}-4\pi  n a^{\e}(u^0-1)-4\pi(\en+P)u^0 u\cdot b. \no
\end{align}
Collecting \eqref{Ediv} in \eqref{apm}, we have
\begin{align}\label{apm0}
  -\Delta  a^{\e}+\frac{\en }{n }\Delta  c^{\e}+ \frac{4\pi \overline{n}}{ T } a^{\e}=& -\sum_{i=1}^3\partial_i\big(\ell_{ai}^{\e}+h^{\e}_{ai}\big)+\Xi_{5}\big[a^{\e},b^{\e},c^{\e}\big],
\end{align}
where 
\begin{align*}
  \Xi_{5}\big[a^{\e},b^{\e},c^{\e}\big]=&-\frac{\en (u^0-1)}{n }\Delta  c^{\e}-4\pi   \big( n -\overline{n}\big)a^{\e}-4\pi  n a^{\e}(u^0-1)\\
  &-4\pi(\en+P)u^0 u\cdot b -\nabla_x\cdot \Big[c^{\e}\nabla_x \Big(\frac{\rho^{\e}_{2}}{\rho^{\e}_{1}}\Big)\Big]+\nabla_x\Big( \frac{u^0}{T } \Big) \cdot E_{R}^{\e}.\no
\end{align*}
Note that
\begin{align*}
   \Big| \br{\frac{\en }{n }\Delta  c^{\e},\overline{n}a^{\e}}\Big|&=\Big| \br{\frac{\en }{n }\nabla_x c^{\e},\nabla_x\big(\overline{n}a^{\e}\big)}\Big|\leq o(1)\nm{\nabla_x a^{\e}}^2+C\nm{\nabla_x c^{\e}}^2
\end{align*}
We multiply \eqref{apm0} by $\overline{n}a^{\e}$ and integrate the resulting equalities to have
\begin{align}\label{pma}
    &\frac{3c}{4}\nm{\sqrt{\overline{n}}\nabla_x a^{\e}}^2
    +\overline{n}^2\nm{\sqrt{\frac{4\pi }{T }}a^{\e}}^2\\
    \leq& \sum_{i=1}^3\br{\ell_{ai}^{\e},\overline{n}\partial_i a^{\e}}+  \sum_{i=1}^3\br{h^{\e}_{ai},\overline{n}\partial_i a^{\e}}+ \Big|\br{\Xi_{5}\big[a^{\e},b^{\e},c^{\e}\big], \overline{n}a^{\e}}\Big|+C\nm{\nabla_x c^{\e}}^2.\no
\end{align}

For the first term in \eqref{pma}, we have
\begin{align}
    &\br{\ell_{ai}^{\e},\overline{n}\partial_i a^{\e}}\\
    =&\br{\Big(-\big(\partial_t+\hp \cdot\nabla_x\big)\Big[(\ik-\pk )[\fe ]\Big]-\frac{1}{\e}\li[\fe ], \zeta_{ai}\Big),\overline{n}\partial_i a^{\e}}\no\\
    =&\br{\Big(-\partial_t\Big[(\ik-\pk )[\fe ]\Big], \zeta_{ai}\Big), \overline{n}\partial_i a^{\e}}\no\\
    &+\br{\Big(-\hp \cdot\nabla_x\Big[(\ik-\pk )[\fe ]\Big]-\frac{1}{\e}\li[\fe ], \zeta_{ ai}\Big), \overline{n}\partial_i a^{\e}}\no\\
    \leq&-\frac{\ud}{\ud t}\br{\Big(\Big[(\ik-\pk )[\fe ]\Big], \zeta_{ai}\Big), \overline{n}\partial_i a^{\e}}+\br{\Big(\Big[(\ik-\pk )[\fe ]\Big], \zeta_{ai}\Big), \overline{n}\partial_i\partial_t a^{\e}}\no\\
    &+o(1) \nm{\nabla_xa^{\e}}^2+C\Big(\nms{\nabla_x(\ik-\pk )[\fe ]}^2+\e^{-2}\nms{(\ik-\pk )[\fe ]}^2\Big).\no
\end{align} 
By \eqref{at} and \eqref{h01}, we have
\begin{align}
   & \left|\br{\Big(\Big[(\ik-\pk )[\fe ]\Big], \zeta_{ai}\Big),\overline{n}\partial_i\partial_t a^{\e}}\right|\\
    \leq&\left|\br{\partial_i\Big(\Big[(\ik-\pk )[\fe ]\Big], \zeta_{ai}\Big),\overline{n}\partial_t a^{\e}}\right|
    \ls \nm{\partial_ta^{\e}}^2+\nm{(\ik-\pk)[\fe]}_{H^1_{\sigma}}^2\no\\
    \ls& \e^{-\frac{1}{3}}\nm{(\ik-\pk)[\fe]}^2_{H^1_{\sigma}}+\e^{-\frac{1}{3}}\nm{\fe}^2_{H^1}+\Big[(1+t)^{-\beta_0}+\e^{\frac{2}{3}}\Big]\Big(\nm{E^{\e}_{R}}_{H^1}^2+\nm{B^{\e}_{R}}^2_{H^1}\Big)+\e^{k}(1+t)^{2k+1}.\no
\end{align}
Therefore,
\begin{align}\label{ai1}
    & \sum_{i=1}^3\br{\ell_{ai}^{\e}, \overline{n}\partial_i a^{\e}}\\
    \leq&-\sum \frac{\ud}{\ud t}\br{\Big(\Big[(\ik-\pk )[\fe ]\Big], \zeta_{ ai}\Big), \overline{n}\partial_i a^{\e}}+o(1) \nm{\nabla_xa^{\e}}^2+C\e^{-\frac{1}{3}}\nms{\nabla_x(\ik-\pk)[\fe]}^2\no\\
    &+C\Big[\e^{-2}\nms{(\ik-\pk)[\fe]}^2+\e^{-\frac{1}{3}}\nm{\fe}^2_{H^1}+\Big[(1+t)^{-\beta_0}+\e^{\frac{2}{3}}\Big]\Big(\nm{E^{\e}_{R}}^2_{H^1}+\nm{B^{\e}_{R}}^2_{H^1}\Big)+\e^{k}(1+t)^{2k+1}\Big].\no
\end{align} 
For the second and third terms in \eqref{pma}, we use \eqref{h01} to have
\begin{align}\label{pma01}
     \sum_{i=1}^3\br{h^{\e}_{ai}, \overline{n}\partial_i a^{\e}}\ls&o(1) \nm{\nabla_xa^{\e}}^2+\e^{-\frac{1}{3}}\nms{(\ik-\pk)[\fe]}^2+\nm{\fe}^2\\
    &+\Big[(1+t)^{-\beta_0}+\e^{\frac{2}{3}}\Big]\Big(\nm{E^{\e}_{R}}^2+\nm{B^{\e}_{R}}^2\Big)+\e^{k}(1+t)^{2k+1}\no.
\end{align}
and
\begin{align}\label{pma02}
    \sum \Big|\br{\Xi_{5}\big[a^{\e},b^{\e},c^{\e}\big], na^{\e}}\Big|\ls o(1)    \nm{\nabla_x\pk[\fe]}^2+C(1+t)^{-\beta_0}\Big(\nm{\fe}^2+\nm{E^{\e}_{R}}^2\Big).
\end{align}
Collecting the estimates \eqref{ai1}, \eqref{pma01} and \eqref{pma02} in \eqref{pma}, we obtain
\begin{align}\label{pma111}
    &\frac{1}{2}\nm{\sqrt{\overline{n}}\nabla_x a^{\e}}^2
    +\overline{n}^2\nm{\sqrt{\frac{4\pi }{T }}a^{\e}}^2\\
    \leq&-\frac{\ud}{\ud t}\br{\Big((\ik-\pk )[\fe ], \zeta_{ai}\Big),\partial_i a^{\e}}+C\e^{-\frac{1}{3}}\nms{\nabla_x(\ik-\pk )[\fe ]}^2\no\\
    &+C\bigg(\e^{-2}\nms{(\ik-\pk)[\fe]}^2+\e^{-\frac{1}{3}}\nm{\fe}^2_{H^1}+\Big[(1+t)^{-\beta_0}+\e^{\frac{2}{3}}\Big]\Big(\nm{E^{\e}_{R}}^2_{H^1}+\nm{B^{\e}_{R}}^2_{H^1}\Big)+\e^{k}(1+t)^{2k+1}\bigg).\no
\end{align}
\end{proof}



\subsection{Electromagnetic Dissipation}

In this subsection, we derive the dissipation of the electromagnetic field $(E^{\e}_R, B^{\e}_R)$.
\begin{proposition}\label{dE01}
It holds that, for 
$i=0,1$
\begin{align}\label{E01}
   & \frac{1}{2}\e^{i+1}\nm{\sqrt{\frac{u^0}{T }}\nabla_x^i E_{R}^{\e}}^2\\
   \leq& - \sum_{j=1}^3\e^{i+1}\frac{\ud}{\ud t}\br{\Big(\nabla_x^{i}(\ik-\pk)[\fe], \zeta_{aj}\Big),\nabla_x^i E_{R,j}^{\e}}+C\e^{i-1}\nms{(\ik-\pk )[\fe ]}^2\no\\
   &+C\Big[\e^{i+1}\nm{\nabla_x^{i+1}\pk[\fe]}^2+\e^{\frac{2}{3}}\big(\ee+\dd\big)+\e^{k+i+1}(1+t)^{2k+1}\Big].\no
\end{align}

\end{proposition}

\begin{proof} For brevity, we only prove  \eqref{E01} for $i=0$ since the case $i=1$ can be proved in the same way.
From  \eqref{macaci}, we have
\begin{align}\label{E0}
    \nm{\sqrt{\frac{u^0}{T }}E_{R}^{\e}}^2\leq\sum_{i=1}^3\br{\ell_{ai}^{\e},E_{R,i}^{\e}}+\sum_{i=1}^3\left|\br{\partial_i a^{\e}-\frac{\rho^{\e}_{2}}{\rho^{\e}_{1}}\partial_i c^{\e},E_{R,i}^{\e}}\right|+\sum_{i=1}^3\left|\br{h^{\e}_{ai},E_{R,i}^{\e}}\right|.
\end{align}
For the first term in \eqref{E0}, we use \eqref{L201} to have
\begin{align*}
   &\br{\ell_{ai}^{\e},E_{R,i}^{\e}}\\
    =&\br{\Big(-\partial_t(\ik-\pk)[\fe], \zeta_{ai}\Big),E_{R,i}^{\e}}+\br{\Big(-\hp\cdot\nabla_x(\ik-\pk)[\fe]-\frac{1}{\e}\li[\fe], \zeta_{ai}\Big),E_{R,i}^{\e}}\no\\
    \leq&-\frac{\ud}{\ud t}\br{\Big(\Big[(\ik-\pk)[\fe]\Big], \zeta_{ai}\Big),E_{R,i}^{\e}}+\br{\Big((\ik-\pk)[\fe], \zeta_{ai}\Big),\partial_tE_{R,i}^{\e}}\no\\
    &+o(1) \nm{E^{\e}_{R}}^2+C\Big(\nms{\nabla_x(\ik-\pk)[\fe]}^2+\e^{-2}\nms{(\ik-\pk)[\fe]}^2\Big)\no\\
    \leq&-\frac{\ud}{\ud t}\br{\Big(\Big[(\ik-\pk)[\fe]\Big], \zeta_{ai}\Big),E_{R,i}^{\e}}+\e^2\big(\nm{\nabla_xB^{\e}_R}^2+\nm{\fe}^2\big)\no\\
    &+o(1) \nm{E^{\e}_{R}}^2+C\Big(\nms{\nabla_x(\ik-\pk)[\fe]}^2+\e^{-2}\nms{(\ik-\pk)[\fe]}^2\Big).\no
\end{align*}
For the second and third terms in \eqref{E0}, we bound them as follows:
\begin{align*}
   & \sum_{i=1}^3\left|\br{\partial_i a^{\e}-\frac{\rho_{2}}{\rho_{1}}\partial_i c^{\e},E_{R,i}^{\e}}\right|\ls o(1)\nm{E^{\e}_{R}}^2+C\nm{\nabla_x\pk[\fe]}^2,\\
  &\sum_{i=1}^3\left|\br{h^{\e}_{ai},E_{R,i}^{\e}}\right|\ls o(1) \nm{E^{\e}_{R}}^2+\e^{-\frac{1}{3}}\big(\nms{(\ik-\pk)[\fe]}^2+\nm{\fe}^2\big)\no\\
    &+\Big[(1+t)^{-\beta_0}+\e^{\frac{2}{3}}\Big]\Big(\nm{E^{\e}_{R}}^2+\nm{B^{\e}_{R}}^2\Big)+\e^{k}(1+t)^{2k+1}\no.
\end{align*}
We collect the above estimates in \eqref{E0} to obtain
\begin{align}\label{E1}
  \frac{1}{2} \nm{\sqrt{\frac{u^0}{T }}E_{R}^{\e}}^2
   \leq& - \sum_{i=1}^3\frac{\ud}{\ud t}\br{\Big((\ik-\pk)[\fe], \zeta_{ai}\Big),E_{R,i}^{\e}}\\
   &+C\Big(\nms{\nabla_x(\ik-\pk)[\fe]}^2+\e^{-2}\nms{(\ik-\pk)[\fe]}^2\Big)\no\\
  &+C\Big(\e^2\nm{\nabla_xB^{\e}_R}^2+\e^{-\frac{1}{3}}\nm{\fe}^2+\nm{\nabla_x\pk[\fe]}^2\Big)\no\\
  &+C\Big[(1+t)^{-\beta_0}+\e^{\frac{2}{3}}\Big]\Big(\nm{E^{\e}_{R}}^2+\nm{B^{\e}_{R}}^2\Big)+\e^{k}(1+t)^{2k+1}\Big].\no
\end{align}
This verifies \eqref{E01} for $i=0$.

\end{proof}

\begin{proposition} \label{EBd}
It holds that 
\begin{align}\label{EB}
    \frac{1}{2}\e^2 \nm{\nabla_x B_R^{\e}}^2\leq& \e^{2}\frac{\ud}{\ud t}\br{ E_R^{\e},\nabla_x \times B_R^{\e}}+\e^{2}\nm{\nabla_x E_R^{\e}}^2+C\e^2\ee.
\end{align}
\end{proposition}
\begin{proof}
From \eqref{L201}, we have
\begin{align}\label{curB}
 \nm{\nabla_x \times B_R^{\e}}^2=&  \br{\ds\dt E_R^{\e},\nabla_x \times B_R^{\e}}+\br{-4\pi \int_{\mathbb R^3}\hat{p} \mh  \fe\ud p,\nabla_x \times B_R^{\e}}\\
 \leq& \frac{\ud}{\ud t}\br{ E_R^{\e},\nabla_x \times B_R^{\e}}- \br{ E_R^{\e},\nabla_x \times \ds\dt B_R^{\e}}+C\nm{\fe}\nm{\nabla_x \times B_R^{\e}}\no\\
\leq& \frac{\ud}{\ud t}\br{ E_R^{\e},\nabla_x \times B_R^{\e}}+\br{ E_R^{\e},\nabla_x \times \big(\nabla_x \times E_R^{\e}\big)}+o(1)\nm{\nabla_x \times B_R^{\e}}^2+C\nm{\fe}^2\no.
\end{align}
Noting
\begin{align}\label{curE}
    \br{ E_R^{\e},\nabla_x \times \big(\nabla_x \times E_R^{\e}\big)}=\nm{\nabla_x \times E_R^{\e}}^2,
\end{align}
we have
\begin{align}
      \frac{1}{2}\nm{\nabla_x \times B_R^{\e}}^2\leq \frac{\ud}{\ud t}\br{ E_R^{\e},\nabla_x \times B_R^{\e}}+\nm{\nabla_x \times E_R^{\e}}^2+C\nm{\fe}^2.
\end{align}
Noting $\nabla\cdot B_R^{\e}=0$, this further implies
 \eqref{EB}.
\end{proof}
\begin{remark}\label{rkEB}
By proper linear combination of \eqref{semp 5}, \eqref{E01} and \eqref{EB}, we can obtain the macroscopic dissipation and the electromagnetic filed dissipation together. However, the dissipation of the electromagnetic field are too weak to be necessarily included in $\dd$.
\end{remark}


\section{Proof of Proposition~\ref{result}}\label{mr}

\ \\
\textbf{Proof of Energy Estimates:}
We multiply \eqref{semp 5} by a sufficiently small constant $\kappa_1$ and collect the resulting inequality, 
    \eqref{L2ener01}, \eqref{H1ener01} and \eqref{H2ener01} to obtain that for some small positive constant $\delta_1$,
\begin{align}\label{ud}
    &\frac{\ud}{\ud t}\bigg\{\sum_{i=0}^2\e^i\bigg(\nm{\sqrt{\frac{4\pi T }{u^0}}\nabla_x^i\fe}^2+\nm{\nabla_x^i E_R^{\e}}^2+\nm{\nabla_x^iB_R^{\e}}^2\bigg)-\kappa_1\Big(\e\mathcal{E}^{mac}_1+\e^2\mathcal{E}^{mac}_2\Big)\bigg\}\\
    &+\delta_1\e\tnm{\nabla_x\pk[\fe]}^2+\delta\e^2\tnm{\nabla_x^2\pk[\fe]}^2\no\\
    &+\delta_1\Big(\e^{-1}\nms{(\ik-\pk)[\fe]}^2+\nms{\nabla_x(\ik-\pk)[\fe]}^2+\e\nms{\nabla_x^2(\ik-\pk)[\fe]}^2\Big)\no\\
    \ls&\Big[(1+t)^{-\beta_0}+\e^{\frac{1}{3}}\Big]\mathcal{E}+\e^{\frac{2}{3}}\mathcal{D} +\e^{2k+1}(1+t)^{4k+2}
+\e^{k}(1+t)^{2k}\sqrt{\mathcal{E}}.\no
\end{align}

Multiplying \eqref{ud} by a large constant $C_1$  and adding it to the sum of \eqref{wL2f01}, \eqref{wH1f12},  and \eqref{w2x001}, we have
\begin{align}\label{prede}
   & \frac{\ud}{\ud t}\ee+\frac{3}{2}\dd
    \ls \Big[(1+t)^{-\beta_0}+\e^{\frac{1}{3}}\Big]\ee+\e^{\frac{2}{3}}\dd+\e^{2k+1}(1+t)^{4k+2}
    +\e^{k}(1+t)^{2k}\sqrt{\mathcal{E}},
\end{align}
where $\dd$ is given in \eqref{ddd}, and
\begin{align*}
    \ee=&C_1\bigg[\sum_{i=0}^2\e^i\bigg(\nm{\sqrt{\frac{4\pi  T }{u^0}}\nabla_x^i\fe}^2+\nm{\nabla_x^i E_R^{\e}}^2+\nm{\nabla_x^iB_R^{\e}}^2\bigg)-\kappa_1\Big(\e\mathcal{E}^{mac}_1+\e^2\mathcal{E}^{mac}_2\Big)\bigg]\\
   &+Y\Big(\nmw{(\ik-\pk)[\fe]}^2+\e\nmww{\nabla_x(\ik-\pk)[\fe]}^2+\nmwww{\nabla_x^2(\ik-\pk)[\fe]}^2\Big).\no
\end{align*}
Note that
\begin{align*}
 \e\mathcal{E}^{mac}_1+\e^2\mathcal{E}^{mac}_2&\ls\sum_{i=1}^2\e^i\nm{\nabla_x^{i-1}\fe}\nm{\nabla_x^i\fe}\ls\e^{\frac{1}{2}}\Big(\nm{\fe}^2+\e\nm{\nabla_x\fe}^2+\e^2\nm{\nabla_x^2\fe}^2\Big)\no
  \end{align*}
by \eqref{i11}. This verifies \eqref{eed}. Then for sufficiently small constant $\e>0$, we use \eqref{prede} to have
\begin{align*}
    \frac{\ud}{\ud t}\ee+\dd\ls \Big[(1+t)^{-\beta_0}+\e^{\frac{1}{3}}+\e^{2k+1}(1+t)^{4k+2}\Big]\ee+\e^{2k+1}(1+t)^{4k+2}
+\e^{k}(1+t)^{2k}.
\end{align*}
Then for $k\geq 3$, we apply Gronwall's inequality to the above inequality to have 
\begin{align*}
\sup_{s\in[0,\overline{t}]}\Big[\mathcal{E}(s)+\int_0^t\dd(s)\ud s \Big]\lesssim  \mathcal{E}(0)+1
\end{align*}
for $\overline{t}=\e^{-1/3}$. This verifies  \eqref{thm2} and \eqref{rr 01}.\\
\ \\
\textbf{Proof of Positivity:}
First we show that there exists $F^{\e}_{R}(0,x,p)$ such that $F^{\e}(0,x,p)\geq0$. The procedure is motivated by the analysis in \cite[Lemma A.2]{Guo2006}.
We first estimate the microscopic part of the coefficients $(\ik-\pk )\left[ \m ^{-\frac{1}{2}}F_{i}\right]$, $1\leq i\leq 2k-1$. By \eqref{expan2} and the definition of $\li$ in \eqref{LAK}, we have 
\begin{align*}
\li \left[(\ik-\pk )\left[ \m ^{-\frac{1}{2}}F_{1}\right]\right]=&-\m ^{-\frac{1}{2}}\Big[\partial_t\mathbf{M}
+\hat{p}\cdot \nabla_x\mathbf{M}- \Big(E +\hat{p}\times B \Big)\cdot\nabla_p\mathbf{M}\Big]
\end{align*}
Then, we use Lemma \ref{ss 01} to have
\begin{align}\label{micro11}
 \abs{(\ik-\pk )\left[ \m ^{-\frac{1}{2}}F_{1}\right]}_{\sigma}\lesssim \big|\nabla_x(n ,u ,T )\big|+|E |+|B |.
    \end{align}
By similar arguments in the proof of Lemma \ref{ss 04}, we can obtain that
for any $\kappa<1$, 
\begin{align*}
 &\brv{\mathbf{M}^{-\kappa}\li\left[(\ik-\pk)\left[\m ^{-\frac{1}{2}}F_{1}\right]\right],(\ik-\pk)\left[\m ^{-\frac{1}{2}}F_{1}\right]}\\
 \gtrsim&  \abs{\mathbf{M}^{-\frac{\kappa}{2}}(\ik-\pk)\left[\m ^{-\frac{1}{2}}F_{1}\right]}_{\sigma}^2
-C\Big(\abs{(\ik-\pk)\left[\m ^{-\frac{1}{2}}F_{1}\right]}_{\sigma}^2+|E |^2+|B |^2\Big).\no
\end{align*}
Namely, 
\begin{align}\label{w11}
    &\abs{\mathbf{M}^{-\frac{\kappa}{2}}(\ik-\pk)\left[\m ^{-\frac{1}{2}}F_{1}\right]}_{\sigma}^2
    -C\abs{(\ik-\pk)\left[\m ^{-\frac{1}{2}}F_{1}\right]}_{\sigma}^2\\
    \lesssim& o(1)\abs{\mathbf{M}^{-\frac{\kappa}{2}}(\ik-\pk)\left[\m ^{-\frac{1}{2}}F_{1}\right]}_{\sigma}^2+C\Big(\big|\nabla_x(n ,u,T )\big|^2+|E |^2+|B |^2\Big).\no
\end{align}
Now we combine \eqref{micro11} and \eqref{w11} to have
\begin{align*}
    \abs{\mathbf{M}^{-\frac{\kappa}{2}}(\ik-\pk)\left[\m ^{-\frac{1}{2}}F_{1}\right]}_{\sigma} \lesssim \big|\nabla_x(n ,u,T )\big|+|E |+|B |.
\end{align*}

Similarly, we can obtain that
\begin{align*}
   \sum_{0\leq j\leq 2} \abs{\nabla_p^j\left(\mathbf{M}^{-\frac{\kappa}{2}}(\ik-\pk)\left[\m ^{-\frac{1}{2}}F_{1}\right]\right)}_{\sigma} \lesssim \big|\nabla_x(n ,u,T )\big|+|E |+|B |.
\end{align*}
By the Sobolev imbedding, this implies 
\begin{align}
   (\ik-\pk)\left[\m ^{-\frac{1}{2}}F_{1}\right]\lesssim \mathbf{M}^{\frac{\kappa}{2}}\Big[\big|\nabla_x(n ,u,T )\big|+|E |+|B |\Big].
   \end{align}
By induction, we can use the equations \eqref{number}, \eqref{moment} and \eqref{energy} in the appendix to obtain    \begin{align}\label{micfi}
   (\ik-\pk)\left[\m ^{-\frac{1}{2}}F_{i}\right]\lesssim &\mathbf{M}^{\frac{\kappa}{2}}\Big[(\big|\nabla_x^i(n ,u,T )\big|+|\nabla_x^{i-1}E |+|\nabla_x^{i-1}B \big|\Big)\\
   &+\sum_{1\leq j\leq i-1}\Big(\big|\nabla_x^{i-j}(a_j,b_j,c_j)\big|+\big|\nabla_x^{i-j-1}E_j\big|+\big|\nabla_x^{i-j-1}B_j\big|\Big)\Big],\no
   \end{align} 
  for all $\kappa<1$ and $2\leq i\leq 2k-1$. Note that here $a_j,b_j,c_j$, the coefficients of the macroscopic part of $\pk\Big(\m ^{-\frac{1}{2}}F_{j}\Big)$,  are defined in  \eqref{decomA}.
Then we use \eqref{micfi}  to have
\begin{align}\label{unii}
F_{i}(0,x,p)\lesssim& \mathbf{M}^{\kappa}\Big[(\big|\nabla_x^i(n ,u ,T )\big|+\big|\nabla_x^{i-1}E \big|+\big|\nabla_x^{i-1}B \big|\Big)+\big|(a_{i},b_i,c_i)\big|\\
   &+\sum_{1\leq j< i}\Big(\big|\nabla_x^{i-j}(a_{j},b_j,c_j)\big|+\big|\nabla_x^{i-j-1}E_j\big|+\big|\nabla_x^{i-j-1}B_j\big|\Big)\Big].\no
\end{align}
Now we choose $F^{\e}_{R}(0,x,p)$ in the following form
\begin{align}
 F^{\e}_{R}(0,x,p)=&\mathbf{M}^{\tau}(0,x,p)\Big[\sum_{ j=1}^{2k-1}\Big(\big|\nabla_x^j(n ,u ,T )\big|+\big|\nabla_x^{j-1}E_j\big|+\big|\nabla_x^{j-1}B_j\big|\Big)\\
   &+\sum_{ i=1}^{2k-1}\big|(a_{i},b_i,c_i)\big|+\sum_{ i=1}^{2k-1}\sum_{ j=1}^{2k-1-i}\Big(\big|\nabla_x^{j}(a_{i},b_i,c_i)\big|+\big|\nabla_x^{j-1}E_i\big|+\big|\nabla_x^{j-1}B_i\big|\Big)\Big]\no
\end{align}
with $0<\tau<1$. We choose $\kappa<1$ such that
\begin{align}\label{kt}
   k(1-\kappa)+\tau<\kappa. 
\end{align}
From \eqref{unii}, we have 
 \begin{align}
     \sum_{i=1}^{2k-1}\e^i F_{i}(0, x, p)\leq &C\e \mathbf{M}^{\kappa}(0,x,p)\Big[\sum_{ j=1}^{2k-1}(\big|\nabla_x^j(n ,u ,T )(0,x)\big|+|\nabla_x^{j-1}(E_j,B_j)(0,x)\big|\Big)\\
   &+\sum_{ i=1}^{2k-1}\big|(a_{i},b_i,c_i)\big|+\sum_{ i=1}^{2k-1}\sum_{ j=1}^{2k-1-i}\Big(\big|\nabla_x^{j}(a_{i},b_i,c_i)\big|+\big|\nabla_x^{j-1}(E_i,B_i)\big|\Big)\Big]\no\\
     \leq& C_0 \e \mathbf{M}^{\kappa}(0,x,p)\no
\end{align}
for some uniform constant $C_0\geq1$.
We discuss the positivity of $F^{\e}_{R}(0,x,p)$ in two domains in $\r^3_x\times\r^3_p$: 
\begin{align*}
    A&:=\Big\{(x,p): \mathbf{M}(0,x,p)\geq C_0 \e \mathbf{M}^{\kappa}(0,x,p)\Big\},\\ 
    B&:=\Big\{(x,p): \mathbf{M}(0,x,p)< C_0 \e \mathbf{M}^{\kappa}(0,x,p)\Big\}.
\end{align*}
In the domain $A$, by the expression of the Hilbert expansion \eqref{expan}, we have   $F^{\e}_{R}(0,x,p)\geq0$.
In the domain $B$, for the chosen $\kappa$, we use \eqref{kt} to have
\begin{align*}
    \e^{k} \mathbf{M}^{\tau}(0,x,p)> C_0^{k+1}\e^{k+1} \mathbf{M}^{\tau}(0,x,p) \ge C_0 \e  \mathbf{M}^{k(1-\kappa)}(0,x,p)\mathbf{M}^{\tau}(0,x,p)\ge  C_0 \e \mathbf{M}^{\kappa}(0,x,p).
\end{align*}
This implies that the remainder term is the dominant term in 
\eqref{expan} and $F^{\e}(0,x,p)\geq0$ for $\e$ small enough. Therefore we have  $F^{\e}(0,x,p)\geq0$ for all $(x,p) $.

Based on the proof of \cite[Lemma 9, Page 307--308]{Strain.Guo2004}, 
we may rearrange the equation \eqref{main1} as
\begin{align}\label{semp}
    &\dt F^{\e} + \hp \cdot \nx F^{\e}-\Big(E^{\e}+\hp\times b^{\e}\Big)\nabla_pF^{\e}\\
    =&\;\frac{1}{\e}\left(\int_{\r^3}\Phi^{ij}(p,q)F^{\e}(q)\ud q\right)\p_{p_i}\p_{p_j}F^{\e}\no\\
    &\;+\frac{1}{\e}\left(\int_{\r^3}\p_{p_i}\Phi^{ij}(p,q)F^{\e}(q)\ud q\right)\p_{p_j}F^{\e}-\frac{1}{\e}\left(\int_{\r^3}\Phi^{ij}(p,q)\p_{q_j}F^{\e}(q)\ud q\right)\p_{p_i}F^{\e}\no\\
    &\;-\frac{1}{\e}\p_{p_i}\left(\int_{\r^3}\Phi^{ij}(p,q)\p_{q_j}F^{\e}(q)\ud q\right)F^{\e}.\no
\end{align}
Then clearly, there is an elliptic structure on the R.H.S. of \eqref{semp}. 
Therefore, using the maximum principle (see the proof of \cite[Lemma 9, Page 308]{Strain.Guo2004} and \cite[Theorem 1.1, Page 201]{Kim.Guo.Hwang2020}),
we have 
\begin{align*}
    \min_{t,x,p}\big\{F^{\e}\big\}=\min_{x,p}\big\{F^{\e}_0\big\} \geq 0.
\end{align*}
Then for sufficiently smooth $F^{\e}$, as long as the initial data $F ^{\e}\geq0$, we naturally have $F^{\e}\geq0$. For general $F^{\e}$, a standard mollification and approximation argument leads to the desired result.

\section{Relativistic Landau Equation}

\subsection{No-Weight Energy Estimates} \label{Sec:no-weight-energy}

In this section, we derive the no-weight energy estimates. Since  the estimates  can be derived via arguments  similar to the r-VML case, we omit most of the details of the proof and only point out  main differences.

Corresponding to Proposition \ref{L2ener}, we have

\begin{proposition}\label{basic 0}
For the remainder $\fe$, it holds that
\begin{align}\label{estimate 0}
    \frac{\ud}{\ud t}\tnm{\fe}^2
    +\e^{-1}\delta\bnms{(\ik-\pk)[\fe]}^2
    \ls\big(\e+\zz\big)\ee+\e\dd+\e^{2k+3}.
\end{align}
\end{proposition}
\begin{proof} The whole proof can be done in a similar and much simpler way as in Proposition \ref{L2ener}. We only proceed the following two estimates:
\begin{align}
    &\abs{\br{\mhh\Big(\dt\mh+\hp\cdot\nx\mh\Big)\fe,\fe}}\ls \zz\nm{\sqrt{p^0}\fe}^2\\
    \ls&\,\zz\bigg(\bnm{\sqrt{p^0}\pk[\fe]}^2+\int_{\r^3}\int_{p^0\leq\e^{-1}\kappa}p^0\babs{(\ik-\pk)[\fe]}^2+\int_{\r^3}\int_{p^0\geq\e^{-1}\kappa}p^0\babs{(\ik-\pk)[\fe]}^2\bigg)\no\\
    \ls&\,\zz\bigg(\nm{\pk[\fe]}^2+\e^{-1}o(1)\bnms{(\ik-\pk)[\fe]}^2+\int_{\r^3}\int_{p^0\geq\e^{-1}\kappa}p^0\babs{(\ik-\pk)[\fe]}^2\bigg)\no\\
    \ls&\,\zz\tnm{\fe}^2+o(1)\e^{-1}\bnms{(\ik-\pk)[\fe]}^2+\zz\int_{\r^3}\int_{p^0\geq\e^{-1}\kappa}p^0\babs{(\ik-\pk)[\fe]}^2\no\\
    \ls&\,\zz\tnm{\fe}^2+o(1)\e^{-1}\nms{(\ik-\pk)[\fe]}^2+\zz\e\nm{(\ik-\pk)[\fe]}_{w_0}^2.\no
\end{align}
by  Lemma \ref{ss 06}, and
\begin{align*}
    \abs{\br{\sb,\fe}}&=\abs{\br{\sb,(\ik-\pk)[\fe]}}\ls  o(1)\e^{-1}\nms{(\ik-\pk)[\fe]}^2+\e^{2k+3},
\end{align*}
by \eqref{growth0=} in Appendix \ref{App}.
\end{proof}

For derivatives of $\fe$, we can obtain the following estimates.
\begin{proposition}\label{basic 1}
For the remainder $\fe$, it holds that
\begin{align}\label{estimate 1}
    \e\bigg(\frac{\ud}{\ud t}\tnm{\nabla_x\fe}^2
    +\e^{-1}\delta\nms{(\ik-\pk)[\nabla_x\fe]}^2\bigg)\ls \big(\e+\zz\big)\ee+o(1)\e^{-1}\nms{(\ik-\pk)[\fe]}^2+\e\dd+\e^{2k+3},
\end{align}
and
\begin{align}\label{estimate 2}
    \,\e^2\bigg(\frac{\ud}{\ud t}\tnm{\nabla_x^2\fe}^2
    +\e^{-1}\delta\nms{(\ik-\pk)[\nabla_x^2\fe]}^2\bigg)
    \ls\,\big(\e+\zz\big)\ee+o(1)\nm{(\ik-\pk)[\fe]}_{H^1_{\sigma}}^2+\e\dd+\e^{2k+4}.
\end{align}
\end{proposition}

\subsection{Weighted Energy Estimates} \label{Sec:weighted-energy}

In this subsection, we will derive the weighted energy estimates of $f^{\e}$.

\subsubsection{Weighted Basic Energy Estimate}

We take microscopic projection onto (i.e. apply operator $\ik-\pk$ on both sides of) \eqref{re-f=} to have
\begin{align}\label{re-f'}
    &\, \dt\big((\ik-\pk)[\fe]\big)+\hp\cdot\nx\big((\ik-\pk)[\fe]\big)+\frac{1}{\e}\li\big[(\ik-\pk)[\fe]\big]\\
    =&\,\e^{k-1}\Gamma[\fe,\fe] +\sum_{i=1}^{2k-1}\e^{i-1} \Big\{\Gamma\big[\mhh F_i,\fe\big]
    +\Gamma\big[\fe,\mhh F_i\big]\Big\}\no\\
    &\,-\mhh\Big(\dt\mh+\hp\cdot\nx\mh\Big)\big((\ik-\pk)[\fe]\big)+(\ik-\pk)[\sb]+\jump{{\bf P},\tau}[\fe],\no
\end{align}
where 
\begin{align}
    \jump{{\pk},\tau}={\pk}\tau-\tau{\pk}=(\ik-\pk)\tau-\tau(\ik-\pk)
\end{align}
denotes a commutator of operators ${\pk}$ and $\tau$ which is
given by
\begin{align}
\tau&=\partial_t+\hat{p}\cdot\nabla_x+\mhh\Big(\dt\mh+\hp\cdot\nx\mh\Big).
\end{align}

\begin{proposition}\label{basic 0'}
For the remainder $\fe$, it holds that
\begin{align}\label{estimate 0'}
    &\,\frac{\ud}{\ud t}\tnmw{(\ik-\pk)[\fe]}^2+\e^{-1}\nmsw{(\ik-\pk)[\fe]}^2
    +\yy\tnmw{\sqrt{p^0}(\ik-\pk)[\fe]}^2\\
    \ls&\, \e\tnm{\nabla_x\fe}^2+\e^{-1}\nms{(\ik-\pk)[\fe]}^2+ \e\big(\ee+\dd)+\e^{2k+3},\no
\end{align}
\begin{align}\label{estimate 1'}
    &\,\e\bigg(\frac{\ud}{\ud t}\tnmww{\nabla_x(\ik-\pk)[\fe]}^2+\e^{-1}\delta\nmsww{\nabla_x(\ik-\pk)[\fe]}^2+\yy\tnmww{\sqrt{p^0}\nabla_x(\ik-\pk)[\fe]}^2\bigg)\\
    \ls&\,\e^2\nm{\nabla_x^2\fe}^2+\nms{\nabla_x(\ik-\pk)[\fe]}^2+\e\big(\ee+\dd\big)+\e^{2k+4},\no
\end{align}
and
\begin{align}\label{estimate 2'}
    &\e^{3}\bigg(\frac{\ud}{\ud t}\tnmwww{\nabla_x^2\fe}^2+\e^{-1}\delta\nmswww{\nabla_x^2(\ik-\pk)[\fe]}^2+\yy\tnmwww{\sqrt{p^0}\nabla_x^2\fe}^2\bigg)\\
    \ls&\,\e^2\nm{\nabla_x^2\fe}^2+\e\big(\ee+\dd\big)+\e^{2k+6}.\no
\end{align}
\end{proposition}

\subsection{Macroscopic Dissipation} \label{Sec:macro-dissipation}

In this subsection, we study the macroscopic structure of \eqref{re-f=}.

As in \eqref{macfe}, denote
\begin{align*}
\pk[\fe]:=\mh\Big\{\Big[a^{\e}(t,x)-\frac{\rho_{2}(t,x)}{\rho_{1}(t,x)}c^{\e}(t,x)\Big]+p\cdot b^{\e}(t,x)+p^0c^{\e}(t,x)\Big\}\in\mathcal{N}.
\end{align*}
\begin{proposition}\label{md00=}
There are two functionals $\mathcal{E}^{mac}_i$ for $i=1,2$ satisfying
\begin{align}\label{i11=}
   \ee^{mac}_i\ls \tnm{\nabla_x^{i-1}\fe}\tnm{\nabla_x\fe}, 
\end{align}
such that 
\begin{align}\label{semp 5=}
   & -\frac{\ud}{\ud t}\Big(\e\ee^{mac}_1+\e^2\ee^{mac}_2\Big)+\Big(\e\tnm{\nabla_x\pk[\fe]}^2+\e^2\tnm{\nabla_x^2\pk[\fe]}^2\Big)\\
    \ls&\, \e^{-1}\nms{(\ik-\pk)[\fe]}^2 +\nms{\nabla_x(\ik-\pk)[\fe]}^2+\e\big(\ee+\dd)+\e^{2k+3}.\no
\end{align}
\end{proposition} 

\begin{proof} 
As in Proposition \ref{md00}, it can be proved via local conservation laws and the macroscopic equations of $\fe$.
Write \eqref{re-f=} as 
\begin{align}\label{re-f1=}
    &\dt\fe+\hp\cdot\nabla_x\fe 
    +\frac{1}{\e}\li[\fe] =\overline{h}^{\e},
\end{align}
where
\begin{align}
   \overline{h}^{\e}=\e^{k-1}\Gamma[\fe,\fe] +\sum_{i=1}^{2k-1}\e^{i-1} \Big\{\Gamma\big[\mhh F_i,\fe\big]
    +\Gamma\big[\fe,\mhh F_i\big]\Big\}-\mhh\big(\dt\mh+\hp\cdot\nx\mh\big)\fe+\sb.\no
\end{align}

\textbf{Local conservation laws:}
Similar to the derivation of \eqref{lcl0}, \eqref{lcl00} and \eqref{lcl000}, we can obtain
\begin{align}
    &n u^0\partial_ta^{\e}+\frac{n}{\gamma}\nabla_x\cdot b^{\e}
    =\Xi_{1}\big[a^{\e},b^{\e},c^{\e}\big]-\nabla_x\cdot\int_{\mathbb R^3} \hat{p}\mh(\ik-\pk)[\fe]\,\ud p+\int_{\mathbb R^3}  \mh\overline{h}^{\e}\,\ud p,\label{lcl0=}\\
    &n\left(\frac{K_1(\gamma)}{\gamma K_2(\gamma)}+\frac{4}{\gamma^2}\right)\partial_t b^{\e}+\frac{n }{\gamma}\nabla_x a^{\e}
     \label{lcl00=} \\
    &\qquad=\Xi_{2}\big[a^{\e},b^{\e},c^{\e}\big]-\nabla_x\cdot\int_{\mathbb R^3}  \hat{p}p\mh(\ik-\pk)[\fe]\,\ud p +\int_{\mathbb R^3}  p\mh\overline{h}^{\e}\,\ud p,\no\\
    &n\left(-\frac{K_1^2(\gamma)}{K_2^2(\gamma)}-\frac{3}{\gamma}\frac{K_1(\gamma)}{K_2(\gamma)}+1+\frac{3}{\gamma^2}\right)\partial_tc^{\e}+\frac{n }{\gamma ^2}\nabla_x\cdot b^{\e} \label{lcl000=}\\
    & \qquad=\Xi_{3}\big[a^{\e},b^{\e},c^{\e}\big]+\int_{\mathbb R^3}  p^0\mh\overline{h}^{\e}\,\ud p\no\\
    &\qquad\quad-u^0\left(\frac{K_1(\gamma)}{K_2(\gamma)}+\frac{3}{\gamma}\right)\left(-\nabla_x\cdot\int_{\mathbb R^3}  \hat{p}\mh(\ik-\pk)[\fe]\,\ud p+\int_{\mathbb R^3}  \mh\overline{h}^{\e}\,\ud p\right).\no
\end{align}

\textbf{Macroscopic equations:}
Secondly, we turn to the macroscopic equations of $\fe$. Splitting $\fe$ as the macroscopic part $\pk[\fe]$ and the microscopic $(\ik-\pk)[\fe]$ part in \eqref{re-f1=}, we have
\begin{align}\label{semp 6=}
\\
    &\Big(\partial_t\Big(a^{\e}-\frac{\rho_{2}}{\rho_{1}}c^{\e}\Big)+\hp\cdot \partial_tb^{\e}+p^0 \partial_t c^{\e}\Big)\mh+\hat{p}\cdot\Big\{\nabla_x\Big(a^{\e}-\frac{\rho_{2}}{\rho_{1}}c^{\e}\Big)+ \nabla_x b^{\e}\cdot \hp+p^0 \nabla_x c^{\e}\Big\}\mh=\ell^{\e}+h^{\e},\no
  \end{align}
where
\begin{align}
   \ell^{\e}&:= -\big(\partial_t+\hp\cdot\nabla_x\big)\big[(\ik-\pk)[\fe]\big]-\frac{1}{\e}\li[\fe],\\
   h^{\e}:=&- \left\{\Big(a^{\e}-\frac{\rho^{\e}_{2}}{\rho^{\e}_{1}}c^{\e}\Big)+p\cdot b^{\e}+p^0 c^{\e}\right\}\big(\partial_t+\hp\cdot\nabla_x\big)\mh+\overline{h}^{\e}. 
\end{align}
For fixed $t,x$, we compare the coefficients in front of 
\begin{align}
    \left\{\mh, p_i\mh, p^0\mh, \frac{p_i}{p^0}\mh,\frac{p_i}{p^0}\mh,\frac{p_i^2}{p^0}\mh,\frac{p_ip_j}{p^0}\mh\right\},\qquad 1\leq i,j\leq 3
\end{align}
on both sides of \eqref{semp 6=} and get the following macroscopic equations:
\begin{align}\label{macabc=}
    \partial_t a^{\e}-\frac{\rho_{2}}{\rho_{1}}\partial_tc^{\e}&= \ell_a^{\e}+h^{\e}_a+\partial_t\Big(\frac{\rho_{2}}{\rho_{1}}\Big)c^{\e},\\
    \partial_t b^{\e}_i+\partial_i c^{\e}&=\ell_{bi}^{\e}+h^{\e}_{bi},\no\\
    \partial_t c^{\e}&=\ell_{c}^{\e}+h^{\e}_{c},\no\\
    \partial_i a^{\e}-\frac{\rho_{2}}{\rho_{1}}\partial_i c^{\e}&= \ell_{ai}^{\e}+h^{\e}_{ai}+c\partial_i\Big(\frac{\rho_{2}}{\rho_{1}}\Big)c^{\e},\no\\
    c\partial_i b^{\e}_i&= \ell_{ii}^{\e}+h^{\e}_{ii},\no\\
   \partial_i b_j^{\e}+\partial_j b_i^{\e}&= \ell_{ij}^{\e}+h^{\e}_{ij},\qquad i\neq j.\no
\end{align}
Here $\ell_a^{\e}, h_a^{\e}$, $\ell_{bi}^{\e}, h^{\e}_{bi}$,  $\ell_{c}^{\e}, h^{\e}_{c}$,  $\ell_{ai}^{\e}, h^{\e}_{ai}$,  $\ell_{ii}^{\e}, h^{\e}_{ii}$, and  $\ell_{ij}^{\e}, h^{\e}_{ij}$ take the form of
\begin{align*}
    (\ell^{\e}, \zeta)\ \ \text{and}\ \  (h^{\e}, \zeta),
\end{align*}
where $\zeta$ is linear combinations of 
\begin{align}
   \left\{\mh, p_i\mh, p^0\mh, \frac{p_i}{p^0}\mh,\frac{p_i}{p^0}\mh,\frac{p_i^2}{p^0}\mh,\frac{p_ip_j}{p^0}\mh\right\}.
\end{align}
For $m=0,1$, we have the following estimates:
\begin{align}
    &\nm{\nabla_x^mh_a^{\e}}+\nm{\nabla_x^mh_{bi}^{\e}}+\nm{\nabla_x^mh_{c}^{\e}}+\nm{\nabla_x^mh_{ai}^{\e}}+\nm{\nabla_x^mh_{ii}^{\e}}+\nm{\nabla_x^mh_{ij}^{\e}}\label{h01=}\\
    \ls&\, \e^{\frac{1}{2}}\nm{\fe}_{H^m}+\sum_{l=1}^{2k-1}\left(\bnm{\mhh F_l}_{W^{m,\infty}_xL^2_{p}}\Big\||\fe|_{H^m_\sigma}\Big\|+\Big\|\mhh |F_l|_{\sigma}\Big\|_{W^{m,\infty}_x}\nm{\nabla_x^j\fe}\right)+\e^{k+1}\no\\
    \ls& \nm{\nabla_x^j(\ik-\pk)[\fe]}_{H^m_\sigma}+\nm{\fe}_{H^m}+\zz\nm{\fe}_{H^m}+\e^{k+1}.\no
\end{align}
For brevity, we only give the estimate of $\|\nabla_x b^{\e}\|$ in \eqref{semp 5=} since other estimates can be derived similarly.
From the last two equalities in \eqref{macabc}, we have
\begin{align}
  -\Delta  b^{\e}_j-\partial_j\nabla_x\cdot b^{\e}=-\sum_{i=1}^3\partial_i\big(\ell_{ij}^{\e}+h^{\e}_{ij}\big)\big(1+\delta_{ij}\big).
\end{align}
We multiply $b^{\e}_j$ and integrate over $\r^3$ to get
\begin{align}
   &\nm{\nabla_xb^{\e}}^2+ \nm{\nabla_x\cdot b^{\e}}^2=\sum_{i=1}^3\br{\big(\ell_{ij}^{\e}+h^{\e}_{ij}\big)\big(1+\delta_{ij}\big),\partial_i b^{\e}_j}\\
   =&\sum_{i=1}^3\big(1+\delta_{ij}\big)\br{\Big(-\big(\partial_t+\hp\cdot\nabla_x\big)\big[(\ik-\pk)[\fe]\big]-\frac{1}{\e}\li[\fe], \zeta_{ij}\Big),\partial_i b^{\e}_j}+\sum_{i=1}^3\big(1+\delta_{ij}\big)\br{\big(h^{\e}_{ij}, \zeta_{ij}\big),\partial_i b^{\e}_j}.\no
\end{align}
For the first term related to $(\ik-\pk)[\fe]$, we have
\begin{align}
    &\br{\Big(-\big(\partial_t+\hp\cdot\nabla_x\big)\big[(\ik-\pk)[\fe]\big]-\frac{1}{\e}\li[\fe], \zeta_{ij}\Big),\partial_i b^{\e}_j}\\
    =&\br{\Big(-\partial_t\big[(\ik-\pk)[\fe]\big], \zeta_{ij}\Big),\partial_i b^{\e}_j}+\br{\Big(-\hp\cdot\nabla_x\big[(\ik-\pk)[\fe]\big]-\frac{1}{\e}\li[\fe], \zeta_{ij}\Big),\partial_i b^{\e}_j}\no\\
    \leq&-\frac{\ud}{\ud t}\br{\Big(\big[(\ik-\pk)[\fe]\big], \zeta_{ij}\Big),\partial_i b^{\e}_j}+\br{\Big(\big[(\ik-\pk)[\fe]\big], \zeta_{ij}\Big),\partial_i\partial_t b^{\e}_j}\no\\
    &+o(1) \nm{\nabla_xb^{\e}}^2+C\Big(\nms{\nabla_x(\ik-\pk)[\fe]}^2+\e^{-2}\nms{(\ik-\pk )[\fe ]}^2\Big).\no
\end{align} 
By \eqref{lcl00} and \eqref{h01} with terms related to the electromagnetic field be zero, we have
\begin{align}
    \br{\Big(\big[(\ik-\pk)[\fe]\big], \zeta_{ij}\Big),\partial_i\partial_t b^{\e}_j}
    &=-\br{\Big(\partial_i\big[(\ik-\pk)[\fe]\big], \zeta_{ij}\Big),\partial_t b^{\e}_j}\\
    &\ls o(1)\big(\nm{\nabla_xa^{\e}}^2+\nm{\nabla_x c^{\e}}^2\big)+\nm{\nabla_x(\ik-\pk)[\fe]}^2_{H^1}+\nm{\fe}^2+\e^{k+1}.\no
\end{align}
Then we use \eqref{h01=} again to obtain
\begin{align}
  \frac{1}{2} \nm{\nabla_xb^{\e}}^2+ \nm{\nabla_x\cdot b^{\e}}^2\leq&-\frac{\ud}{\ud t}\br{\Big(\big[(\ik-\pk)[\fe]\big], \zeta_{ij}\Big),\partial_i b^{\e}_j}+o(1)\big(\nm{\nabla_xa^{\e}}^2+\nm{\nabla_x c^{\e}}^2\big)\\
   &+\e^{-2}\nms{(\ik-\pk)[\fe]}^2+\nms{\nabla_x(\ik-\pk)[\fe]}^2+\big(1+\zz\big)\nm{\fe}^2+\e^{k+1}.\no
\end{align}
\end{proof}

\bigskip

\subsection{Proof of Propostion~\ref{result 2=}} \label{Sec:main-thm-pf}

\ \\
\textbf{Proof of Energy Estimates:}
Multiplying \eqref{semp 5} by a small constant $\kappa_2$ and adding it to the sum of \eqref{estimate 0}, \eqref{estimate 1} and \eqref{estimate 2}, we obtain that for some small constant $\delta_2>0$,
\begin{align}\label{ud=}
  &\frac{\ud}{\ud t}\bigg(\Big(\nm{\fe}^2+\e\nm{\nabla_x\fe}^2+ \e^2\nm{\nabla_x^2\fe}^2\Big)-\kappa_2\Big(\e\mathcal{E}^{mac}_1+\e^2\mathcal{E}^{mac}_2\Big)\bigg)\\
  &+\delta_2\Big(\e\nm{\nabla_x\pk[\fe]}^2+\e^2\nm{\nabla_x^2\pk[\fe]}\Big)\no\\
  &+\delta_2\Big(\e^{-1}\nms{(\ik-\pk)[\fe]}^2+\nms{\nabla_x(\ik-\pk)[\fe]}^2+\e\nms{\nabla_x^2(\ik-\pk)[\fe]}^2\Big)\no\\
  \ls&\,\big(\e+\zz\big)\ee+\e\dd+\e^{2k+3}.\no
\end{align}
Multiplying \eqref{ud=} by a large constant $C_2$  and adding it to the sum of \eqref{estimate 0'}, \eqref{estimate 1'},  and \eqref{estimate 2'}, we have
\begin{align}
    \frac{\ud}{\ud t}\bigg(\ee-\kappa_2\Big(\e\mathcal{E}^{mac}_1+\e^2\mathcal{E}^{mac}_2\Big)\bigg)+\frac{3}{2}\dd\ls \big(\e+\zz\big)\ee+\e\dd+\e^{2k+3}.
\end{align}
where $\dd$ is given in \eqref{ddd=}, and
\begin{align*}
    \ee=&C_1\bigg[\sum_{i=0}^2\e^i\nm{\sqrt{\frac{4\pi  T }{u^0}}\nabla_x^i\fe}^2-\kappa_2\Big(\e\mathcal{E}^{mac}_1+\e^2\mathcal{E}^{mac}_2\Big)\bigg]\\
   &+Y\Big(\nmw{(\ik-\pk)[\fe]}^2+\e\nmww{\nabla_x(\ik-\pk)[\fe]}^2+\nmwww{\nabla_x^2(\ik-\pk)[\fe]}^2\Big).\no
\end{align*}
Note that
\begin{align*}
 \e\mathcal{E}^{mac}_1+\e^2\mathcal{E}^{mac}_2&\ls\sum_{i=1}^2\e^i\nm{\nabla_x^{i-1}\fe}\nm{\nabla_x^i\fe}\ls\e^{\frac{1}{2}}\Big(\nm{\fe}^2+\e\nm{\nabla_x\fe}^2+\e^2\nm{\nabla_x^2\fe}^2\Big)\no
  \end{align*}
by \eqref{i11=}. This verifies \eqref{eed=}.
When $\e$ is sufficiently small, we know $\e\ls\zz$. Thus, we have
\begin{align}
    \frac{\ud}{\ud t}\ee+ \dd\ls \zz\ee+\e^{2k+3}.
\end{align}
By Gronwall's inequality, for $t\leq t_0$, we have
\begin{align}
    \ee(t)+\int_0^t\dd(s)\ud s \ls\ue^{\zz t}\ee(0)+\e^{2k+3}\int_0^t\ue^{\zz(t-s)}\ud s\ls \ue^{\zz t}\ee(0)+\zz^{-1}\e^{2k+3}.
\end{align}
Due to \eqref{assump=}, we know $\zz t_0\ls1$. Hence, we have
\begin{align}
    \ee(t)+\int_0^t\dd(s)\ud s \ls\ee(0)+\e^{2k+3}, \qquad t\leq t_0.
\end{align}
This verifies the validity of \eqref{thm2=} and \eqref{rr 01=}.\\
\ \\
\textbf{Proof of Positivity:}
It is analogous to the corresponding part in the proof of Proposition~\ref{result}. 
We omit the proof for brevity.


\appendix

\makeatletter
\renewcommand \theequation {%
A.%
\ifnum\c@subsection>\z@\@arabic\c@subsection.%
\fi\@arabic\c@equation} \@addtoreset{equation}{section}
\@addtoreset{equation}{subsection} \makeatother

\section {Expansion of the Relativistic Vlasov-Maxwell-Landau System}\label{App}

In this part, we list our result about the construction and regularity estimates of the coefficients in the Hilbert expansion \eqref{expan}.
For any integer $n\in[1, 2k-1]$, we decompose  $\m ^{-\frac{1}{2}}F_{n}$ as the sum of macroscopic and microscopic parts:
\begin{align}\label{decomA}
\m ^{-\frac{1}{2}}F_{n}&={\bf P}\left[\m ^{-\frac{1}{2}}F_{n}\right]+(\ik-\pk )\left[\m ^{-\frac{1}{2}}F_{n}\right]\\
&=\Big(a_{n}(t,x)+b_{n}(t,x)\cdot p+c_{n}(t,x) p^0\Big)\mh +(\ik-\pk )\left[\m ^{-\frac{1}{2}}F_{n}\right].\no
\end{align}

\begin{proposition}
\label{fn}
For any integer $n\in[0,2k-2]$, assume that $(F_i, E_i, B_i)$ have been constructed for all $0\leq i\leq n$. Then  the  microscopic part $(\ik-\pk )\left[\m ^{-\frac{1}{2}}F_{n+1}\right]$ can be written as:
\begin{align}\label{inverse}
(\ik-\pk )\Big(\m ^{-\frac{1}{2}}F_{n+1}\Big)=&\;\li^{-1}\bigg[-\m ^{-\frac{1}{2}}\Big(\dt F_{n}
+\hat{p}\cdot \nabla_xF_{n}-\frac{1}{\e}\sum_{\substack{i+j=n+1\\i,j\geq1}}\big[\mathcal{C}(F_{i},F_{j})\\
 &\hspace{1cm}+\mathcal{C}(F_{i},F_{j})\big] +\sum_{\substack{i+j=n\\i,j\geq0}} \Big(E_i+\hat{p} \times B_i \Big)\cdot\nabla_pF_{j}\Big)\bigg].\no
\end{align}

And $a_{n+1}(t,x), b_{n+1}(t,x), c_{n+1}(t,x)$, $E_{n+1}(t,x), B_{n+1}(t,x)$ satisfy the following system:
\begin{align}\label{number}
\begin{aligned}
\dt \left(n u^0a_{n+1}+(\en +P)u^0(u \cdot b_{n+1})+(\en (u^0)^2+P|u |^2)c_{n+1}\right)\\
+\nabla_x\cdot\left(n u  a_{n+1}+(\en +P)u  (u \cdot b_{n+1})+P b_{n+1}+(\en +P)u^0 u  c_{n+1}\right)\\
+\nabla_x\cdot\int_{\mathbb R^3}  \hat{p}\mh (\ik-\pk )\left[\m ^{-\frac{1}{2}}F_{n+1}\right]\ud p&=0,
\end{aligned}
\end{align}

\begin{align}
& \dt \bigg\{(\en +P)u^0u_{j}a_{n+1}+\frac{n}{\gamma  K_2(\gamma )}\Big((6K_3(\gamma )+\gamma  K_2(\gamma ))u^0 u_{j}(u \cdot b_{n+1})+K_3(\gamma )u^0 b_{n+1,j}\Big)\no\\
&+\frac{n}{\gamma  K_2(\gamma )}\Big((5K_3(\gamma )+\gamma  K_2(\gamma ))(u^0)^2  +K_3(\gamma )|u |^2\Big)u_{j}c_{n+1}\bigg\}\no\\
&+\nabla_x\cdot\bigg((\en +P)u_{j}u  a_{n+1}+\frac{n}{\gamma  K_2(\gamma )}(6K_3(\gamma )+\gamma  K_2(\gamma ))u_{j} u  \Big((u \cdot b_{n+1})+u^0c_{n+1}\Big)\bigg)\no\\
&+\partial_{x_j}(Pa_{n+1})+\nabla_x\cdot\left(\frac{nK_3(\gamma )}{\gamma  K_2(\gamma )}(u b_{n+1,j}+u_{j}b_{n+1})\right)\label{moment}\\
&+ \partial_{x_j}\left(\frac{nK_3(\gamma )}{\gamma  K_2(\gamma )}\Big(u \cdot b_{n+1}+u^0c_{n+1}\Big)\right)\no\\
&+ E_{0,j}\left(n u^0a_{n+1}+(\en +P)u^0(u \cdot b_{n+1})+(\en (u^0)^2+P|u |^2)c_{n+1}\right)\no\\
&+ \left(\Big(n u  a_{n+1}+(\en +P)u  (u \cdot b_{n+1})+P b_{n+1}+(\en +P)u^0 u  c_{n+1}\Big)\times B \right)_j\no\\
&+\left(n u^0E_{n+1,j}+\Big(n u\times B_{n+1}\Big)_j\right)\no\\
&+ \sum_{\substack{k+l=n+1\\k,l\geq1}}E_{k,j}\left(n u^0a_{l}+(\en +P)u^0(u\cdot b_{l})+(\en (u^0)^2+P|u|^2)c_{l}\right)\no\\
&+  \sum_{\substack{k+l=n+1\\k,l\geq1}}\left(\Big(n u a_{l}+(\en +P)u(u\cdot b_{l})+Pb_{l}+(\en +P)u^0 uc_{l}\Big)\times B_k\right)_j\no\\
&+\nabla_x\cdot\int_{\mathbb R^3}  \frac{p_jp}{p^0}\mh (\ik-\pk )\left[\m ^{-\frac{1}{2}}F_{n+1}\right]\,\ud p+ \left(\int_{\mathbb R^3}  \hat{p}\times B \mh (\ik-\pk )\left[\m ^{-\frac{1}{2}}F_{ n+1}\right]\,\ud p\right)_j\no\\
&+ \sum_{\substack{k+l=n+1\\k,l\geq1}}\left(\int_{\mathbb R^3}  \hat{p}\times B_k\mh (\ik-\pk )\left[\m ^{-\frac{1}{2}}F_{l}\right]\,\ud p\right)_j=0,\no
\end{align}
for $j=1, 2, 3$ with $b_{n+1}=(b_{n+1,1}, b_{n+1,2}, b_{n+1,3})$, $E_{n+1}=(E_{n+1,1}, E_{n+1,2}, E_{n+1,3})$,

\begin{align}
& \dt \bigg\{(\en (u^0)^2+P|u |^2) a_{n+1}+\frac{n(u \cdot b_{n+1})}{\gamma  K_2(\gamma )}\Big((5K_3(\gamma )+\gamma  K_2(\gamma )) (u^0)^2+K_3(\gamma )|u |^2\Big)\no\\
& +\frac{n}{\gamma  K_2(\gamma )}\Big((3K_3(\gamma )+\gamma  K_2(\gamma ))(u^0)^2  +3K_3(\gamma )|u |^2\Big)u^0c_{n+1}\bigg\}\no\\
&+\nabla_x\cdot\left((\en +P)u^0 u  a_{n+1}\right)\no\\
&+\nabla_x\cdot\bigg\{\frac{n}{\gamma  K_2(\gamma )}(6K_3(\gamma )+\gamma  K_2(\gamma ))u^0 u (u \cdot b_{n+1})+\frac{nu^0K_3(\gamma )}{\gamma  K_2(\gamma )}u^0b_{n+1}\label{energy}\\
&+\frac{n}{\gamma  K_2(\gamma )}\Big((5K_3(\gamma )+\gamma  K_2(\gamma ))(u^0)^2+K_3(\gamma )|u |^2\Big)u c_{n+1} \bigg\}\no\\
&+\bigg(n u \cdot E_{n+1}+ n u \cdot E  a_{n+1}+(\en +P)(u \cdot b_{n+1})(u \cdot E )\no\\
&+PE \cdot b_{n+1}+(\en +P)u^0(u \cdot E )c_{n+1}\bigg)\no\\
&+ \int_{\mathbb R^3}  \hat{p}\mh (\ik-\pk )\left[\m ^{-\frac{1}{2}}F_{n+1}\right]\,\ud p \cdot E \no\\
&+ \sum_{\substack{k+l=n+1\\k,l\geq1}}\Big(n u\cdot E_k a_{l}+(\en +P)(u\cdot b_{l})(u\cdot E_k)+PE_k\cdot b_{l}+(\en +P)u^0(u\cdot E_k)c_{l}\no\\
&+\int_{\mathbb R^3}\hat{p}\mh (\ik-\pk )\left[\m ^{-\frac{1}{2}}F_{l}\right]\,\ud p \cdot E_k\Big)+\nabla_x\cdot\int_{\mathbb R^3}p\mh (\ik-\pk )\left[\m ^{-\frac{1}{2}}F_{n+1}\right]\ud p=0,\no
\end{align}

\begin{align}
    &\dt E_{n+1}-\nabla_x \times B_{n+1} \no\\
    &\qquad=
    4\pi\left(n u  a_{n+1}+Pb_{n+1}+(\en +P)u (u\cdot b_{n+1})+(\en +P)u^0u c_{n+1}\right)\no\\
    &\qquad+4\pi\int_{\mathbb R^3} \left( \hat{p}\mh {\{\bf I-P\}}\left[\m ^{-\frac{1}{2}}F_{n+1}\right]\right)\,\ud p, \no\\
    &\dt  B_{n+1}+ \nabla_x \times E_{n+1}=0,\label{EM}\\
    & \nabla_x\cdot E_{n+1}=
    -4\pi\left(n u^0a_{n+1}
    +(\en +P)u^0(u\cdot b_{n+1})+(\en (u^0)^2+P|u|^2)c_{n+1}\right), \no\\
    & \nabla_x\cdot  B_{n+1}=0.\no
\end{align}
Furthermore, assume $a_{n+1}(0,x), b_{n+1}(0,x), c_{n+1}(0,x), E_{n+1}(0,x)$, $B_{n+1}(0,x)\in H^N, N\geq1$,
be given initial data to the system consisted of equations \eqref{number}, \eqref{moment}, \eqref{energy} and \eqref{EM}. Then the
linear system is well-posed in $C^0([0,\infty);H^N)$. Moreover, it holds that
\begin{align} 
&|F_{n+1}|\lesssim (1+t)^{n+1}\m ^{1_-},\qquad |\nabla_pF_{n+1}|\lesssim (1+t)^{n+1}\m ^{1_-},\no\\
&|\nabla_xF_{n+1}|\lesssim (1+t)^{n+1} \m ^{1_-},\qquad |\nabla_x\nabla_pF_{n+1}|\lesssim (1+t)^{n+1}\m ^{1_-},\no\\
&|\nabla_x^2F_{n+1}|\lesssim (1+t)^{n}\m ^{1_-},\qquad |\nabla_x^2\nabla_pF_{n+1}|\lesssim (1+t)^{n+1}\m ^{1_-},\label{growth0}\\
&|E_{n+1}|+|B_{n+1}|+|\nabla_xE_{n+1}|+|\nabla_xB_{n+1}|+|\nabla_x^2E_{n+1}|+|\nabla_x^2B_{n+1}|\lesssim(1+t)^{n+1}.\no 
\end{align}

\end{proposition}

\begin{proof} 
The whole proof follows from analogous arguments as in  \cite[Appendix 3]{Guo.Xiao2021}. We omit the details for brevity.
\end{proof}

\makeatletter
\renewcommand \theequation {%
B.%
\ifnum\c@subsection>\z@\@arabic\c@subsection.%
\fi\@arabic\c@equation} \@addtoreset{equation}{section}
\@addtoreset{equation}{subsection} \makeatother

\section {Expansion of the Relativistic Landau Equation}\label{App=}

In this part, we list our result about the construction and regularity estimates of the coefficients in the Hilbert expansion \eqref{expan}.
The proof can be done in a similar way as that in  \cite[Appendix 3]{Guo.Xiao2021}, so we only record the results.
For any integer $n\in[1, 2k-1]$, we decompose  $\m ^{-\frac{1}{2}}F_{ n}$ as the sum of macroscopic and microscopic parts:
\begin{align}\label{decom=}
\frac{F_n}{\mh }&={\bf P}\left[\m ^{-\frac{1}{2}}F_n\right]+(\ik-\pk)\left[\m ^{-\frac{1}{2}}F_n\right]\\
&=\Big(a_n(t,x)+b_n(t,x)\cdot p+c_n(t,x) p^0\Big)\mh +(\ik-\pk)\left[\m ^{-\frac{1}{2}}F_n\right].\no
\end{align}

\begin{proposition}\label{fn=} For any integer $n\in[0,2k-2]$, assume that $F_i$ have been constructed for all $0\leq i\leq n$. Then  the  microscopic part $(\ik-\pk)\left[\m ^{-\frac{1}{2}}F_{n+1}\right]$ can be written as:
\begin{align}
\begin{aligned}
(\ik-\pk)\Big[\m ^{-\frac{1}{2}}F_{n+1}\Big]=\li^{-1}\Big[-\m ^{-\frac{1}{2}}\Big(\partial_tF_n
+\hat{p}\cdot \nabla_xF_n-\sum_{\substack{i+j=n+1\\i,j\geq1}}\c[F_i,F_j]\Big)\Big].
\end{aligned}
\end{align}
And $a_{n+1}(t,x), b_{n+1}(t,x), c_{n+1}(t,x)$ satisfy  \eqref{moment}, \eqref{moment} and \eqref{energy} by deleting all terms related to the electromagnetic field.
Furthermore, assume $a_{n+1}(0,x), b_{n+1}(0,x), c_{n+1}(0,x)\in H^N$ with $N\geq 1$
are given initial data to the corresponding linear system. Then this
linear system is well-posed in $C^0([0,\infty);H^N)$. Moreover, it holds that for sufficiently large $N$
\begin{align}
&|F_{n+1}|\lesssim \mathbf{M}^{1-},\qquad |\nabla_pF_{n+1}|\lesssim \mathbf{M}^{1-},\nonumber\\
&|\nabla_xF_{n+1}|\lesssim \mathbf{M}^{1-},\qquad |\nabla_x\nabla_pF_{n+1}|\lesssim \mathbf{M}^{1-},\label{growth0=}\\
&|\nabla_x^2F_{n+1}|\lesssim \mathbf{M}^{1-},\qquad |\nabla_x^2\nabla_pF_{n+1}|\lesssim \mathbf{M}^{1-}.\nonumber
\end{align}

\end{proposition}





\bibliographystyle{siam}
\bibliography{Reference}

\end{document}